\pdfoutput=1
\RequirePackage{ifpdf}
\ifpdf 
\documentclass[pdftex]{sigma}
\else
\documentclass{sigma}
\fi

\usepackage[all]{xy}

\newtheorem{thm}{Theorem}[section]
\newtheorem{cor}[thm]{Corollary}
\newtheorem{lem}[thm]{Lemma}
\newtheorem{prop}[thm]{Proposition}

\theoremstyle{definition}
\newtheorem{Def}[thm]{Definition}
\newtheorem{rem}[thm]{Remark}
\newtheorem{cond}[thm]{Condition}
\newtheorem{ex}{Example}[section]
\newtheorem{conj}[thm]{Conjecture}

\newtheorem{quest}[thm]{Question}

\numberwithin{equation}{section}
\numberwithin{figure}{section}

\def\End{{\text{\rm{End}}}}

\def\tr{{\text{\rm{tr}}}}
\def\rchi{{\hbox{\raise1.5pt\hbox{$\chi$}}}}
\def\Aut{{\text{\rm{Aut}}}}

\def\isom{\cong}
\def\tensor{\otimes}
\def\dsum{\oplus}

\def\a{\alpha}
\def\b{\beta}

\def\lam{\lambda}

\def\gam{\gamma}
\def\Gam{\Gamma}
\def\sig{\sigma}
\def\tpi{{\tilde{\pi}}}
\def\tsig{{\tilde{\sigma}}}
\def\TSig{{\widetilde{\Sigma}}}

\def\Sym{{\text{\rm{Sym}}}}

\def\Spec{{\text{\rm{Spec}}}}
\def\Arg{{\text{\rm{Arg}}}}
\def\Proj{{\text{\rm{Proj}}}}
\def\supp{{\text{\rm{supp}}}}
\def\Airy{{\text{\rm{Airy}}}}

\def\Pic{{\text{\rm{Pic}}}}
\def\Gal{{\text{\rm{Gal}}}}
\def\NS{{\text{\rm{NS}}}}

\def\Ext{{\text{\rm{Ext}}}}
\def\Dol{{\text{\rm{Dol}}}}
\def\deR{{\text{\rm{deR}}}}

\def\rank{{\text{\rm{rank}}}}

\newcommand{\bea}{\begin{eqnarray}}
\newcommand{\eea}{\end{eqnarray}}
\newcommand{\be}{\begin{equation}}
\newcommand{\ee}{\end{equation}}

\newcommand{\Mbar}{{\overline{\mathcal{M}}}}
\newcommand{\bA}{{\mathbb{A}}}
\newcommand{\bP}{{\mathbb{P}}}
\newcommand{\bC}{{\mathbb{C}}}

\newcommand{\bF}{{\mathbb{F}}}

\newcommand{\bH}{{\mathbb{H}}}
\newcommand{\bQ}{{\mathbb{Q}}}
\newcommand{\bR}{{\mathbb{R}}}

\newcommand{\bZ}{{\mathbb{Z}}}

\newcommand{\cM}{{\mathcal{M}}}

\newcommand{\cD}{{\mathcal{D}}}
\newcommand{\cE}{{\mathcal{E}}}

\newcommand{\cO}{{\mathcal{O}}}

\newcommand{\cS}{{\mathcal{S}}}

\newcommand{\la}{{\langle}}
\newcommand{\ra}{{\rangle}}
\newcommand{\half}{{\frac{1}{2}}}

\newcommand{\bq}{{\mathbf{q}}}

\newcommand{\rar}{\rightarrow}
\newcommand{\lrar}{\longrightarrow}

\newcommand{\fh}{\bC[\![\hbar]\!]}

\begin{document}
\allowdisplaybreaks

\newcommand{\arXivNumber}{1702.00511}

\renewcommand{\thefootnote}{}

\renewcommand{\PaperNumber}{036}

\FirstPageHeading

\ShortArticleName{Interplay between Opers, Quantum Curves, WKB Analysis, and Higgs Bundles}

\ArticleName{Interplay between Opers, Quantum Curves,\\ WKB Analysis, and Higgs Bundles\footnote{This paper is a~contribution to the Special Issue on Integrability, Geometry, Moduli in honor of Motohico Mulase for his 65th birthday. The~full collection is available at \href{https://www.emis.de/journals/SIGMA/Mulase.html}{https://www.emis.de/journals/SIGMA/Mulase.html}}}

\Author{Olivia DUMITRESCU~$^{\rm ab}$ and Motohico MULASE~$^{\rm cd}$}

\AuthorNameForHeading{O.~Dumitrescu and M.~Mulase}

\Address{$^{\rm a)}$~Department of Mathematics, University of North Carolina at Chapel Hill,\\
\hphantom{$^{\rm a)}$}~340 Phillips Hall, CB 3250, Chapel Hill, NC 27599--3250 USA}
\EmailD{\href{mailto:dolivia@unc.edu}{dolivia@unc.edu}}

\Address{$^{\rm b)}$~Simion Stoilow Institute of Mathematics, Romanian Academy,\\
\hphantom{$^{\rm b)}$}~21 Calea Grivitei Street, 010702 Bucharest, Romania}

\Address{$^{\rm c)}$~Department of Mathematics, University of California, Davis, CA 95616--8633, USA}
\EmailD{\href{mailto:mulase@math.ucdavis.edu}{mulase@math.ucdavis.edu}}

\Address{$^{\rm d)}$~Kavli Institute for Physics and Mathematics of the Universe, The University of Tokyo,\\
\hphantom{$^{\rm d)}$}~Kashiwa, Japan}

\ArticleDates{Received December 31, 2019, in final form March 12, 2021; Published online April 09, 2021}

\Abstract{\textit{Quantum curves} were introduced in the physics literature. We develop a mathe\-matical framework for the case associated with Hitchin spectral curves. In this context, a~quantum curve is a Rees $\mathcal{D}$-module on a smooth projective algebraic curve, whose semi-classical limit produces the Hitchin spectral curve of a Higgs bundle. We give a method of~quantization of~Hitchin spectral curves by concretely constructing one-parameter deformation families of \textit{opers}. We propose a variant of the topological recursion of Eynard--Orantin and Mirzakhani for the context of singular Hitchin spectral curves. We show that a~PDE version of topological recursion provides all-order WKB analysis for the Rees $\mathcal{D}$-modules, defined as the quantization of Hitchin spectral curves associated with meromorphic ${\rm SL}(2,\mathbb{C})$-Higgs bundles. Topological recursion can be considered as a process of quantization of Hitchin spectral curves. We prove that these two quantizations, one via the construction of families of opers, and the other via the PDE recursion of topological type, agree for holomorphic and meromorphic ${\rm SL}(2,\mathbb{C})$-Higgs bundles. Classical differential equations such as the Airy differential equation provides a typical example. Through these classical examples, we see that quantum curves relate Higgs bundles, opers, a conjecture of Gaiotto, and quantum invariants, such as Gromov--Witten invariants.}


\Classification{14H15; 14N35; 81T45; 14F10; 14J26; 33C05; 33C10; 33C15; 34M60; 53D37}


\renewcommand{\thefootnote}{\arabic{footnote}}
\setcounter{footnote}{0}

\section{Introduction}

The purpose of this paper is to construct a
geometric theory of \textit{quantum curves}. The~notion of~quantum curves
was
introduced in
the physics literature (see for example,
\cite{ADKMV, DFM,DHS, DHSV, GHM,GS, Hollands,Mar2, Norbury, Sch, Teschner}). A~quantum curve is supposed to
compactly capture
topological invariants, such as certain
Gromov--Witten invariants, Seiberg--Witten
invariants, and quantum knot polynomials.
Geometrically,
a quantum curve is a unique
quantization of the B-model
geometry, when it is encoded in a
holomorphic curve, that gives a generating function
of A-model theory of genus $g$ for all $g\ge 0$.
In a broad setting, a quantum curve can be
a differential operator, a difference operator,
a mixture of them, or a linear operator defined by a
trace-class kernel function.

The geometric theory we present here is
focused on the process of
\textit{quantization}
of Hitchin spectral curves~\cite{H1,H2}. A~concise overview of our theory is available in~\cite{O-Paris}.
In Definitions~\ref{Def:QC holomorphic}
and~\ref{Def:QC meromorphic}, we introduce
a quantum curve
as a \textit{Rees $\cD$-module}
on a smooth projective algebraic curve $C$
whose \emph{semi-classical limit}
is the Hitchin spectral curve associated with
a~Higgs bundle
on $C$. The~process of quantization is
therefore an assignment
of a Rees $\cD$-module to every Hitchin
spectral curve.

The \textit{Planck constant} $\hbar$ is a
deformation parameter that appears in the
definition of Rees $\cD$-modules. For~us, it has
a geometric meaning, and is naturally identified with
 an element
\be\label{Planck}
\hbar \in H^1(C, K_C),
\ee
where $K_C$ is the canonical sheaf over $C$. The~cohomology group $H^1(C, K_C)$
controls the deformation of
a classical object, i.e., a geometric object
such as a Higgs bundle in our case, into
a~quantum object, i.e.,
a non-commutative quantity such
as a differential operator. In our case, the
result of quantization is
an \textit{oper}.

Using a fixed choice of a
theta characteristic
and a projective structure on $C$, we determine
a unique quantization of the Hitchin
spectral curve of a holomorphic
or meromorphic ${\rm SL}(r,\bC)$-Higgs bundle
through a concrete construction of an $\hbar$-family
of ${\rm SL}(r,\bC)$-opers on $C$,
as proved in
Theorem~\ref{thm:quantization-holomorphic}
for holomorphic case,
and in Theorem~\ref{thm:quantization-meromorphic}
for meromorphic case. The~$\hbar$-family interpolates
opers and Higgs fields.
We then prove, in Theorem~\ref{thm:SCL}, that
the Rees $\cD$-module
 as the quantization result
recovers the starting Hitchin spectral curve
via semi-classical limit of~WKB analysis.
This is our main theorem of the paper.
When we choose the projective structure
of~$C$ of genus $g\ge 2$
coming from the Fuchsian uniformization,
our construction of~opers is the same as
those opers predicted by a conjecture
of Gaiotto~\cite{G},
as explained in~Section~\ref{sub:Gaiotto}.
 This conjecture
has been solved in~\cite{DFKMMN} {(see~\cite{CW} for a subsequent
development.)}

It has been noticed that \textit{topological
recursion} {of Eynard--Orantin~\cite{EO1,EO1+} and others~\cite{CEO,E2004,
E2011}, and also
its
more recent generalizations (such as~\cite{ABO,BCHORS,KS2017} and
many papers cited in these articles),}
provide another aspect of quantization. A~notable one is
 the \emph{remodeling
conjecture} of~Mari\~no
\cite{Marino} and his collaborators~\cite{BKMP,BKMP2}, and its
complete solution by
mathematicians~\cite{FLZ, FLZ2}. {(For many earlier contributions
to the remodeling conjecture and physics oriented discussions, we refer to the references cited in~\cite{EO2,FLZ}.)}
From this point of view,
 a quantum curve is a~quantization of
 B-model geometry that is obtained
 as an application of topological recursion. It~then becomes
 a natural question:

 \begin{quest}
 \label{quest:TR and QC}
 What is the relation
 between quantization
 via topological recursion and the quantization
 through our construction of
 Rees $\cD$-modules from Hitchin
 spectral curves?
 \end{quest}

Topological recursion was originally developed as
a computational mechanism to
calculate the multi-resolvent
correlation functions of random matrices
(see~\cite{CEO, E2004, EO1} and references cited there). As mentioned above,
it generates a mirror symmetric
B-model counterpart of genus $g$ A-model
for all $g\ge 0$.
This correspondence
has been rigorously established for many examples
(see for example,~\cite{BHLM, DoMan, OM3, DMSS, DMNPS, EMS,
 FLZ, FLZ2,MP2012, MSS, MS, MZ, Norbury, Sch}, and others).
 Yet so far still no clear geometric
 relation between topological
 recursion and quantum curves (in particular, when
they appear as difference operators) has been
 established.

Another tantalizing subject is the relation between the
theory of $\tau$-functions and topo\-lo\-gi\-cal recursion/quantum curves. The~present paper does not attempt to address this relation. The~subject presented in Section~\ref{sect:opers}
is closely related to the work of~\cite{BDY,HK,Sch}, in terms of~for\-malism and
mathematical structures. No deep understanding is offered in this paper.

Among the earliest striking applications of
 topological recursion in
algebraic geometry, there are \emph{new} proofs
obtained in
\cite{MZ} for the
Witten conjecture on cotangent class intersection
numbers
and the $\lam_g$-conjecture. Indeed, these celebrated
formulas are straightforward consequences
of the \emph{Laplace transform} of a
combinatorial formula known as the
\emph{cut-and-join equations} of~\cite{GJ,V}.

Applications of topological recursion to enumerative problems
are effective when
the spectral curve
in the theory is of genus $0$. In this case, the residue
calculations required in the formalism of~\cite{EO1}
can be
explicitly performed. The~computational aspect of the formalism
as a tool is not effective
 in a more
general context, such as when the spectral curve
is a high genus
 non-hyperelliptic curve, or has singularities.

A novel approach proposed in~\cite{OM1}
is the implementation of~PDE
recursions of~topological type, which appear
naturally in enumerative geometry
problems, to the context of~Hitchin
spectral curves. It~replaces the
integral topological recursion formulated in terms of~residue
calculations at the ramification
divisor of~a spectral curve
by a recursive set of~partial
differential equations that captures \emph{local}
nature of~topological recursion.
As we explain in Section~\ref{sect:TR},
the main difference of~the two recursion formulas
lies in the choice of~contours
of~integration in the original format of~integral
topological recursion. All other ingredients
are similar. For a~genus~$0$ spectral curve, the two sets of~recursions are equivalent. In general, these two
recursions aim at achieving different goals. The~original choice of~contours should
capture some global nature of~\emph{periods}
hidden in the quantum invariants. Due to the
difficulties of~residue calculations of~higher genus curves, still we do not have a full understanding in this direction. The~PDE recursion of~topological type~\cite{OM1, OM2},
on the other hand,
captures local nature of~the functions involved,
and leads to an all-order WKB analysis of~quantum curves for ${\rm SL}(2,\bC)$-Higgs bundles. The~issue of~singular spectral curves is addressed
in~\cite{OM2}, in which we have developed a
systematic process of~normalization
of~singular Hitchin spectral curves associated
with meromorphic rank~$2$ Higgs bundles.

Theorem~\ref{thm:WKB} is
 our answer to Question~\ref{quest:TR and QC}. It~states
that for the case of ${\rm SL}(2,\bC)$, the nor\-ma\-lization
process of~\cite{OM2} and the PDE
recursion
of~\cite{OM1} produce an all-order
WKB expansion for the meromorphic
Rees $\cD$-modules obtained by quantizing
singular Hitchin spectral curves through
the construction of $\hbar$-families of opers.
In this sense, our result shows that
quantization of Hitchin
spectral curves, singular or non-singular,
through the PDE recursion of topological type
and construction of $\hbar$-family of opers
are equivalent, for the case of ${\rm SL}(2,\bC)$-Higgs bundles.

We note a relation between meromorphic
Higgs bundles over $\bP^1$ and
Painlev\'e equations~\cite{B2001}. An~application of topological recursion to establishing
new results in Painlev\'e theory
and construction of associated quantum curves
are presented in~\cite{IMS,IS}.

The interplay between Rees $\cD$-modules,
$\hbar$-families of opers, Hitchin
spectral curves as semi-classical
limit, Gaiotto's correspondence, and WKB
analysis through PDE recursion of topological type,
creates a sense of inevitability of the notion of
quantization. Section~\ref{sect:Airy}
serves as an overview to this interplay,
where we present the \emph{Airy differential equation} as
a prototypical example.

\looseness=1
A totally new mathematical
framework is presented in~\cite{KS2017}, in which Kontsevich and Soibelman
formulate topological recursion as a special case
of \textit{deformation quantization}.
They call the formalism \textit{Airy structures}.
 In their work, spectral curves no longer
serve as input data for topological recursion.
Although construction of quantum curves is not the
only purpose of the original
 topological recursion, what we
present in our current paper is that our general
procedure
of quantization of Hitchin
spectral curves has nothing to do with individual
spectral curve, in parallel to the philosophy of~\cite{KS2017}. As we show in~\eqref{family SCL},
the family of
spectral curves is (re)constructed from
our deformation family of Rees $\cD$-modules,
not the other way
around. Yet at this moment
we do not have a mechanism
to give the WKB expansion directly for the family of
Rees $\cD$-modules,
without studying individual
spectral curves. Investigating a possible
connection between the Airy structures of~\cite{KS2017} and this paper's results
is a future subject. A~relation between
quantum curves and deformation quantization
was first discussed in~\cite{Petit}.

Let us briefly describe our quantization
process of this paper now. Our geometric setting
is~a~smooth projective algebraic
curve $C$ over $\bC$
of an arbitrary genus $g=g(C)$
with a choice of~a~\emph{spin structure},
or a \emph{theta characteristic},
$K_C^\half$.
There are $2^{2g}$ choices of such
spin structures. We choose any one of them.
Let $(E,\phi)$ be an
${\rm SL}(r,\bC)$-Higgs bundle on $C$
with a meromorphic
Higgs field $\phi$. Denote by
\begin{gather*}
\overline{T^*C}:= \bP(K_C\dsum \cO_C)
\overset{\pi}{\lrar} C
\end{gather*}
the compactified cotangent bundle of~$C$
(see~\cite{BNR,KS2013}),
which is a ruled surface on the base $C$. The~\emph{Hitchin spectral curve}
\begin{gather*}
\xymatrix{
\Sigma \ar[dr]_{\pi}\ar[r]^{i}
&\overline{T^*C}\ar[d]^{\pi}
\\
&C		}
\end{gather*}
 for a meromorphic
Higgs bundle is defined as the divisor
of zeros on $\overline{T^*C}$
of the characteristic polynomial of
$\phi$:
\begin{gather}
\label{char poly}
\Sigma :=\Sigma(\phi)= \left(\det(\eta - \pi^*\phi)\right)_0,
\end{gather}
where $\eta \in H^0(T^*C, \pi^*K_C)$ is the
tautological $1$-form on $T^*C$ extended as
a meromorphic $1$-form on the
compactification $\overline{T^*C}$. The~morphism $\pi\colon \Sigma\lrar C$ is
a degree $r$ map.

We denote by $\cM_{\Dol}$
the moduli space of holomorphic
stable ${\rm SL}(r,\bC)$-Higgs bundles
on $C$ for~$g\ge 2$. The~assignment of
the coefficients of the characteristic
polynomial~\eqref{char poly} to~$(E,\phi)\in \cM_{\Dol}$
defines the Hitchin fibration
\begin{gather}
\label{Hitchin fibration}
\mu_H\colon\ \cM_{\Dol} \lrar B:=\bigoplus_{i=2}^r H^0(C,K_C^{\tensor i}).
\end{gather}
With the choice of
a spin structure $K_C^\half$ and
Kostant's \emph{principal three-dimensional
subgroup} TDS of~\cite{Kostant}, one constructs
a cross-section $\kappa\colon B\lrar \cM_{\Dol}$.
We denote by $\la H,X_+,X_-\ra\subset
\mathfrak{sl}(r,\bC)$ the Lie algebra of a principal TDS,
where we use the standard representation
as traceless matrices acting on $\bC^r$.
 Thus $H$ is diagonal,
$X_-$ is lower triangular, $X_+ =X_-^t$,
and their relations are
\begin{gather}\label{sl2}
[H,X_\pm ] = \pm 2X_\pm, \qquad
[X_+,X_-] = H.
\end{gather}
The map $\kappa$ is
defined by
\begin{gather*}
B\owns \bq =(q_2,\dots,q_r)
\longmapsto \kappa(\bq)\in
\big(E_0,\phi(\bq)\big)\in \cM_{\Dol},
\end{gather*}
where
\begin{gather*}
E_0:= \big(K_C^\half\big)^H,\qquad
\phi(\bq):= X_-+\sum_{\ell=2}^r q_\ell X_+^{\ell-1}.
\end{gather*}
Clearly $\kappa$ is not a \emph{section of
the fibration} $\mu_H$ in a strict sense, because
$\mu_H\circ \kappa$ is \emph{not} the identity
map of $B$ for $r\ge 3$. But it is a section in a
more general sense that the image of
$\kappa$ always intersects with every fiber of
$\mu_H$ exactly at one point.
Note that $B$ is the moduli space of~Hitchin spectral curves associated
with holomorphic ${\rm SL}(r,\bC)$-Higgs bundles
on $C$. We use an~unconventional
way of defining the universal family $\cS$ of
spectral curves over $B$, instead of~the natural family associated with~\eqref{Hitchin fibration},
rather appealing to the
Hitchin section $\kappa$, as
\begin{gather}
\label{family of spectral}
\xymatrix{
\cS \ar[dr]_p \ar[r]
&B\times \overline{T^*C}\ar[d]^{{\rm pr}_1}
\\
&\;B},
\qquad
p^{-1}(\bq) =\Sigma\big(\phi(\bq)\big).
\end{gather}

Now we choose and fix, once and for all, a
\emph{projective coordinate system}
 of~$C$ subordinating
the complex structure of~$C$. This process
does not depend algebraically on the moduli
space of~$C$. For~a curve of genus $g\ge 2$,
the Fuchsian projective structure,
that appears in our solution~\cite{DFKMMN} to a
conjecture of Gaiotto~\cite{G}, is a natural
choice for our purpose.
As we show in Section~\ref{sect:opers},
there is a unique \emph{filtered extension}
$E_\hbar$ for every $\hbar\in H^1(C,K_C)$. For~$r=2$, $E_\hbar$ is the canonical
extension
\begin{gather*}
0\lrar K_C^\half \lrar E_\hbar \lrar K_C^{-\half}
\lrar 0
\end{gather*}
associated with
\begin{gather*}
\hbar\in H^1(C,K_C) =\Ext^1\Big(K_C^{-\half},K_C^\half\Big).
\end{gather*}
With respect to the projective coordinate
system, we can define a one-parameter
family of {opers}
\begin{gather*}
\big(E_\hbar, \nabla^\hbar(\bq)\big)\in \cM_{\deR}
\end{gather*}
for $\hbar\ne 0$,
where
\be\label{nabla-hbar intro}
\nabla^\hbar(\bq) : = {\rm d}+\frac{1}{\hbar} \phi(\bq),
\ee
${\rm d}$ is the exterior differentiation on C,
and $\cM_{\deR}$ is the moduli space
of holomorphic irreducible ${\rm SL}(r,\bC)$-connections
on $C$. The~sum of the exterior differentiation and
a Higgs field is \emph{not} a~con\-nec\-tion in general.
Here, the point is that the original vector
bundle $E_0$ is deformed to~$E_\hbar$,
and we have chosen a projective coordinate system
on $C$.
Therefore,~\eqref{nabla-hbar intro} makes
sense as~a~glo\-bal connection on $C$ in $E_\hbar$.

Note that $\hbar\nabla^\hbar(\bq)$ is
Deligne's $\hbar$-connection interpolating
a connection ${\rm d}+\phi(\bq)$ and a Higgs field
$\phi(\bq)$. We also note that
$\left(E_\hbar, \hbar\nabla^\hbar(\bq)\right)$ defines
a global Rees $\cD$-module on $C$.
Its generator is a globally defined
differential operator $P$ on $C$
that acts on $K_C^{-\frac{r-1}{2}}$, which
is what we call the \emph{quantum curve}
of the Hitchin spectral curve $\Sigma\big(\phi(\bq)\big)$
corresponding to~$\bq\in B$. The~actual shape~\eqref{order r qc} of $P$ is quite involved
due to non-commutativity of the coordinate
of~$C$ and differentiation. It~is determined in the proof
of Theorem~\ref{thm:SCL}. In
Example~\ref{example} we list quantum curves
$P$ for $r=2,3,4$.
 No matter how
complicated its form is,
the semi-classical limit of $P$ recovers the
spectral curve $\sigma^*\Sigma\big(\phi(\bq)\big)$
of the Higgs field $-\phi(\bq)$, where
\begin{gather}
\label{involution}
\sigma\colon\quad\overline{T^*C}\lrar \overline{T^*C}, \qquad \sigma^2 = 1,
\end{gather}
is the involution defined by the fiber-wise action
of $-1$. This extra sign comes from the
difference of conventions in the
characteristic polynomial~\eqref{char poly}
and the connection~\eqref{nabla-hbar intro}.

The above process can be generalized in a
straightforward way to
\emph{meromorphic} spectral data~$\bq$
for a curve $C$ of arbitrary genus. The~corresponding connections
$\nabla^\hbar(\bq)$, and hence
the Rees $\cD$-modules,
 then have regular and irregular
singularities.

We note that when we use the
Fuchsian projective coordinate system of
a curve $C$ of~ge\-nus \mbox{$g\ge 2$} and holomorphic
${\rm SL}(r,\bC)$-Higgs bundles, our quantization
process is exactly the same as the
construction of ${\rm SL}(r,\bC)$-opers
of~\cite{DFKMMN} that was established
by solving a conjecture of~Gaiotto~\cite{G}.

In Section~\ref{sect:WKB}, we perform a PDE variant of
topological recursion for the case of meromorphic
${\rm SL}(2,\bC)$-Higgs bundles. For~this purpose, we use a normalization
method of~\cite{OM2} for singular
Hitchin spectral curves.
We then show that the PDE recursion
provides the WKB analysis for~the~quantum curve constructed through~\eqref{nabla-hbar intro}.
When we deal with a singular spectral curve
$\Sigma\subset \overline{T^*C}$,
the key question
is how to relate the singular curve with smooth ones.
In~terms of~the Hitchin fibration, a~singular spectral
curve corresponds to a degenerate Abelian variety
in the family.
There are two different approaches to
this question: one is to deform~$\Sigma$ locally
to~a~non-singular curve, and the other is
to blow up $\overline{T^*C}$ and obtain
a~resolution of singularities~$\widetilde{\Sigma}$ of~$\Sigma$.
In this paper we will pursue the second path, and
give a WKB analysis of the quantum curve using the
geometric
information of the desingularization.

Kostant's principal TDS plays a crucial role
in our quantization through the relation~\eqref{sl2}. For~example, it selects a particular
fixed point of $\bC^*$-action on the Hitchin
section, which corresponds to
the $\hbar\rar \infty$ limit of~\eqref{nabla-hbar intro}. It~is counterintuitive,
but this limit is the connection ${\rm d}+X_-$
acting on $E_{\hbar = 1}$, not just $d$ which
looks to be the case from the formula.
This limiting connection then
defines a vector space
structure in the moduli space of opers.

 This paper is organized as follows. The~notion of quantum curves as Rees $\cD$-modules
 quantizing Hitchin spectral curves
 is presented in Section~\ref{sect:QC Rees}.
 Then in Section~\ref{sect:opers},
 we quantize Hitchin spectral curves as Rees
 $\cD$-modules through
 a concrete construction of $\hbar$-families
 of holomorphic and meromorphic
 ${\rm SL}(r,\bC)$-opers. The~semi-classical limit
 of these resulting opers is calculated.
 Since our PDE recursion depends
solely on the
geometry of normalization of
singular Hitchin spectral curves, we provide
detailed study of the blow-up process
in Sections~\ref{sect:spectral}.
 We give the genus formula
 for the normalization of the spectral curve
 in terms of the characteristic polynomial of
 the Higgs field $\phi$.
 Then in Section~\ref{sect:TR}, we
 define topological recursions for the case of~degree $2$ coverings.
 In Section~\ref{sect:WKB},
 we prove that an all-order WKB analysis
 for quantization of~meromorphic ${\rm SL}(2,\bC)$-Hitchin
 spectral curves is established
 through PDE recursion of topological type.
 We thus show that two quantizations
 procedures, one through $\hbar$-family
 of opers and the other through PDE
 recursion, agree for
 ${\rm SL}(2,\bC)$. The~general structure of the theory is~explained
 using the Airy differential equation
 as an example in
 Section~\ref{sect:Airy}. This example shows
 how the WKB analysis computes quantum invariants.

The current paper does not address
difference equations that appear as quantum curves
in knot theory, nor the mysterious
spectral theory of~\cite{Mar2}.

\section[Rees D-modules as quantum curves for Higgs bundles]{Rees $\cD$-modules as quantum curves for Higgs bundles}\label{sect:QC Rees}

In this section, we give the definition of quantum
curves
in the context of Hitchin spectral curves.
 Let $C$ be a non-singular projective
algebraic curve defined over $\bC$. The~sheaf $\cD_C$ of differential operators on $C$ is
the subalgebra of
the $\bC$-linear endomorphism algebra
$\cE{\rm nd}_{\bC}(\cO_C)$
generated by the anti-canonical
sheaf $K_C^{-1}$ and the structure sheaf
$\cO_C$. Here, $K_C^{-1}$ acts
on $\cO_C$ as holomorphic vector fields, and
$\cO_C$ acts on itself by multiplication.
Locally every element of $\cD_C$ is written
as
\begin{gather*}
\cD_C\owns P(x) = \sum_{\ell=0}^r a_\ell(x)
\bigg(\frac{\rm d}{{\rm d}x}\bigg)^{r-\ell}, \qquad
a_\ell(x)\in \cO_C
\end{gather*}
for some $r\ge 0$. For~a fixed $r$, we introduce
the filtration by order of differential operators into~$\cD_C$ as follows:
\begin{gather*}
F_r\cD_C = \bigg\{P(x) = \sum_{\ell=0}^r a_\ell(x)
\bigg(\frac{\rm d}{{\rm d}x}\bigg)^{r-\ell}\,\bigg|\, a_\ell(x)\in \cO_C\bigg\}.
\end{gather*}
The \emph{Rees} ring $\widetilde{\cD_C}$ is defined by
\begin{gather*}
\widetilde{\cD_C} = \bigoplus_{r=0}^\infty
\hbar^r F_r\cD_C \subset \fh\tensor_\bC \cD_C.
\end{gather*}
An element of $\widetilde{\cD_C}$ on a coordinate
neighborhood $U\subset C$ can be written
as
\begin{gather}
\label{local Rees P}
P(x,\hbar) = \sum_{\ell=0}^r a_\ell(x, \hbar)
\bigg(\hbar\frac{\rm d}{{\rm d}x}\bigg)^{r-\ell}.
\end{gather}

\begin{Def}[Rees $\cD$-module, cf.~\cite{MS2002}]
The Rees construction
\begin{gather*}
\widetilde{\cM}
= \bigoplus_{r=0}^\infty
\hbar^r F_r \cM
\end{gather*}
 associated with a~filtered
$\cD_C$-module $(F_\bullet, \cM)$
is a \textit{Rees $\cD$-module}
 if the compatibility condition
$F_a \cD_C\cdot F_b \cM\subset F_{a+b}\cM$
holds.
\end{Def}

Let
\begin{gather*}
D = \sum_{j=1}^n m_j p_j, \qquad m_j >0
\end{gather*}
be an effective
divisor on $C$. The~point set $\{p_1,\dots,p_n\}\subset C$ is
the support of $D$. A~\emph{meromorphic
Higgs bundle} with poles at $D$ is
a pair $(E,\phi)$ consisting of an algebraic
 vector bundle
$E$ on $C$ and a Higgs field
\begin{gather*}
\phi\colon\ E\lrar K_C(D)\tensor_{\cO_C}E.
\end{gather*}
Since the cotangent bundle
\begin{gather*}
T^*C = \Spec\big(\Sym\big(K_C^{-1}\big)\big)
\end{gather*} is the total
space of $K_C$, we have the tautological $1$-form
$\eta\in H^0(T^*C,\pi^*K_C)$ on $T^*C$ coming from the projection
\begin{gather*}
\begin{CD}
T^*C @<<< \pi^*K_C
\\
@V{\pi}VV
\\
C@<<< K_C.
\end{CD}
\end{gather*}
The natural holomorphic symplectic form
of $T^*C$ is given by $-{\rm d}\eta$. The~\textit{compactified cotangent bundle}
of~$C$ is a
ruled surface defined by
\begin{gather*}
\overline{T^{*}C}:=\bP(K_C\dsum \cO_C)
=\Proj \bigg(\bigoplus_{n=0}^\infty\big(
K_C^{-n}\cdot I^0\dsum K_C^{-n+1}\cdot
I \dsum \cdots\dsum K_C^0\cdot I^n\big)\bigg),
\end{gather*}
where $I$ represents $1\in \cO_C$ being
considered as a degree $1$ element. The~divisor at infinity
\begin{gather*}
C_\infty := \bP(K_C\dsum \{0\})
\end{gather*}
{\samepage
is reduced in the ruled surface and supported on the
subset $\bP(K_C\dsum \cO_C)
\setminus T^*C$. The~tautological $1$-form $\eta$
extends on $\overline{T^{*}C}$ as a
meromorphic $1$-form with simple poles
along $C_\infty$. Thus the divisor of $\eta$ in
$\overline{T^{*}C}$ is given by
\begin{gather*}
(\eta) = C_0-C_\infty,
\end{gather*}
where $C_0$ is the zero section of $T^*C$.

}

The relation between the sheaf $\cD_C$
and the geometry of the compactified cotangent
bundle $\overline{T^*C}$ is the following.
First we have
\begin{gather*}
\Spec\bigg(\bigoplus_{m=0}^\infty F_m\cD_C\big/F_{m-1}\cD_C\bigg)
=\Spec\bigg(\bigoplus_{m=0}^\infty K_C^{-m}\bigg)=T^*C.
\end{gather*}
Let us denote by ${\rm gr}_m \cD_C =
F_m\cD_C\big/F_{m-1}\cD_C$.
By writing $I = 1\in H^0(C,\cD_C)$, we then have
\begin{gather*}
\overline{T^*C}=
\Proj\bigg(\bigoplus_{m=0}^\infty \big({\rm gr}_m\cD_C\cdot I^0\dsum
{\rm gr}_{m-1}\cD_C \cdot I \dsum
{\rm gr}_{m-2}\cD_C \cdot I^{\tensor 2} \dsum
\cdots \dsum {\rm gr}_0\cD_C \cdot I^{\tensor m}\big)\!\bigg).
\end{gather*}

\begin{Def}[spectral curve] A \textit{spectral curve} of degree $r$
is a divisor $\Sigma$ in $\overline{T^*C}$ such
that the projection $\pi\colon\Sigma \lrar C$ defined by
the restriction
\begin{gather*}
\xymatrix{
\Sigma \ar[dr]_{\pi}\ar[r]^{i}
&\overline{T^*C}\ar[d]^{\pi}
\\
&C}
\end{gather*}
is a finite morphism of degree $r$. The~\textit{spectral curve of a Higgs
bundle} $(E,\phi)$ is the divisor of~zeros
\begin{gather*}
\Sigma = (\det(\eta - \pi^*\phi))_0
\end{gather*}
on $\overline{T^*C}$ of the characteristic
polynomial $\det(\eta - \pi^*\phi)$. Here,
\begin{gather*}
\pi^*\phi\colon\ \pi^* E \lrar \pi^*(K_C(D))\tensor_{\cO_{\bP(K_C\dsum \cO_C)}}\pi^*E.
\end{gather*}
\end{Def}

\begin{rem}
The Higgs field $\phi$ is holomorphic on
$C\setminus \supp (D)$. Thus we can define the divisor of zeros
\begin{gather*}
\Sigma^\circ=\big(\det(\eta-\pi^*(\phi|_{C\setminus \supp (D)}))\big)_0
\end{gather*}
of the characteristic
polynomial
on $T^*(C\setminus \supp (D))$. The~spectral curve
$\Sigma$ is the complex topology closure of
$\Sigma^\circ$
with respect to the compactification
\begin{gather*}
T^*(C\setminus \supp (D))\subset \overline{T^*C}.
\end{gather*}
\end{rem}

A left $\cD_C$-module $\cE$ on $C$ is naturally
an $\cO_C$-module with a $\bC$-linear integrable (i.e., flat) connection
$\nabla\colon\cE \lrar K_C\tensor_{\cO_C} \cE$. The~construction goes as follows:
\begin{gather}
\label{D-module -> flat connection}
\begin{CD}
\nabla\colon\ \cE @>{\a}>> \cD_C\tensor_{\cO_C}\cE
@>{\nabla_{\!\cD}\tensor {\rm id}}>>
 (K_C \tensor_{\cO_C}\cD_C)\tensor_{\cO_C} \cE
 @>{\b\tensor {\rm id}}>> K_C\tensor_{\cO_C} \cE,
\end{CD}
\end{gather}
where
\begin{itemize}\itemsep=0pt

\item $\a$ is the natural inclusion
$\cE \owns v \longmapsto 1\tensor v\in
\cD_C\tensor_{\cO_C}\cE$,

\item $\nabla_{\!\cD}\colon \cD_C \lrar K_C \tensor_{\cO_C}
\cD_C$ is the connection defined by the
$\bC$-linear left-multiplication operation of
$K_C^{-1}$ on $\cD_C$, which satisfies the
derivation property
\begin{gather}
\label{D-connection}
\nabla_{\!\cD}(f\cdot P) = f\cdot\nabla_{\!\cD}(P)+{\rm d}f \cdot P
\in K_C \tensor_{\cO_C}\cD_C
\end{gather}
for $f\in \cO_C$ and $P\in \cD_C$, and

\item $\b$ is the canonical right $\cD_C$-module
structure in $K_C$ defined by the Lie derivative
of vector fields.
\end{itemize}
If we choose a local coordinate neighborhood
$U\subset C$ with a coordinate $x$, then~\eqref{D-connection} takes the
 following form. Let us
denote by $P' = [{\rm d}/{\rm d}x, P]$, and
 define
\begin{gather*}
\nabla_{\!\!\frac{\rm d}{{\rm d}x}}(P) :=
P\cdot \frac{\rm d}{{\rm d}x} + P'.
\end{gather*}
Then we have
\begin{gather*}
\nabla_{\!\!\frac{\rm d}{{\rm d}x}}(f\cdot P) = f\cdot
\nabla_{\!\!\frac{\rm d}{{\rm d}x}}(P) +\frac{{\rm d}f}{{\rm d}x}\cdot P.
\end{gather*}
The connection $\nabla$ of
\eqref{D-module -> flat connection}
is integrable because
$d^2 = 0$. Actually, the statement is true for
any dimensions. We note that
there is no reason for $\cE$ to be
coherent as an $\cO_C$-module.

Conversely, if an algebraic vector bundle
$E$ on $C$ of rank $r$
admits a holomorphic connection $\nabla\colon
E\lrar K_C\tensor E$, then $E$ acquires the
structure of a~$\cD_C$-module. This is because $\nabla$ is
automatically flat, and
the covariant derivative $\nabla_{\!X}$
for $X\in K_{C}^{-1}$ satisfies
\begin{gather}
\label{covariant}
\nabla_{\!X}(f v) =f \nabla_{\!X}( v) + X(f) v
\end{gather}
for $f\in \cO_C$ and $v\in E$. A~repeated application of~\eqref{covariant} makes
$E$ a $\cD_C$-module. The~fact that every $\cD_C$-module on a curve is
principal implies that for every point $p\in C$,
there is an open neighborhood $p\in U\subset C$ and
a linear differential operator $P$ of
order $r$ on $U$, called a generator, such that
$E|_U\isom \cD_U/\cD_U P$.
Thus on an open curve $U$,
 a holomorphic connection in a
vector bundle of rank $r$ gives rise to a differential
operator of order $r$. The~converse
is true if $\cD_U/\cD_U P$ is $\cO_U$-coherent.

\begin{Def}[formal $\hbar$-connection, cf.~\cite{A}]
A formal $\hbar$-connection on a vector bundle
$E\lrar C$ is a~$\fh$-linear homomorphism
\begin{gather*}
\nabla^\hbar \colon\ \fh\tensor E \lrar \fh\tensor K_C \tensor_{\cO_C} E
\end{gather*}
subject to the derivation condition
\begin{gather*}
\nabla^\hbar (f\cdot v) = f \nabla^\hbar(v) + \hbar\, {\rm d}f\tensor v,
\end{gather*}
where $f\in \cO_C\tensor \fh$ and $v\in \fh\tensor E$.
\end{Def}

When we consider \emph{holomorphic}
dependence of a quantum curve with respect
to the quantization parameter $\hbar$, we need to
use a particular $\hbar$-deformation family of
vector bundles. We will discuss the holomorphic
case in Section~\ref{sect:opers},
where we explain how~\eqref{Planck}
appears in our quanti\-zation.

\begin{rem}
The \textit{classical limit} of a
formal $\hbar$-connection
is the evaluation $\hbar = 0$ of $\nabla^\hbar$, which
is simply an $\cO_C$-module homomorphism
\begin{gather*}
\nabla^0\colon\ E\lrar K_C \tensor_{\cO_C} E,
\end{gather*}
i.e., a holomorphic Higgs field in the vector bundle $E$.
\end{rem}

\begin{rem}
An
$\cO_C\tensor\fh$-coherent
 $\widetilde{\cD_C}$-module
is equivalent to a vector bundle on $C$
equipped with an
$\hbar$-connection.
\end{rem}

In analysis, the \textit{semi-classical limit}
of a differential operator $P(x,\hbar)$ of the form~\eqref{local Rees P} is a~function defined by
\begin{gather}
\label{P-SCL}
\lim_{\hbar\rar 0} \big({\rm e}^{-\frac{1}{\hbar} S_0(x)}P(x,\hbar)
{\rm e}^{\frac{1}{\hbar} S_0(x)}\big) =
\sum_{\ell=0}^r a_\ell(x, 0) (S_0'(x))^{r-\ell},
\end{gather}
where $S_0(x)\in \cO_C(U)$. The~equation
\begin{gather}
\label{SCL=0}
\lim_{\hbar\rar 0} \big({\rm e}^{-\frac{1}{\hbar} S_0(x)}P(x,\hbar)
{\rm e}^{\frac{1}{\hbar} S_0(x)}\big) = 0
\end{gather}
then determines the first term of the \textit{singular
perturbation expansion}, or the \textit{WKB asymptotic expansion},
\begin{gather}
\label{SPE}
\psi(x,\hbar) = \exp\bigg(\sum_{m=0} ^\infty \hbar^{m-1} S_m(x)\bigg)
\end{gather}
of a solution $\psi(x,\hbar)$
to the differential equation
\begin{gather*}
P(x,\hbar)\psi(x,\hbar)=0
\end{gather*}
on $U$. We note that the expression~\eqref{SPE} is never meant to be a
convergent series in $\hbar$.

Since ${\rm d}S_0(x)$ is a local section of $T^*C$
on $U\subset C$, $y=S_0'(x)$ gives a
local trivialization of~$T^*C|_U$, with
$y\in T_x^*C$ a fiber coordinate.
 Then~\eqref{P-SCL} and~\eqref{SCL=0} give an equation
\begin{gather*}
\sum_{\ell=0}^r a_\ell(x,0) y^{r-\ell}=0
\end{gather*}
of a curve in $T^*C|_U$.
This motivates us to give the following
definition:

\begin{Def}[semi-classical limit of a Rees differential operator]
Let $U\subset C$ be an open subset of~$C$
with a local coordinate $x$ such that $T^*C$ is
trivial over $U$ with a fiber coordinate $y$. The~semi-classical limit of a local section
\begin{gather*}
P(x,\hbar) = \sum_{\ell =0}^r a_\ell(x,\hbar)
\left(\hbar \frac{\rm d}{{\rm d}x}\right)^{r-\ell}
\end{gather*}
of the
Rees ring $\widetilde{\cD_C}$ of the
sheaf of differential operators $\cD_C$ on $U$
is the holomorphic function
\begin{gather*}
\sum_{\ell =0}^r a_\ell(x,0)y^{r-\ell}
\end{gather*}
defined on $T^*C|_{U}$.
\end{Def}

\begin{Def}[semi-classical limit of a Rees $\cD$-module]\label{def:SCL}
Suppose a Rees $\widetilde{\cD_C}$-module $\widetilde{\cM}$
globally defined on $C$ is written as
\begin{gather*}
\widetilde{\cM}(U) = \widetilde{\cD_C}(U)\big/\widetilde{\cD_C}(U)P_U
\end{gather*}
on every coordinate neighborhood $U\subset C$
with a differential operator
$P_U$ of the form~\eqref{local Rees P}.
Using this expression~\eqref{local Rees P}
for $P_U$, we construct a meromorphic function
\begin{gather}
\label{local Sigma}
p_U(x,y) = \sum_{\ell=0}^r a_\ell(x,0) y^{r-\ell}
\end{gather}
on $\overline{T^*C}|_U$, where $y$ is the fiber
coordinate of $T^*C$, which is trivialized on $U$.
Define
\begin{gather*}
\Sigma_U = (p_U(x,y))_0
\end{gather*}
as the divisor of zero of the function $p_U(x,y)$.
If $\Sigma_U$'s glue together to a spectral
curve $\Sigma\subset \overline{T^*C}$, then
we call $\Sigma$ the \textit{semi-classical limit}
of the Rees $\widetilde{\cD_C}$-module
$\widetilde{\cM}$.
\end{Def}

\begin{rem}
For the local equation~\eqref{local Sigma}
to be consistent globally on $C$, the
 coefficients of~\eqref{local Rees P}
 have to satisfy
\begin{gather*}
a_\ell(x,0)\in \Gamma\big(U,K_C^{\tensor \ell}\big).
\end{gather*}
\end{rem}

\begin{Def}[quantum curve for holomorphic Higgs bundle]
\label{Def:QC holomorphic}
A \textit{quantum curve} associated with
the spectral curve $\Sigma\subset T^*C$
of a holomorphic Higgs bundle on
a projective algebraic curve~$C$ is a Rees
$\widetilde{\cD_C}$-module $\cE$
 whose semi-classical limit is $\Sigma$.
\end{Def}

The main reason we wish to extend our framework
to meromorphic connections is that
there are no non-trivial holomorphic connections
on $\bP^1$, whereas many important
classical examples of differential equations
are naturally defined over $\bP^1$
with regular and irregular singularities. A~$\bC$-linear homomorphism
\begin{gather*}
\nabla\colon\ E\lrar K_C(D) \tensor_{\cO_C} E
\end{gather*}
is said to be a \emph{meromorphic connection}
with poles along
an effective divisor $D$ if
\begin{gather*}
\nabla(f\cdot v) = f\nabla(v) +{\rm d}f\tensor v
\end{gather*}
for every $f\in \cO_C$ and $v\in E$. Let us denote by
\begin{gather*}
\cO_C(*D) :=\lim_{\lrar}\cO_C(mD),\qquad
E(*D) := E\tensor_{\cO_C} \cO_C(*D).
\end{gather*}
Then $\nabla$ extends to
\begin{gather*}
\nabla\colon \ E(*D)\lrar K_C(*D)
\tensor_{\cO_C(*D)} E(*D).
\end{gather*}
Since $\nabla$ is holomorphic on
$C\setminus \supp (D)$, it induces
a $\cD_{C\setminus \supp (D)}$-module structure
in $E|_{C\setminus \supp (D)}$. The~$\cD_C$-module direct image
$\widetilde{E}=j_*\big(E|_{C\setminus \supp (D)}\big)$
 associated with the open
inclusion map $j\colon C\setminus \supp (D)\lrar C$
is then naturally isomorphic to
\begin{gather}
\label{meromorphic extension}
\widetilde{E}=j_*\big(E|_{C\setminus \supp (D)}\big) \isom E(*D)
\end{gather}
as a $\cD_C$-module.
Equation~\eqref{meromorphic extension} is called
the \emph{meromorphic extension} of
the $\cD_{C\setminus \supp (D)}$-module
$E|_{C\setminus \supp (D)}$.

Let us take a
local coordinate $x$ of~$C$, this time around
 a pole $p_j\in \supp(D)$.
 If a generator~$\widetilde{P}$
 of $\widetilde{E}$ near $x=0$ has a local expression
\begin{gather*}
\widetilde{P}\bigg(x, \frac{\rm d}{{\rm d}x}\bigg) =
x^k \sum_{\ell=0}^r b_\ell(x) \bigg(x \frac{\rm d}{{\rm d}x}\bigg)^{r-\ell}
\end{gather*}
around $p_j$ with locally defined holomorphic functions
$b_\ell(x)$, $b_0(0)\ne 0$, and an integer $k\in \bZ$,
then $\widetilde{P}$ has a \emph{regular} singular point
at $p_j$. Otherwise, $p_j$ is an \emph{irregular}
singular point of $\widetilde{P}$.

\begin{Def}[quantum curve for a meromorphic Higgs bundle]\label{Def:QC meromorphic}
Let $(E,\phi)$ be a meromorphic
Higgs bundle defined over a projective algebraic
curve $C$ of any genus with poles
along an~effective divisor $D$, and
$\Sigma\subset \overline{T^*C}$
 its spectral curve. A~\textit{quantum curve}
associated with $\Sigma$ is the meromorphic
extension of a Rees $\widetilde{\cD_C}$-module
$\cE$ on $C\setminus \supp (D)$ such that the
complex topology closure of its
semi-classical limit $\Sigma^\circ \subset
T^*C|_{C\setminus \supp (D)}$
in the compactified cotangent bundle
$\overline{T^*C}$ agrees with $\Sigma$.
\end{Def}

In Section~\ref{sect:opers}, we prove
that every Hitchin spectral curve
associated with a holomorphic
or~a~meromorphic
${\rm SL}(r,\bC)$-Higgs bundle has a quantum curve.

\begin{rem}
We remark that several examples of quantum
curves that are constructed in~\cite{BHLM, DMNPS,MSS}, for various Hurwitz
numbers
and Gromov--Witten
theory of $\bP^1$,
 do not fall into our definition
 in terms of Rees
 $\cD$-modules. This is because in the
 above mentioned examples, quantum curves
 involve
 \emph{infinite}-order differential operators,
 or \emph{difference} operators, while
 we consider only differential operators of
 finite order in this paper.
\end{rem}

\section{Opers}\label{sect:opers}

There is a simple mechanism to construct
a quantization of a Hitchin spectral curve,
using a~particular choice of isomorphism between a
\emph{Hitchin section} and the moduli
of \emph{opers}. The~quantum deformation parameter
$\hbar$, originated in physics as the
Planck \emph{constant}, is a purely
formal para\-me\-ter in WKB analysis.
Since we will be using the PDE recursion
\eqref{differential TR} for the analysis
of~quantum curves, $\hbar$ plays the role of
a formal parameter for the asymptotic
expansion. This point of view
motivates our
definition of quantum curves as Rees $D$-modules
in the previous section.
However, the quantum curves appearing in
the quantization of Hitchin spectral curves
associated with $G$-Higgs bundles for
a complex simple Lie group $G$
always depend \emph{holomorphically}
on $\hbar$. Therefore, we need a
more geometric setup for quantum curves
to deal with this holomorphic dependence. The~purpose of this section is to explain
\textit{holomorphic $\hbar$-connections
as quantum curves}, and the geometric
interpretation of $\hbar$ given
in~\eqref{Planck}. The~key concept is
\emph{opers} of Beilinson--Drinfeld
\cite{BD}.
Although a vast generalization of the current
paper is possible,
we restrict our attention to~${\rm SL}(r,\bC)$-opers
 for an arbitrary $r\ge 2$
in this paper.

In this section, most of the time
 $C$ is a smooth projective
algebraic curve of genus $g\ge 2$ defined over
$\bC$, unless otherwise specified.


\subsection[Holomorphic SL(r,C)-opers and quantization of Higgs bundles]
{Holomorphic $\boldsymbol{{\rm SL}(r,\bC)}$-opers and quantization of Higgs bundles}

We first recall projective structures on $C$ following
Gunning~\cite{Gun}. Recall that every compact
Riemann surface has a projective structure
subordinating the given complex structure. A~\emph{complex projective coordinate system}
is a coordinate neighborhood covering
\begin{gather*}
C = \bigcup_\a U_\a
\end{gather*}
with a local coordinate $x_\a$ of $U_\a$
such that for every $U_\a \cap U_\b$,
we have a \emph{M\"obius} coordinate transformation
\begin{gather*}
x_\a = \frac{a_{\a\b}x_\b + b_{\a\b}}
{c_{\a\b}x_\b + {\rm d}_{\a\b}},
\qquad
\begin{bmatrix}
a_{\a\b}&b_{\a\b}\\
c_{\a\b}&{\rm d}_{\a\b}
\end{bmatrix}
\in {\rm SL}(2,\bC).
\end{gather*}
Since we solve differential equations on
$C$, we always assume that each
coordinate neighborhood~$U_\a$ is \textit{simply connected}.
In what follows, we choose and fix a projective
coordinate system on~$C$. Since
\begin{gather*}
{\rm d}x_\a = \frac{1}{(c_{\a\b}x_\b + {\rm d}_{\a\b})^2}
\; {\rm d}x_\b,
\end{gather*}
the transition function for the canonical line bundle
$K_C$ of~$C$ is given by
the cocycle
\begin{gather*}
\big\{(c_{\a\b}x_\b + {\rm d}_{\a\b})^2\big\}\qquad
 \text{on} \quad U_\a\cap U_\b.
\end{gather*}
We choose and fix, once and for all,
 a theta characteristic,
or a spin structure, $K_C^\half$ such that
$\Big(K_C^\half\Big)^{\tensor 2} \isom K_C$.
Let $\{\xi_{\a\b}\}$ denote the $1$-cocycle
corresponding to~$K_C^\half$.
Then we have
\begin{gather}
\label{xiab}
\xi_{\a\b} = \pm (c_{\a\b}x_\b + {\rm d}_{\a\b}).
\end{gather}
The choice of $\pm$ here is an element of
$H^1(C,\bZ/2\bZ) = (\bZ/2\bZ)^{2g}$,
indicating that there are $2^{2g}$ choices for
spin structures in $C$.

The significance of the projective coordinate system
lies in the fact that $\partial^2_\b \xi_{\a\b} = 0$.
This simple property plays an essential role
 in our construction of global
connections on $C$, as we see in this section.
Another way of appreciating the projective coordinate
system is the vanishing of Schwarzian derivatives,
as explained in~\cite{O-Paris}. A~scalar valued
single linear ordinary differential equation
of any order can be globally defined in
terms of a projective coordinate.

A holomorphic Higgs bundle $(E,\phi)$ is \emph{stable}
if for every vector subbundle
$F\subset E$ that is invariant with respect to~$\phi$, i.e., $\phi\colon F\lrar F\tensor K_C$,
the slope condition
\begin{gather*}
\frac{\deg F}{\rank\, F} < \frac{\deg E}{\rank\, E}
\end{gather*}
holds. The~moduli space of
stable Higgs bundles
is constructed~\cite{Simpson}. An~${\rm SL}(r,\bC)$-Higgs bundle is a pair
$(E,\phi)$
with a fixed isomorphism $\det E = \cO_C$ and
$\tr \phi = 0$.
We denote by $\cM_{\Dol}$
the moduli space of
stable holomorphic ${\rm SL}(r,\bC)$-Higgs bundles on $C$.
Hitchin~\cite{H1} defines a holomorphic
fibration
\begin{gather*}
\mu_H\colon\quad \cM_{\Dol} \owns (E,\phi)\longmapsto \det(\eta - \pi^*\phi)\in B, \qquad
B := \bigoplus_{\ell =2}^r H^0\big(C,K_C^{\tensor \ell}\big),
\end{gather*}
that induces the structure of
an algebraically completely integrable
Hamiltonian system in $\cM_{\Dol}$.
With the choice of a spin structure $K_C^\half$,
we have a natural section
$\kappa\colon B\hookrightarrow \cM_{\Dol}$
defined by utilizing Kostant's
\emph{principal three-dimensional subgroup}
(TDS)~\cite{Kostant} as follows.

First, let
\begin{gather*}
\mathbf{q}=(q_2,q_3,\dots,q_r)
\in B=\bigoplus_{\ell = 2}^r H^0\big(C,K_C^{\tensor \ell}\big)
\end{gather*}
be an arbitrary point of the Hitchin base $B$.
Define
\begin{gather*}
\begin{aligned}
&X_- := \big[ \sqrt{s_{i-1}} \delta_{i-1,j}\big]_{ij}
=\begin{bmatrix}
0&0&\cdots&0&0\\
\sqrt{s_1}&&&&0\\
0&\sqrt{s_2}&&&0\\
\vdots&&\ddots&&\vdots\\
0&0&\cdots&\sqrt{s_{r-1}}&0
\end{bmatrix},
\qquad X_+ := X_-^t,
\\
& H :=[X_+,X_-],
\end{aligned}
\end{gather*}
where $s_i := i(r-i)$.
$H$ is a diagonal matrix whose
$(i,i)$-entry is $s_i-s_{i-1}
= r - 2i+1$, with $s_0=s_r=0$. The~Lie algebra $\la X_+,X_-,H\ra\isom
\mathfrak{sl}(2,\bC)$ is the Lie algebra of
the principal TDS in~${\rm SL}(r,\bC)$.

\begin{lem}
Define a Higgs bundle
$\left(E_0,\phi(\mathbf{q})\right)$
consisting of a vector bundle
\begin{gather}
\label{E0}
E_0 := \Big(K_C^\half\Big)^{\tensor H}
=\bigoplus_{i=1}^r \Big(K_C^\half\Big)^{\tensor (r-2i+1)}
\end{gather}
and a Higgs field
\begin{gather}
\label{phi q}
\phi(\mathbf{q}): =
X_- + \sum_{\ell = 2}^r q_\ell X_+ ^{\ell-1}.
\end{gather}
Then it is
a stable ${\rm SL}(r,\bC)$-Higgs bundle. The~Hitchin section is defined by
\begin{gather}
\label{Hitchin section}
\kappa\colon\ B\owns \bq
\longmapsto (E_0,\phi(\mathbf{q}))\in \cM_{\Dol},
\end{gather}
which gives a biholomorphic map between
$B$ and $\kappa(B)\subset \cM_{\Dol}$.
\end{lem}

\begin{proof}
We first note that $X_-\colon E_0\lrar E_0\tensor K_C$
is a globally defined $\End_0(E_0)$-valued
$1$-form, since it is a collection of constant maps
\begin{gather}\label{si}
\sqrt{s_i}\colon\ \Big(K_C^\half\Big)^{\tensor (r-2i+1)} \overset{=}{\lrar}
 \Big(K_C^\half\Big)^{\tensor (r-2(i+1)+1)} \tensor K_C.
\end{gather}
Similarly, since $X_+^{\ell-1}$ is an upper-diagonal
matrix with non-zero entries along the
$(\ell-1)$-th upper diagonal, we have
\begin{gather*}
q_\ell\colon\ \Big(K_C^\half\Big)^{\tensor (r-2i+1)}{\lrar}
\Big(K_C^\half\Big)^{\tensor (r-2i+1+2\ell)}=
\Big(K_C^\half\Big)^{\tensor (r-2(i-\ell+1)+1)}\tensor K_C.
\end{gather*}
Thus $\phi(\bq)\colon E_0\lrar E_0\tensor K_C$
is globally defined as a Higgs field in $E_0$. The~Higgs pair is stable because no subbundle of
$E_0$ is invariant under $\phi(\bq)$, unless
$\bq = 0$. And when $\bq=0$, the invariant subbundles
all have positive degrees, since $g\ge 2$.
\end{proof}

The image $\kappa(B)$ is a holomorphic
 Lagrangian submanifold of a holomorphic
 symplectic space~$\cM_{\Dol}$.

To define $\hbar$-connections holomorphically
depending on $\hbar$, we need to construct
a one-pa\-ra\-me\-ter holomorphic
family of deformations of vector bundles
\begin{gather*}
\begin{CD}
E_\hbar @>>> \cE\\
@VVV @VVV\\
C\times \{\hbar\} @>>> C\times H^1(C,K_C)
\end{CD}
\end{gather*}
and a $\bC$-linear first-order differential operator
\begin{gather*}
\hbar \nabla^\hbar\colon\ E_\hbar \lrar E_\hbar \tensor K_C
\end{gather*}
depending holomorphically on~$\hbar\in H^1(C,K_C)\isom \bC$ for $\hbar \ne 0$.
Let us introduce the notion of~\emph{filtered extensions}.

\begin{Def}[filtered extension]\label{def:filtered extension}
A one-parameter
family of filtered holomorphic vector bundles
 $\big(F_\hbar ^\bullet, E_\hbar\big)$ on $C$
with a trivialized determinant
$\det (E_\hbar)\isom \cO_C$ is
a filtered extension of~the vector bundle
$E_0$ of~\eqref{E0}
parametrized by $\hbar\in H^1(C,K_C)$
if the following conditions hold:
\begin{itemize}\itemsep=0pt
\item $E_\hbar$ has a filtration
\begin{gather*}
0=F_\hbar^r
\subset
 F_\hbar^{r-1}
\subset
 F_\hbar^{r-2}
\subset
\cdots
\subset
 F_\hbar^0 = E_\hbar.
\end{gather*}

\item The second term is given by
\be
\label{Fr-1}
F_\hbar^{r-1}=\Big(K_C^\half\Big)^{\tensor (r-1)}.
\ee

\item For every $i = 1, 2, \dots, r-1$,
there is
an $\cO_C$-module isomorphism
\be
\label{filt-iso}
F_\hbar^{i}\big/F_\hbar^{i+1}
\overset{\sim}{\lrar} \left(F_\hbar^{i-1}\big/F_\hbar^i\right)
\tensor K_C.
\ee
\end{itemize}
\end{Def}

\begin{rem}
Since we need to identify a deformation
 parameter $\hbar$ and extension classes,
 we~make the natural identification
 \begin{gather*}
 \Ext^1(E,F) = H^1(C,E^*\tensor F)
 \end{gather*}
 for every pair of vector bundles
 $E$ and $F$. We also identify
 $\Ext^1(E, F)$ as the class of
 extensions
 \begin{gather*}
 0\lrar F\lrar V\lrar E\lrar 0
 \end{gather*}
 of $E$ by a vector bundle $V$.
These identifications are done
by a choice of a projective coordinate system
on $C$ as explained below.
\end{rem}

\begin{prop}[construction of filtered extension]\label{prop:filter}
For every choice of the theta charac\-te\-ris\-tic
$K_C^\half$
and a non-zero element
$\hbar\in H^1(C,K_C)$,
there is a unique
non-trivial
filtered extension~$\big(F_\hbar ^\bullet,E_\hbar\big)$
of $E_0$.
\end{prop}

\begin{proof}
First let us examine the case of $r=2$
to see how things work. Since
\begin{gather*}
\hbar \in H^1(C,K_C) =
\Ext^1\Big(K_C^{-\half},K_C^\half\Big)
\isom \bC,
\end{gather*}
we have a unique extension
\begin{gather}
\label{r=2 extension}
0\lrar K_C^\half \lrar E_\hbar \lrar K_C^{-\half} \lrar 0
\end{gather}
corresponding to~$\hbar$. Obviously
\begin{gather*}
K_C^\half \overset{\sim}{\lrar}\Big(E_\hbar \big/K_C^{\half}\Big)\tensor K_C,
\end{gather*}
which proves~\eqref{filt-iso}.
We also note that as a vector bundle of rank $2$,
we have the isomorphism
\begin{gather*}
E_\hbar \isom
\begin{cases}
E_1, \qquad \hbar \ne 0,
\\
E_0, \qquad \hbar = 0.
\end{cases}
\end{gather*}

Now consider the general case. We use the induction
on $i = r-1, r-2, r-3, \dots, 0$, in the reverse direction
to construct each term of the filtration $F_\hbar^i$
subject to~\eqref{Fr-1} and~\eqref{filt-iso}. The~base case $i=r-1$ is the following. Since
\begin{gather*}
0\lrar F_\hbar^{r-1}\lrar F_\hbar^{r-2}\lrar F_\hbar^{r-2}/F_\hbar^{r-1}\lrar 0
\end{gather*}
and
\begin{gather*}
F_\hbar^{r-2}/F_\hbar^{r-1} \isom
F_\hbar^{r-1}\tensor K_C^{-1}\isom
\Big(K_C^\half\Big)^{\tensor (r-1)}\tensor K_C^{-1}
\end{gather*}
from~\eqref{Fr-1} and~\eqref{filt-iso},
$F_\hbar^{r-2}$ is determined by the class $\hbar$ in
\begin{gather*}
\Ext^1\big(F_\hbar^{r-2}/F_\hbar^{r-1},F_\hbar^{r-1}\big)
=H^1\big(C,F_\hbar^{r-1}\tensor \big(F_\hbar^{r-2}/F_\hbar^{r-1}\big)^{-1}\big) \isom H^1(C,K_C).
\end{gather*}
Assume that for a given $i+1$, we have
\begin{gather}
\label{hypo1}
H^1\big(C,F_\hbar^n\tensor \big(F_\hbar^{r-1}\big)^{-1}\tensor
K_C^{\tensor (-n+r)}\big) \isom H^1(C,K_C),
\\
\label{hypo2}
H^1\big(C,F_\hbar^n\tensor \big(F_\hbar^{r-1}\big)^{-1}\tensor
K_C^{\tensor (-n+m+r)}\big)= 0, \qquad m\ge 1,
\end{gather}
for every $n$ in the range $i+1\le n\le r-1$.
We wish to prove that the same relation
holds for~$n=i$.

The sequence of isomorphisms~\eqref{filt-iso} implies that
\begin{gather*}
F_\hbar^{i-1}/F_\hbar^i \isom \big(F_\hbar^{i}/F_\hbar^{i+1}\big)
\tensor K_C^{-1} \isom F_\hbar^{r-1}\tensor K_C^{\tensor(i-r)}
= \Big(K_C^\half\Big)^{\tensor (r-1)}\tensor K_C^{\tensor (i-r)}
\isom K_C^{\tensor\big( i - \frac{r+1}{2}\big)}.
\end{gather*}
Then $F_\hbar^{i-1}$ as an extension
\begin{gather}
\label{ext i-1}
0\lrar F_\hbar^{i} \lrar F_\hbar^{i-1}\lrar
F_\hbar^{r-1}\tensor K_C^{\tensor(i-r)}\lrar 0
\end{gather}
is determined by a class in
\begin{gather*}
\Ext^1\big(F_\hbar^{r-1}\tensor K_C^{\tensor(i-r)},F_\hbar^i\big)
=H^1\big(C,F_\hbar^{i}\tensor\big(F_\hbar^{r-1}\big)^{-1}\tensor K_C^{\tensor(-i+r)}\big).
\end{gather*}
The exact sequence
\begin{gather*}
0\lrar F_\hbar^{n} \lrar F_\hbar^{n-1}\lrar
F_\hbar^{r-1}\tensor K_C^{\tensor(n-r)}\lrar 0
\end{gather*}
implies that
\begin{gather*}
0\lrar F_\hbar^{n} \tensor
\big(F_\hbar^{r-1}\big)^{-1}\!\tensor K_C^{\tensor (-n+m+r)}\!\lrar F_\hbar^{n-1}\!\tensor
\big(F_\hbar^{r-1}\big)^{-1}\!\tensor K_C^{\tensor (-n+m+r)}\!\lrar K_C^{\tensor m}\lrar 0
\end{gather*}
for every $m\ge 1$.
Taking the cohomology long exact sequence, we obtain
\begin{gather*}
H^1\big(C,F_\hbar^{n-1}\tensor \big(F_\hbar^{r-1}\big)^{-1}\tensor
K_C^{\tensor (-n+1+r)}\big) \isom H^1(C,K_C)
\end{gather*}
for $m=1$, which proves~\eqref{hypo1} for $n=i$.
Similarly,
\begin{gather*}
H^1\big(C,F_\hbar^{n-1}\tensor \big(F_\hbar^{r-1}\big)^{-1}\tensor
K_C^{\tensor (-(n-1)+(m-1)+r)}\big)
\\ \qquad
{}=H^1\big(C,F_\hbar^{n-1}\tensor \big(F_\hbar^{r-1}\big)^{-1}\tensor
K_C^{\tensor (-n+m+r)}\big)
\isom H^1(C,K_C^m)=0
\end{gather*}
for $m\ge 2$, which is~\eqref{hypo2} for $n=i$.
By induction on $i$ in the decreasing direction,
we have established that the class $\hbar$
determines the unique extension
\eqref{ext i-1} for every $i$.
\end{proof}

\begin{Def}[${\rm SL}(r,\bC)$-opers]
A point $(E,\nabla)\in \cM_{\deR}$, i.e.,
an irreducible holomorphic ${\rm SL}(r,\bC)$-connection
$\nabla\colon E\lrar E\tensor K_C$
acting on a vector bundle $E$, is an~${\rm SL}(r,\bC)$-oper if the following
conditions are satisfied.

\vspace{1ex}
\textit{Filtration}. There is a filtration $F^\bullet$
by vector subbundles
\begin{gather}
\label{oper filt}
0=F^r\subset F^{r-1}\subset F^{r-2}\subset\cdots\subset F^0 = E.
\end{gather}

\textit{Griffiths transversality}. The~connection respects the filtration:
\begin{gather}
\label{Griffiths}
\nabla|_{F^i}\colon\quad F^i\lrar F^{i-1}\tensor K_C, \qquad i=1, \dots,r.
\end{gather}

\textit{Grading condition}. The~connection induces $\cO_C$-module isomorphisms
\begin{gather}
\label{grading}
\overline{\nabla}\colon\quad F^i/F^{i+1}\overset{\sim}{\lrar}
\big(F^{i-1}/F^{i}\big)\tensor K_C,\qquad i = 1, \dots, r-1.
\end{gather}
\end{Def}

For the purpose of defining differential operators
globally on the Riemann surface $C$, we need
a projective coordinate system
on $C$ subordinating its complex
structure. The~coordinate also allows us to give a concrete
$\hbar\in H^1(C,K_C)$-dependence in the
filtered extensions. For~example, the extension $E_\hbar$ of
\eqref{r=2 extension} is given by
a system of transition functions
\begin{gather*}
E_\hbar \longleftrightarrow
\left\{
\begin{bmatrix}
\xi_{\a\b} &\hbar \sig_{\a\b}\\
0&\xi_{\a\b}^{-1}
\end{bmatrix}
\right\}
\end{gather*}
on each $U_\a\cap U_\b$. The~cocycle condition
for the transition functions translates into a condition
\be
\label{sigma relation}
\sig_{\a\gam} = \xi_{\a\b}\sig_{\b\gam}
+ \sig_{\a\b}\xi_{\b\gam}^{-1}.
\ee
The application of the exterior
differentiation ${\rm d}$ to the cocycle condition
$\xi_{\a\gam} = \xi_{\a\b}\xi_{\b\gam}$ yields
\begin{gather*}
\frac{{\rm d}\xi_{\a\gam}}{{\rm d}x_\gam}{\rm d}x_\gam
= \frac{{\rm d}\xi_{\a\b}}{{\rm d}x_\b}{\rm d}x_\b \xi_{\b\gam}
+\xi_{\a\b}
\frac{{\rm d}\xi_{\b\gam}}{{\rm d}x_\gam}{\rm d}x_\gam.
\end{gather*}
Noticing that
\begin{gather*}
\xi_{\a\b}^2 = \frac{{\rm d}x_\b}{{\rm d}x_\a},
\end{gather*}
we see that
\be
\label{sigma ab}
\sig_{\a\b} := \frac{{\rm d}\xi_{\a\b}}{{\rm d}x_\b}
=\partial_\b \xi_{\a\b}
\ee
solves
\eqref{sigma relation}.
We note that
\begin{gather*}
\begin{bmatrix}
\xi_{\a\b}&\hbar \sig_{\a\b}
\\
&\xi_{\a\b}^{-1}
\end{bmatrix}
=
\exp\left(
\log \xi_{\a\b} {\begin{bmatrix}
1&0\\
0&-1
\end{bmatrix}}
\right)
\exp\left(
\hbar \partial_\b \log \xi_{\a\b}
\begin{bmatrix}
0&1
\\
0&0
\end{bmatrix}
\right).
\end{gather*}
Therefore, in the multiplicative sense, the
extension class is determined by
$\partial_\b \log \xi_{\a\b}$.

\begin{lem}
The extension class
 $ \sig_{\a\b}$ of~\eqref{sigma ab}
 defines a non-trivial extension
~\eqref{r=2 extension}.
\end{lem}

\begin{proof}
The cohomology long exact sequence
\begin{gather*}
\begin{CD}
H^1(C,\bC)@>>> H^1(C,\cO_C)
@>>> H^1(C,K_C) @>\sim>>H^2(C,\bC)\\
@VVV@VVV@|\\
H^1(C,\bC^*)@>>> H^1(C,\cO^*_C)
@>{{\rm d}\log}>> H^1(C,K_C)\\
@VV{0}V@VV{c_1}V@VVV\\
H^2(C,\bZ)@=H^2(C,\bZ)@>>>0
\end{CD}
\end{gather*}
associated with
\begin{gather*}
\begin{CD}
 @.0 @. 0\\
&&@VVV @VVV\\
0@>>>\bZ @=\bZ@>>>0\\
&&@VVV @VVV@VVV\\
0@>>>\bC @>>>\cO_C @>{\rm d}>>K_C @>>> 0\\
&&@VVV @VVV@VVV\\
0@>>>\bC^* @>>>\cO_C^* @>{{\rm d}\log}>>K_C
@>>> 0\\
&&@VVV @VVV@VVV\\
@.0@.0@.0
\end{CD}
\end{gather*}
tells us that $\{\sig_{\a\b}\}$ corresponds
to the image
of $\{\xi_{ab}\}$ via the map
\begin{gather*}
\begin{CD}
H^1(C,\cO^*_C)@>{{\rm d}\log}>>
H^1(C,K_C).
\end{CD}
\end{gather*}
From the exact sequence, we see that
if ${\rm d}\log \{\xi_{\a\b}\} = 0\in H^1(C,K_C)$,
then it comes from a class in $H^1(C,\bC^*)$,
which is the moduli space of
line bundles with holomorphic connections
(see, for example,~\cite{Bloch}).
Hence the first Chern class
of the theta characteristic
$c_1\big(K_C^\half\big)$ should be $0$.
But it is $g-1$, not $0$, since $g(C)\ge 2$.
\end{proof}

\begin{rem}The same exact sequences in the above proof were used in~\cite{Bloch} for
constructing Bloch regulators of the algebraic $K_2$-group. The~torsion property
of the Steinberg symbol of~generators of $K_2(\Sigma)$
and quantizability of a spectral curve $\Sigma$ was first
discussed in~\cite{GS} for
the case of difference equations.
\end{rem}

The class $\{\sigma_{\a\b}\}$
of~\eqref{sigma ab} gives a natural isomorphism
$H^1\left(C,K_C\right)\isom\bC$. We identify
the deformation parameter $\hbar\in \bC$
with the cohomology class $\{\hbar \sig_{\a\b}\}
\in H^1\left(C,K_C\right)=\bC$.
Let
\begin{gather*}
\bq = (q_{2},q_{3},\dots,q_{r})\in B
=\bigoplus_{\ell=2}^r
H^0\big(C,K_C^{\tensor \ell}\big)
\end{gather*}
be an arbitrary point of the Hitchin base $B$.
We trivialize the line bundle $K_C^{\tensor \ell}$
with respect to our projective coordinate chart
$C = \bigcup_\a U_\a$, and write each $q_{\ell}$
as $\{(q_\ell)_\a\}$ that satisfies the
transition relation
\begin{gather*}
(q_\ell)_\a = (q_\ell)_\b \xi_{\a\b}^{2\ell}.
\end{gather*}
The transition function
of the vector bundle $E_0$ is given by
\begin{gather*}
\xi_{\a\b}^H = \exp(H \log \xi_{\a\b}).
\end{gather*}
Since $X_-\colon E_0\lrar E_0\tensor K_C$ is
a global Higgs field, its local expressions
$\{X_- {\rm d}x_\a\}$
with respect to the projective coordinate system
satisfies the transition relation
\be
\label{X- local}
X_-{\rm d}x_\a = \exp(H \log \xi_{\a\b})
X_-{\rm d}x_\b \exp(-H \log \xi_{\a\b})
\ee
on every $U_\a\cap U_\b$. The~same relation holds for the Higgs field
$\phi(\bq)$ as well:
\be
\label{Higgs local}
\phi_\a(\bq){\rm d}x_\a = \exp(H \log \xi_{\a\b})
\phi_\b(\bq){\rm d}x_\b \exp(-H \log \xi_{\a\b}).
\ee

We have the following:

\begin{thm}[construction of ${\rm SL}(r,\bC)$-opers]\label{thm:construction of opers}
On each $U_\a\cap U_\b$ define a
transition function
\be
\label{f hbar}
f_{\a\b}^\hbar :=
\exp(H \log \xi_{\a\b})
\exp\big(\hbar \partial_\b \log\xi_{\a\b}
X_+\big),
\ee
where $\partial_\b = \frac{\rm d}{{\rm d}x_\b}$,
and $\hbar \partial_\b \log\xi_{\a\b}\in H^1(C,K_C)$.
Then
\begin{itemize}\itemsep=0pt
\item The collection
$\big\{f_{\a\b}^\hbar\big\}$ satisfies the cocycle condition
\begin{gather*}
f_{\a\b}^\hbar f_{\b\gam}^\hbar = f_{\a\gam}^\hbar,
\end{gather*}
hence it defines a holomorphic bundle on $C$. It~is exactly the filtered extension
$E_\hbar$ of~Pro\-po\-si\-tion~{\rm \ref{prop:filter}}.

\item The locally defined differential operator
\begin{gather*}
\nabla_\a ^\hbar(0) :={\rm d}+\frac{1}{\hbar}X_- {\rm d}x_a
\end{gather*}
for every $\hbar \ne 0$
forms a global holomorphic connection in
$E_\hbar$, i.e.,
\begin{gather}
\label{d+X gauge equation}
\frac{1}{\hbar}X_-{\rm d}x_\a
= \frac{1}{\hbar}f_{\a\b}^\hbar
X_-{\rm d}x_\b \big(f_{\a\b}^\hbar\big)^{-1}
-{\rm d}f_{\a\b}^\hbar \cdot \big(f_{\a\b}^\hbar\big)^{-1}.
\end{gather}

\item Every point $(E_0,\phi(\bq))
\in\kappa (B)\subset \cM_{\Dol}$ of the Hitchin
section~\eqref{Hitchin section} gives rise to a one-parameter
family of ${\rm SL}(r,\bC)$-opers $\big(E_\hbar,
\nabla^\hbar(\bq)\big)\in \cM_{\deR}$.
In other words, the locally defined differential operator
\begin{gather}
\label{nabla q}
\nabla_\a^\hbar(\bq):= {\rm d} +\frac{1}{\hbar}\phi_\a(\bq){\rm d}x_\a
\end{gather}
for every $\hbar \ne 0$
determines a global holomorphic connection
\begin{gather}
\label{nabla q gauge}
\nabla^\hbar_\a(\bq)
= f_{\a\b}^\hbar\nabla^\hbar_\b(\bq)
\big(f_{\a\b}^\hbar\big)^{-1}
\end{gather}
in~$E_\hbar$ satisfying the oper conditions.

\item
Deligne's $\hbar$-connection
\begin{gather}
\label{Deligne}
\big(E_\hbar,\hbar \nabla^\hbar (\bq)\big)
\end{gather}
interpolates the Higgs pair and the oper, i.e.,
at $\hbar = 0$, the family~\eqref{Deligne}
gives the Higgs pair
$(E,\phi(\bq))\in \cM_{\Dol}$,
and at $\hbar = 1$ it gives an
${\rm SL}(r,\bC)$-oper $\big(E_1,\nabla^1(\bq)\big)
\in \cM_{\deR}$.

\item After a suitable gauge transformation
depending on $\hbar$,
the $\hbar \rar \infty$ limit of
the oper $\nabla^\hbar (\bq)$ exists and is equal to~$\nabla^{\hbar = 1}(0)$.
\end{itemize}
\end{thm}

\begin{proof}
Recall the Baker--Campbell--Hausdorff
formula: Let $A,B$ be elements of a Lie algebra
such that $[A,B] = c B$ for a constant $c\in \bC$. Then
\begin{gather}
\label{BCH}
{\rm e}^A {\rm e}^B {\rm e}^{-A} = {\rm e}^{B \exp c}.
\end{gather}
From this formula,
 $\frac{{\rm d}x_\b}{{\rm d}x_\a} =\xi_{\a\b}^2$, and $[H,X_+] = 2X_+$, we calculate
\begin{gather*}
f_{\a\b}^\hbar \big(f_{\gam\b}^\hbar\big)^{-1} =
\exp(H \log \xi_{\a\b})\exp(\hbar \partial_\b \log\xi_{\a\b}X_+)
\exp(-\hbar \partial_\b \log\xi_{\gam\b}X_+)\exp(-H \log \xi_{\gam\b})
\\ \hphantom{f_{\a\b}^\hbar \big(f_{\gam\b}^\hbar\big)^{-1}}
{}= \exp(H \log \xi_{\a\b})\exp(\hbar \partial_\b \log\xi_{\a\gam}X_+)
\exp(-H \log \xi_{\gam\b})
\\ \hphantom{f_{\a\b}^\hbar \big(f_{\gam\b}^\hbar\big)^{-1}}
{}= \exp(H \log \xi_{\a\b})\exp(\hbar \xi_{\b\gam}^2
\partial_\gam \log\xi_{\a\gam}X_+)\exp(-H \log \xi_{\gam\b})
\\ \hphantom{f_{\a\b}^\hbar \big(f_{\gam\b}^\hbar\big)^{-1}}
{}= \exp(H \log \xi_{\a\gam}) \exp(H \log \xi_{\gam\b})\exp(\hbar \xi_{\b\gam}^2
\partial_\gam \log\xi_{\a\gam}X_+)\exp(-H \log \xi_{\gam\b})
\\ \hphantom{f_{\a\b}^\hbar \big(f_{\gam\b}^\hbar\big)^{-1}}
{}= \exp(H \log \xi_{\a\gam})\exp(\hbar \xi_{\gam\b}^2\xi_{\b\gam}^2
\partial_\gam \log\xi_{\a\gam}X_+)
=f_{\a\gam}^\hbar.
\end{gather*}
Hence the cocycle condition
$f_{\a\b}^\hbar = f_{\a\gam}^\hbar f_{\gam\b}^\hbar$
follows. We note that the factor $\exp c$ in~\eqref{BCH} produces exactly the
cocycle $\xi_{\a\b}^2$ corresponding to the canonical sheaf $K_C$.

To prove~\eqref{d+X gauge equation},
we use the power series expansion of the adjoint action
\begin{gather}
\label{adjoint formula}
{\rm e}^{\hbar A}B{\rm e}^{-\hbar A}
=\sum_{n=0}^\infty\frac{1}{n!} \hbar ^n ({\rm ad}_A)^n(B)
:=\sum_{n=0}^\infty
\frac{1}{n!}\; \hbar^n\;
\overset{n}
{\overbrace{[A,[A,[\cdots,[A}},B]\cdots]]].
\end{gather}
{\samepage
We then find
\begin{gather*}
f_{\a\b}^\hbar X_-\big(f_{\a\b}^\hbar \big)^{-1}
\\[.5ex] \qquad
{}=\exp(H \log \xi_{\a\b})\exp(\hbar \partial_\b \log\xi_{\a\b}X_+)X_-
\exp(-\hbar \partial_\b \log\xi_{\a\b}X_+)\exp(-H \log \xi_{\a\b})
\\[.5ex] \qquad
{}= \exp(H \log \xi_{\a\b})X_-\exp(-H \log \xi_{\a\b})
+ \hbar \partial_\b \log\xi_{\a\b}H
\\[.5ex] \qquad\phantom{=}
{}- \hbar^2 (\partial_\b \log\xi_{\a\b})^2\exp(H \log \xi_{\a\b})X_+
\exp(-H \log \xi_{\a\b}),
\end{gather*}

}\noindent
and
\begin{gather*}
\partial_\b f_{\a\b}^\hbar \big(f_{\a\b}^\hbar \big)^{-1}
=\partial_\b \log\xi_{\a\b}H-\hbar (\partial_\b \log\xi_{\a\b})^2
\exp(H \log \xi_{\a\b})X_+\exp(-H \log \xi_{\a\b}).
\end{gather*}
Here, we have used the formula
\begin{gather*}
\partial_\b \partial_\b\log\xi_{\a\b}
=\partial_\b \big(\xi_{\a\b}^{-1}\partial_\b\xi_{\a\b}\big)
=-\xi_{\a\b}^{-2}(\partial_\b\xi_{\a\b})^2
=- (\partial_\b \log\xi_{\a\b})^2,
\end{gather*}
which follows from~\eqref{xiab}.
Therefore, noticing~\eqref{X- local}, we obtain
\begin{align*}
\bigg(\frac{1}{\hbar}f_{\a\b}^\hbar X_- \big(f_{\a\b}^\hbar \big)^{-1}
-\partial_\b f_{\a\b}^\hbar \big(f_{\a\b}^\hbar \big)^{-1}\bigg) {\rm d}x_\b&=
\frac{1}{\hbar} \exp(H \log \xi_{\a\b})X_-{\rm d}x_\b\exp(-H \log \xi_{\a\b})
\\
&=\frac{1}{\hbar} X_-{\rm d}x_\a.
\end{align*}

To prove~\eqref{nabla q gauge}, we need,
in addition to~\eqref{d+X gauge equation}, the following relation:
\begin{gather}
\label{q gauge computation}
\sum_{\ell = 2}^r (q_\ell)_\a X_+^{\ell-1}
 {\rm d}x_\a = f_{\a\b}^\hbar
\sum_{\ell = 2}^r (q_\ell)_\b X_+^{\ell-1} {\rm d}x_\b\big(f_{\a\b}^\hbar\big)^{-1}.
\end{gather}
But~\eqref{q gauge computation} is obvious from~\eqref{Higgs local}
and~\eqref{f hbar}.

Noticing that $f_{\a\b}^\hbar$ is an
upper-triangular matrix, we denote by
$\big(f_{\a\b}^\hbar\big)_i$
the principal truncation of $f_{\a\b}^\hbar$
to the first $(r-i+1)\times (r-i+1)$ upper-left
corner. For~example,
\begin{gather*}
\big(f_{\a\b}^\hbar\big)_r=\big[\xi_{\a\b}^{r-1}\big],
\\
\big(f_{\a\b}^\hbar\big)_{r-1}=
\begin{bmatrix}
\xi_{\a\b}^{r-1}
 \\
&\xi_{\a\b}^{r-3}
\end{bmatrix}
\begin{bmatrix}
1&\hbar \sqrt{s_1}\xi_{\a\b}^{-1}\sig_{\a\b}
 \\[.5ex]
&1
\end{bmatrix}
=
\begin{bmatrix}
\xi_{\a\b}^{r-1}&\hbar\sqrt{s_1}\xi_{\a\b}^{r-2}
\sig_{\a\b}
 \\[.5ex]
&\xi_{\a\b}^{r-3}
\end{bmatrix},
\\[1ex]
\big(f_{\a\b}^\hbar\big)_{r-2}=
\begin{bmatrix}
\xi_{\a\b}^{r-1}&\hbar \sqrt{s_1}\xi_{\a\b}^{r-2}\sig_{\a\b}
&\half\hbar^2\sqrt{s_1s_2}\xi_{\a\b}^{r-3}\sig_{\a\b}^2
 \\[.5ex]
&\xi_{\a\b}^{r-3}& \hbar \sqrt{s_2}\xi_{\a\b}^{r-4}\sig_{\a\b}
\\[.5ex]
&&\xi_{\a\b}^{r-5}
\end{bmatrix},\qquad \text{etc.}
\end{gather*}
They all satisfy the cocycle condition
with respect to the projective structure, and
define a~sequence of
vector bundles $F_\hbar^i$ on $C$.
From the shape of matrices
we see that these vector bundles
give the filtered extension
$(F_\hbar^\bullet, E_\hbar)$
of Proposition~\ref{prop:filter},
and hence satisfies the requirement~\eqref{oper filt}.
Since the connection $\nabla^\hbar(\bq)$ has
non-zero lower-diagonal entries in its
matrix representation coming from $X_-$,
$F_\hbar^i$ is not invariant under $\nabla^\hbar(\bq)$.
But since $X_-$ has non-zero entries
exactly along the first lower-diagonal,~\eqref{Griffiths} holds.
The~isomorphism~\eqref{grading} is~a consequence of~\eqref{si},
since $s_i = i(r-i)\ne 0$ for $i=1, \dots, r-1$.

Finally, the gauge transformation of
$\nabla^\hbar(\bq)$ by a bundle automorphism
\begin{gather}
\label{hbar gauge transf 1}
\hbar^{-\frac{H}{2}}=
\begin{bmatrix}
\hbar^{-\frac{r-1}{2}}&
\\
&\ddots&
\\
&&\hbar^{\frac{r-1}{2}}
\end{bmatrix}
\end{gather}
on each coordinate neighborhood $U_\a$ gives
\begin{gather}
\label{hbar gauge transf 2}
{\rm d}+\frac{1}{\hbar}\phi(\bq)\longmapsto\hbar^{-\frac{H}{2}}
\bigg({\rm d}+\frac{1}{\hbar}\phi(\bq)\bigg)\hbar^{\frac{H}{2}}
={\rm d}+X_-+ \sum_{\ell=2}^r \frac{q_\ell}{\hbar^\ell}X_+^{\ell-1}.
\end{gather}
This is because
\begin{gather*}
\hbar^{-\frac{H}{2}}X_-\hbar^{\frac{H}{2}}
= \hbar X_- \qquad \text{and}\qquad
\hbar^{-\frac{H}{2}}X_+^\ell \hbar^{\frac{H}{2}}
= \hbar^{-\ell} X_+^\ell,
\end{gather*}
which follows from the adjoint formula
\eqref{adjoint formula}.
Therefore,
\begin{gather*}
\lim_{\hbar\rar \infty}\nabla^\hbar(\bq) \sim
{\rm d}+X_- = \nabla^{\hbar=1}(0),
\end{gather*}
where the symbol $\sim$ means gauge equivalence.

This completes the proof of the theorem.
\end{proof}

\begin{rem}
In the construction theorem, our
use of a projective coordinate system is
essential, through
\eqref{xiab}. Only in such a coordinate,
the definition of the global connection
\eqref{nabla q gauge} makes sense. This is due
to the vanishing of the second derivative
of $\xi_{\a\b}$.
\end{rem}

The above construction theorem
yields the following.

\begin{thm}[biholomorphic quantization of Hitchin spectral curves]\label{thm:quantization-holomorphic}
Let $C$ be a compact Riemann surface
of genus $g\ge 2$
with a chosen projective coordinate system
subordinating its complex structure.
We denote by $\cM_{\Dol}$ the moduli
space of stable holomorphic ${\rm SL}(r,\bC)$-Higgs
bundles over $C$, and by $\cM_{\deR}$ the
moduli space of irreducible
holomorphic ${\rm SL}(r,\bC)$-connections
on $C$. For~a fixed theta characteristic
$K_C^\half$, we have a Hitchin section
$\kappa(B)\subset \cM_{\Dol}$ of~\eqref{Hitchin section}.
We denote by $Op\subset \cM_{\deR}$ the moduli
space of ${\rm SL}(r,\bC)$-opers with the
condition that the second term of the filtration
is given by
$F^{r-1} = K_C^{\frac{r-1}{2}}$.
Then the map
\be\label{biholomorphic quantization}
\cM_{\Dol}\supset
\kappa(B)\owns (E_0,\phi(\bq) )
\overset{\gam}{\longmapsto}
\big(E_\hbar,\nabla^\hbar(\bq)\big)
\in Op\subset \cM_{\deR}
\ee
evaluated at $\hbar = 1$
is a biholomorphic map with respect to the
natural complex structures induced from the
ambient spaces.

The biholomorphic quantization
\eqref{biholomorphic quantization} is also
$\bC^*$-equivariant. The~$\lam\in \bC^*$ action on the Hitchin
section is defined by
$\phi\longmapsto \lam\phi$. The~oper corresponding to~$(E_0,\lam\phi(\bq))\in \kappa(B)$ is
${\rm d}+\frac{\lam}{\hbar}\phi(\bq)$.
\end{thm}

\begin{proof}
The $\bC^*$-equivariance follows from the
same argument of the gauge transformation~\eqref{hbar gauge transf 1}, \eqref{hbar gauge transf 2}.
Since the Hitchin section $\kappa$ is not
the section of the Hitchin fibration $\mu_H$, we~need the gauge transformation. The~action $\phi\longmapsto \lam\phi$
on the Hitchin section induces a weighted action
\begin{gather*}
B\owns (q_2,q_3,\dots,q_r)
\longmapsto \left(\lam^2 q_2,\lam^3 q_3,
\dots, \lam^r q_r\right)\in B
\end{gather*}
through $\mu_H$.
Then we have the gauge equivalence via the gauge
transformation
$\big(\frac{\lam}{\hbar}\big)^{\frac{H}{2}}$:
\begin{gather*}
{\rm d} + \frac{\lam}{\hbar}\phi(\bq)
\sim\bigg(\frac{\lam}{\hbar}\bigg)^{\frac{H}{2}}
\bigg({\rm d} + \frac{\lam}{\hbar}\phi(\bq)\bigg)
\bigg(\frac{\lam}{\hbar}\bigg)^{-\frac{H}{2}}
= {\rm d} + X_- +\sum_{\ell=2}^r \frac{\lam^\ell q_\ell}
{\hbar^\ell} X_+^{\ell-1}.\tag*{\qed}
\end{gather*}
\renewcommand{\qed}{}
\end{proof}


\subsection[Semi-classical limit of SL(r,C)-opers]
{Semi-classical limit of $\boldsymbol{{\rm SL}(r,\bC)}$-opers}

A holomorphic connection on a
compact Riemann surface $C$ is automatically flat.
Therefore, it~defines a $\cD$-module over
$C$. Continuing the last subsection's conventions,
let us fix a projective coordinate system on $C$,
and let $(E_0,\phi(\bq)) = \kappa(\bq)$
be a point on the Hitchin section of~\eqref{Hitchin section}.
It~uniquely defines an $\hbar$-family of opers
$\big(E_\hbar,\nabla^\hbar(\bq)\big)$.

In this subsection, we establish that
the $\hbar$-connection $\hbar \nabla^\hbar(\bq)$
defines a family of Rees $\cD$-modules on
$C$ parametrized by $B$ such that
the semi-classical limit of the family agrees with the
family of spectral curves~\eqref{family of spectral} over $B$.

To calculate the semi-classical limit,
let us trivialize the vector bundle $E_\hbar$
on each simply connected coordinate neighborhood
$U_\a$ with
coordinate $x_\a$ of the chosen
projective coordinate system. A~flat section $\Psi_\a$ of $E_\hbar$ over
$U_\a$ is a solution of
\begin{gather}
\label{trivialization}
\hbar \nabla^\hbar_\a(\bq)\Psi_\a:=(\hbar {\rm d} + \phi_\a (\bq))
\begin{bmatrix}
\psi_{r-1}\\\psi_{r-2}\\\vdots\\\psi_1\\\psi
\end{bmatrix}_\a
=0,
\end{gather}
with an appropriate unknown function $\psi$.
Since $\Psi_\a = f_{\a\b}^\hbar \Psi_\b$,
the function $\psi$ on $U_\a$
satisfies the transition relation
$(\psi)_\a = \xi_{\a\b}^{-r+1}(\psi)_\b$. It~means
that $\psi$ is actually
a local section of the line bundle $K_C^{-\frac{r-1}{2}}$.
There are $r$ linearly independent solutions
of~\eqref{trivialization}, because
$q_2,\dots,q_r$ are represented by
holomorphic functions on $U_\a$. The~entries of $X_+^\ell$ are given by the formula
\begin{gather*}
X_+^\ell =\big[s_{ij}^{(\ell)}\big], \qquad
s_{ij}^{(\ell)}= \delta_{i+\ell,j}\sqrt{s_i s_{i+1}\cdots s_{i+\ell-1}}.
\end{gather*}
Therefore, \eqref{trivialization} is equivalent to
\begin{gather}
0=\sqrt{s_{r-k-1}}\psi_{k+1} +\hbar \psi_k'
+ \sqrt{s_{r-k}}q_2\psi_{k-1}+\sqrt{s_{r-k}s_{r-k+1}}q_3\psi_{k-2}+\cdots\nonumber
\\[.5ex] \hphantom{0=}
{}+\sqrt{s_{r-k}s_{r-k+1}\cdots s_{r-1}}q_{k+1}\psi\nonumber
\\[.5ex] \hphantom{0}
{}=\sqrt{s_{k+1}}\psi_{k+1} +\hbar \psi_k'
+ \sqrt{s_{k}}q_2\psi_{k-1}+\sqrt{s_{k}s_{k-1}}q_3\psi_{k-2}+\cdots
+\sqrt{s_{k}s_{k-1}\cdots s_{1}}q_{k+1}\psi
\label{psi k}
\end{gather}
for $k = 0, 1, \dots, r-1$,
where we use $s_k = s_{r-k}$.
Note that the differentiation
is always multiplied by $\hbar$
as $\hbar {\rm d}$ in~\eqref{trivialization}, and that
 $\phi(\bq)$ is independent of $\hbar$ and takes the form
\begin{gather*}
\phi(\bq)\setlength{\arraycolsep}{4pt}
=\begin{bmatrix}
0&\sqrt{s_{r-1}}q_2&\sqrt{s_{r-2}s_{r-1}}q_3&\cdots&\cdots
&\sqrt{s_2\cdots s_{r-1}}q_{r-1}
&\sqrt{s_1s_2\cdots s_{r-1}}q_r
\\[.5ex]
\sqrt{s_{r-1}}&0&\sqrt{s_{r-2}}q_2&\cdots&\cdots&
\sqrt{s_2\cdots s_{r-2}}q_{r-2}
&
\sqrt{s_1\cdots s_{r-2}}q_{r-1}
\\[.5ex]
&\sqrt{s_{r-2}}&0&\ddots&\cdots&\sqrt{s_2\cdots s_{r-3}}q_{r-3}
&\sqrt{s_1\cdots s_{r-3}}q_{r-2}
\\[.5ex]
&&\ddots&\ddots&\ddots&\vdots&\vdots
\\[.5ex]
&&&\sqrt{s_3}&0&\sqrt{s_2}q_2&\sqrt{s_1s_2}q_3
\\[.5ex]
&&&&\sqrt{s_2}&0&\sqrt{s_{1}}q_2
\\[.5ex]
&&&&&\sqrt{s_{1}}&0
\end{bmatrix}\!.
\end{gather*}
By solving~\eqref{psi k} for $k=0, 1, \dots, r-2$
recursively, we obtain an expression of $\psi_k$ as
a linear combination of
\begin{gather*}
\psi=\psi_0,\qquad \hbar \psi'=\hbar \frac{\rm d}{{\rm d}x_\a}\psi,\qquad \dots,\qquad\hbar^k\psi^{(k)}=\hbar^k\frac{{\rm d}^k}{{\rm d}x_\a^k}\psi,
\end{gather*}
with coefficients in differential polynomials of
$q_2, q_3,\dots,q_k$. Moreover,
 the coefficients of these
differential polynomials are in $\overline{\bQ}[\hbar]$. For~example,
\begin{gather*}
\psi_1 =-\frac{1}{\sqrt{s_1}}\hbar \psi',
\\
\psi_2 =\frac{1}{\sqrt{s_1s_2}}(\hbar^2 \psi'' -s_1q_2\psi),
\\
\psi_3 =\frac{1}{\sqrt{s_1s_2s_3}}
(-\hbar^3 \psi'''+\hbar(s_1+s_2)q_2 \psi'
+(\hbar s_1 q_2'-s_1s_2q_3)\psi),
\\
\psi_4 =\frac{1}{\sqrt{s_1s_2s_3s_4}}
\big(\hbar^4\psi''''-\hbar^2 (s_1+s_2+s_3)q_2\psi''
+(-\hbar^2(2s_1+s_2)q_2'+\hbar(s_1s_2+s_2s_3)q_3)\psi'
\\ \hphantom{\psi_4 =}
+(\hbar^2s_1q_2''-\hbar s_1s_2 q_3'+s_1s_3q_2^2-s_1s_2s_3q_4)\psi\big), \quad \text{etc.}
\end{gather*}
Since $\psi_1$ is proportional to~$\psi'$,
inductively we can show that the linear
combination expression of~$\psi_k$ by derivatives of $\psi$
does not contain the $(k-1)$-th order differentiation
of $\psi$. Equation~\eqref{psi k} for $k = r-1$ is
a differential equation
\begin{gather}
\label{order r}
\hbar \psi_{r-1}' +\sqrt{s_1}q_2\psi_{r-2}
+\sqrt{s_1s_2}q_3\psi_{r-3}+\cdots
+\sqrt{s_1s_2\cdots s_{r-1}}q_r\psi =0,
\end{gather}
which is an order $r$ differential equation
for $\psi\in K_C^{-\frac{r-1}{2}}$.
Its actual shape can be calculated
using the procedure of the proof of
the next theorem. The~equation becomes~\eqref{order r qc}.
Since we~are using a fixed projective coordinate
system, the connection $\nabla^\hbar(\bq)$
takes the same form on~each coordinate
neighborhood $U_\a$. Therefore,
the shape of the differential equation
of~\eqref{order r} as an~equation for
$\psi\in K_C^{-\frac{r-1}{2}}$ is again the
same on every coordinate neighborhood.

We are now ready to calculate the semi-classical
limit of the Rees $\cD$-module
corresponding to~$\hbar\nabla^\hbar(\bq)$. The~following is our main theorem of the paper.

\begin{thm}[semi-classical limit of an oper]\label{thm:SCL}
Under the same setting of Theorem~$\ref{thm:quantization-holomorphic}$,
let~$\cE(\bq)$ denote the Rees $\cD$-module
$\big(E_\hbar, \hbar\nabla^\hbar(\bq)\big)$
associated with the oper of~\eqref{biholomorphic quantization}. Then the semi-classical limit
of $\cE(\bq)$ is the spectral curve $\sigma^*\Sigma\subset T^*C$
of $-\phi(\bq)$ defined by the cha\-rac\-te\-ris\-tic equation
$\det(\eta+\phi(\bq))=0$, where $\sigma$ is the involution of~\eqref{involution}.
\end{thm}

\begin{rem}
The actual shape of the single scalar valued
differential equation~\eqref{order r} of order~$r$ is given
by~\eqref{order r qc} below. This is exactly what
we usually refer to as the
\textit{quantum curve} of~the
spectral curve $\det(\eta+\phi(\bq))=0$.
\end{rem}

\begin{proof}
Since the connection and the differential equation
\eqref{order r} are globally defined,
we need to analyze the
semi-classical limit only at each coordinate neighborhood
$U_\a$ with a projective coordinate $x_\a$.
Let
\begin{gather*}
\Psi_0 = \begin{bmatrix}
\hbar^{r-1}\psi^{(r-1)}\\ \vdots\\ \hbar \psi'\\ \psi
\end{bmatrix}\!.
\end{gather*}
From the choice of local trivialization~\eqref{trivialization} we
have an expression
\begin{gather}
\label{A}
\Psi = \big(\Delta+A(\bq, \hbar)\big)\Psi_0,
\end{gather}
where $\Delta$ is a constant diagonal
matrix
\begin{gather}
\label{Delta}
\Delta =
\bigg[\delta_{ij}(-1)^{r-i} \frac{1}{\sqrt{s_{r-1}s_{r-2}\cdots s_{i}}}\bigg]_{i,j=1,\dots,r}
\end{gather}
with the understanding that its $(r,r)$-entry is $1$,
and $A(\bq,\hbar) = \big[a(\bq,\hbar)_{i,j}\big]$
satisfies the following properties:
\begin{cond}
\label{cond:A}\qquad
\begin{itemize}\itemsep=0pt

\item $A(\bq,\hbar)$ is a nilpotent upper triangular matrix.

\item Each entry
$a(\bq,\hbar)_{i,j}$ is a differential
polynomial in $q_2,\dots,q_r$
with coefficients in $\overline{\bQ}[\hbar]$.

\item All diagonal and the first upper diagonal
entries of $A(\bq,\hbar)$ are $0$:
\begin{gather*}
a(\bq,\hbar)_{i,i} = a(\bq,\hbar)_{i,i+1} =0,
\qquad i = 1, \dots, r.
\end{gather*}
\end{itemize}
\end{cond}
The last property is because $\psi_k$ does not
contain the $(k-1)$-th derivative of $\psi$,
as we have noted above. The~diagonal matrix $\Delta$ is designed so that
\begin{gather*}
\Delta^{-1}X_-\Delta = \begin{bmatrix}
0\\[-.5ex]
-1&0
\\[-.5ex]
&-1&0
\\[-.5ex]
&&\ddots
\\[-.5ex]
&&&-1&0
\end{bmatrix}
=:-X.
\end{gather*}
Let us calculate the gauge transformation
of $\hbar\nabla^\hbar(\bq)$ by $\Delta+A(\bq, \hbar)$:
\begin{gather}
\label{gauge transf to canonical}
(\Delta + A(\bq, \hbar))^{-1}(\hbar d +\phi(\bq))(\Delta + A(\bq,\hbar))=
(\hbar d -X{\rm d}x_\a - \omega(\bq,\hbar){\rm d}x_\a).
\end{gather}
We wish to determine the matrix $\omega(\bq,\hbar)$.

\begin{lem} The matrix
$\omega(\bq,\hbar)$ of~\eqref{gauge transf to canonical}
 consists of only the first row,
and the matrix $X+\omega(\bq,\hbar)$
takes the following canonical form
\begin{gather}
\label{connection canonical form}
X+\omega(\bq,\hbar) =
\begin{bmatrix}
0& \omega_2(\bq,\hbar)&\cdots &\omega_{r-1}(\bq,\hbar)
& \omega_r(\bq,\hbar)
\\
1&0&\cdots&0&0
\\[-.5ex]
&1&\ddots&\vdots&\vdots
\\[-.5ex]
&&\ddots&0&0
\\[-.5ex]
&&&1&0
\end{bmatrix}\!.
\end{gather}
\end{lem}

\begin{proof}[Proof of Lemma]
First, define
\begin{gather}
\label{B}
B(\bq,\hbar):=\sum_{n=1}^{r-1}(-1)^n
\big(\Delta^{-1}A(\bq, \hbar)\big)^n
\end{gather}
so that
\begin{gather*}
I + B(\bq,\hbar)=
(I + \Delta^{-1}A(\bq, \hbar))^{-1}.
\end{gather*}
By definition, $B(\bq,\hbar)$ satisfies the exact same
properties listed in~Condition~\ref{cond:A} above.
We~then have
\begin{gather*}
(\Delta + A(\bq, \hbar))^{-1}
\bigg(\hbar \frac{\rm d}{{\rm d}x_\a} +\phi(\bq)\bigg)
(\Delta + A(\bq,\hbar))
\\[-.5ex] \qquad
{}= (I + \Delta^{-1}A(\bq, \hbar))^{-1}\Delta^{-1}
\bigg(\hbar \frac{\rm d}{{\rm d}x_\a} +X_-+\sum_{\ell=2}^r
q_\ell X_+^{\ell-1}\bigg)\Delta(I + \Delta^{-1}A(\bq,\hbar))
\\[-.5ex] \qquad
{}=(I + B(\bq, \hbar))\bigg(\hbar \frac{\rm d}{{\rm d}x_\a} -X+\Delta^{-1}\sum_{\ell=2}^r
q_\ell X_+^{\ell-1}\Delta\bigg)(I + \Delta^{-1}A(\bq,\hbar))
\\[-.5ex] \qquad
{}=\hbar \frac{\rm d}{{\rm d}x_\a} -X-B(\bq,\hbar)X -X\Delta^{-1}A(\bq,\hbar)+
\hbar(I + B(\bq, \hbar))\Delta^{-1}\frac{{\rm d}A(\bq, \hbar)}{{\rm d}x_\a}
\\[-.5ex] \qquad \phantom{=}
{}-B(\bq,\hbar)X\Delta^{-1}A(\bq,\hbar)+(I + B(\bq, \hbar))
\Delta^{-1}\bigg(\sum_{\ell=2}^r q_\ell X_+^{\ell-1}\Delta\bigg)
(I + \Delta^{-1}A(\bq,\hbar)).
\end{gather*}
\pagebreak

\noindent
Thus we define
\begin{gather}
\omega(\bq,\hbar)=B(\bq,\hbar)X +X\Delta^{-1}A(\bq,\hbar)
-\hbar(I + B(\bq, \hbar))\Delta^{-1}\frac{{\rm d}A(\bq, \hbar)}{{\rm d}x_\a}\nonumber
\\ \hphantom{\omega(\bq,\hbar)=}
{}+B(\bq,\hbar)X\Delta^{-1}A(\bq,\hbar)
-(\Delta + A(\bq, \hbar))^{-1}\bigg(\sum_{\ell=2}^r
q_\ell X_+^{\ell-1}\bigg)(\Delta + A(\bq,\hbar)).
\label{omega}
\end{gather}
Every single matrix in the right-hand side of
\eqref{omega} is either diagonal or
upper triangular and nilpotent, except for $X$.
Because of Condition~\ref{cond:A},
$B(\bq,\hbar)X +X\Delta^{-1}A(\bq,\hbar)$ is
an upper tri\-an\-gu\-lar matrix with $0$
along the diagonal.
Obviously, so are all other terms of~\eqref{omega}.
Thus~$\omega(\bq,\hbar)$ is upper triangular
and nilpotent, and is a polynomial in $\hbar$.

Now we note that~\eqref{trivialization} and~\eqref{A} yield
\begin{gather}
\label{Psi0 eq}
\bigg(\hbar \frac{\rm d}{{\rm d}x_\a}-X-\omega(\bq,\hbar)\bigg)
\Psi_0 = 0.
\end{gather}
 The basis vector
$\Psi_0$ defined on a simply connected
coordinate neighborhood $U_\a$ is designed
so that the following
equation holds for any solution $\psi$ of~\eqref{trivialization}:
\begin{gather*}
\bigg(\hbar \frac{\rm d}{{\rm d}x_\a}-X\bigg)\Psi_0=
\bigg(\hbar \frac{\rm d}{{\rm d}x_\a}-X\bigg)
\begin{bmatrix}
\hbar^{r-1}\psi^{(r-1)}\\
\hbar^{r-2}\psi^{(r-2)}\\
\vdots
\\
\hbar \psi'
\\
\psi
\end{bmatrix} =
\begin{bmatrix}
\hbar^r\psi^{(r)}\\
0\\
\vdots\\
0\\
0
\end{bmatrix}\!.
\end{gather*}
Let $\psi_{[1]},\dots,\psi_{[r]}$
be $r$ linearly independent solutions of
\eqref{trivialization}. The~\emph{Wronskian matrix} is defi\-ned~by
\begin{gather*}
\mathbf{W}
=
\begin{bmatrix}
\hbar^{r-1}\psi_{[1]}^{(r-1)}&\hbar^{r-1}\psi_{[2]}^{(r-1)}
&\cdots&\hbar^{r-1}\psi_{[r]}^{(r-1)}
\\
\hbar^{r-2}\psi_{[1]}^{(r-2)}&\hbar^{r-2}\psi_{[2]}^{(r-2)}
&\cdots&\hbar^{r-2}\psi_{[r]}^{(r-2)}
\\
\vdots&\vdots&\ddots&\vdots
\\
\hbar \psi_{[1]}'&\hbar \psi_{[2]}'&\cdots&\hbar \psi_{[r]}'
\\
\psi_{[1]}&\psi_{[2]}&\cdots&\psi_{[r]}
\end{bmatrix}\!.
\end{gather*}
Then from~\eqref{Psi0 eq}, we have
\begin{gather*}
\omega(\bq,\hbar)=\bigg[
\bigg(\hbar \frac{\rm d}{{\rm d}x_\a}-X\bigg)\mathbf{W}\bigg]
\cdot \mathbf{W}^{-1}=
\begin{bmatrix}
\hbar^r\psi_{[1]}^{(r)}&\hbar^r\psi_{[2]}^{(r)}&\cdots&\hbar^r\psi_{[r]}^{(r)}
\\
0&0&\cdots&0
\\
\vdots&\vdots&\ddots&\vdots
\\
0&0&\cdots&0
\end{bmatrix}
\mathbf{W}^{-1}.
\end{gather*}
Clearly we see that $\omega(\bq,\hbar)$
has non-zero entries only in the first row. The~$(1,1)$-entry is also $0$ because $\tr\, \omega(\bq,\hbar)=0$.
\end{proof}

Thus~\eqref{trivialization}
is equivalent to a single
linear ordinary differential equation of order $r$
\begin{gather}
\label{order r qc}
P_\a(x_\a,\hbar;\bq)\psi:=
\bigg[\hbar^r\bigg(\frac{\rm d}{{\rm d}x_\a}
\bigg)^r -\sum_{i=2}^r \hbar^{r-i}\omega_i(\bq,\hbar)
\bigg(\frac{\rm d}{{\rm d}x_\a}\bigg)^{r-i}\bigg]\psi = 0.
\end{gather}
In other words, $P_\a(x_\a,\hbar;\bq)$ is
a generator of the Rees $\cD$-module
$\left(E_\hbar, \hbar\nabla^\hbar(\bq)\right)$
on $U_\a$.
This expression of differential equation of
order $r$ is what is commonly referred to as a
\textit{quantum curve}. The~coefficients $\omega_i(\bq,\hbar)$ are
calculated by determining the matrix
$A(\bq,\hbar)$ of~\eqref{A}. Then
\eqref{Delta},~\eqref{B}, and~\eqref{omega}
give the exact form of $P_\a(x_\a,\hbar;\bq)$.
As we can see, its shape
is quite involved, and is \emph{not}
obtained by simply
replacing $y$
in $\det\big(y+\phi(\bq)\big)$
by $\hbar\frac{\rm d}{{\rm d}x_\a}$.
In~particular, the coefficients contain
derivatives of $q_\ell$'s, which never appear in
the characteristic polynomial of $-\phi(\bq)$.
What we are going to prove now is that nonetheless, the
semi-classical limit of~$P_\a(x_\a,\hbar;\bq)$ is
exactly the characteristic polynomial.

As defined in Definition~\ref{def:SCL},
the semi-classical limit of~\eqref{order r qc}
is the limit
\begin{gather}
\label{char poly = qc}
\lim_{\hbar \rar 0}
{\rm e}^{-\frac{1}{\hbar}S_0(x_\a)}
\bigg[\hbar^r\bigg(\frac{\rm d}{{\rm d}x_\a}\bigg)^r
\!\!-\!\sum_{i=2}^r \hbar^{r-i}\omega_i(\bq,\hbar)
\bigg(\frac{\rm d}{{\rm d}x_\a}\bigg)^{r-i}\bigg]
{\rm e}^{\frac{1}{\hbar}S_0(x_a)}
y^r\!-\!\sum_{i=2}^r \omega_i(\bq,0) y^{r-i},
\end{gather}
where $S_0(x_\a)$ is a holomorphic
function on $U_\a$ so that ${\rm d}S_0 = y{\rm d}x_\a$
gives a local trivialization of~$T^*C$ over
$U_\a$.
Since $\omega(\bq,\hbar)$ is a polynomial
in $\hbar$, we can evaluate it at $\hbar = 0$.
Notice that~\eqref{char poly = qc} is
 the characteristic polynomial
of the matrix $X+\omega(\bq,\hbar)$
of~\eqref{connection canonical form}
at $\hbar = 0$. The~computation of semi-classical limit
is the same as the calculation of
the determinant
of the connection $\hbar\nabla^\hbar(\bq)$, after
taking conjugation by the scalar diagonal
matrix ${\rm e}^{-\frac{1}{\hbar}S_0(x_\a)}I_{r\times r}$,
and then take the limit as~$\hbar$ goes to~$0$:
\begin{gather*}
{\rm e}^{-\frac{1}{\hbar}S_0(x_\a)}I\cdot
(\Delta + A(\bq,\hbar))^{-1}
\bigg(\hbar \frac{\rm d}{{\rm d}x_\a} +\phi(\bq)\bigg)
(\Delta + A(\bq,\hbar)) \cdot
{\rm e}^{\frac{1}{\hbar}S_0(x_\a)}I
\\ \qquad
{}=(\Delta + A(\bq,\hbar))^{-1}
{\rm e}^{-\frac{1}{\hbar}S_0(x_\a)}I\cdot
\bigg(\hbar \frac{\rm d}{{\rm d}x_\a} +\phi(\bq)\bigg)
\cdot {\rm e}^{\frac{1}{\hbar}S_0(x_\a)}I
(\Delta + A(\bq,\hbar))
\\ \qquad
{}=(\Delta + A(\bq,\hbar))^{-1}
\bigg(\frac{{\rm d}S_0(x_\a)}{{\rm d}x_\a} +\phi(\bq)\bigg)
(\Delta + A(\bq,\hbar)) +O(\hbar)
\\ \qquad
{}\overset{\hbar\rar 0}{\lrar}(\Delta + A(\bq,0))^{-1}
(y +\phi(\bq))(\Delta + A(\bq,0)).
\end{gather*}
The determinant of the above matrix is
the characteristic polynomial $\det \left(y+\phi(\bq)\right)$.
Note that from~\eqref{gauge transf to canonical},
we have
\begin{gather*}
{\rm e}^{-\frac{1}{\hbar}S_0(x_\a)}I\cdot
(\Delta + A(\bq,\hbar))^{-1}
\bigg(\hbar \frac{\rm d}{{\rm d}x_\a} +\phi(\bq)\bigg)
(\Delta + A(\bq,\hbar)) \cdot
{\rm e}^{\frac{1}{\hbar}S_0(x_\a)}I
\\ \qquad
{}={\rm e}^{-\frac{1}{\hbar}S_0(x_\a)}I\cdot
\bigg(\hbar \frac{\rm d}{{\rm d}x_\a} -(X + \omega(\bq,\hbar))\bigg)\cdot
{\rm e}^{\frac{1}{\hbar}S_0(x_\a)}I
\\ \qquad
{}=\hbar \frac{\rm d}{{\rm d}x_\a}+y -(X + \omega(\bq,\hbar))
\overset{\hbar\rar 0}{\lrar}y -(X + \omega(\bq,0)).
\end{gather*}
Taking the determinant of the above,
we conclude that
\begin{gather*}
\det \left(y+\phi(\bq)\right) =
y^r-\sum_{i=2}^r \omega_i(\bq,0)\; y^{r-i}.
\end{gather*}
This completes the proof of Theorem~\ref{thm:SCL}.
\end{proof}

For every $\hbar\in H^1(C,K_C)$,
the $\hbar$-connection
$\left(E_\hbar,\hbar \nabla^\hbar(\bq)\right)$
of~\eqref{Deligne} defines a global Rees
$\cD_C$-module structure in $E_\hbar$.
Thus we have constructed a universal family
$\cE_C$ of Rees
$\cD_C$-modules on a given $C$
with a fixed spin and a projective structures:
\begin{gather*}
\begin{CD}
\cE_C @<\supset << \left(E_\hbar,\nabla^\hbar(\bq)\right)
\\
@VVV @VVV
\\
C\times B\times H^1(C,K_C)
@<\supset<< \;C\times \{\bq\}\times\{\hbar\}.
\end{CD}
\end{gather*}
The universal family $\cS_C$ of spectral curves is defined
over $C\times B$.
\begin{gather*}
\begin{CD}
\bP\left(K_C\dsum \cO_C\right)\times B
@<\supset<< \cS_C @<\supset << \left(\det\big(\eta-\phi(\bq)\big)\right)_0
\\
@VVV@VVV @VVV
\\
C\times B @=C\times B
@<\supset<< \;C\times \{\bq\}.
\end{CD}
\end{gather*}
The semi-classical limit is thus a map of families
\be\label{family SCL}
\begin{CD}
\cE_C @>>> \cS_C
\\
@VVV @VVV
\\
C\times B\times H^1(C,K_C)
@>>> \;C\times B.
\end{CD}
\ee

Our concrete construction
\eqref{nabla q} using the projective coordinate system
is not restricted to the holomorphic Higgs field
$\phi(\bq)$.
Since the moduli spaces appearing in the
correspondence~\eqref{biholomorphic quantization}
become more subtle to define due to the
\emph{wildness} of connections, we avoid
the moduli problem, and state our quantization
theorem as a generalization of
Theorem~\ref{thm:construction of opers}.

Let $C$ be a smooth projective algebraic
curve of an arbitrary genus, and
$C = \cup_{\a} U_\a$ a projective coordinate
system subordinating the complex structure of~$C$.

\begin{thm}[quantization of meromorphic data]\label{thm:quantization-meromorphic}
Let $D$ be an effective divisor of~$C$, and
\begin{gather*}
\bq\in B(D) :=\bigoplus_{\ell = 2}^r H^0\big(C,K_C(D)^{\tensor \ell}\big).
\end{gather*}
We define a meromorphic Higgs bundle
$\big(E_0,\phi(\bq)\big)$
by the same formulae~\eqref{E0} and~\eqref{phi q},
as well as a meromorphic oper
$\big(E_\hbar,\nabla^\hbar(\bq)\big)$
by~\eqref{nabla q}. The~meromorphic oper
defines a meromorphic Rees $\cD$-module
\begin{gather*}
\cE(\bq) = \big(E_\hbar, \hbar\nabla^\hbar(\bq))
\end{gather*}
on $C$. Then the semi-classical limit of
$\cE(\bq)$ is the spectral curve
\begin{gather}
\label{spectral meromorphic}
\big(\det(\eta +\phi(\bq)\big)_0 \subset \overline{T^*(C)}
\end{gather}
of $-\phi(\bq)$.
\end{thm}

\begin{proof}
The proof is exactly the same as
Theorem~\ref{thm:SCL} on each simply connected
coordinate neighborhood $U_\a
\subset C\setminus \supp(D)$.
Thus the semi-classical limit
of $\cE(\bq)|_{C\setminus \supp(D)}$
as a divisor is
defined in
$\overline{T^*(C)}\setminus \pi^{-1}(D)$. The~spectral curve~\eqref{spectral meromorphic} is its closure
in $\overline{T^*C}$
with respect to the
complex topology. Thus by
Definition~\ref{Def:QC meromorphic},
\eqref{spectral meromorphic} is the semi-classical~limit of the meromorphic extension $\cE(\bq)$.
\end{proof}

\begin{ex}\label{example}
 Here we list
characteristic polynomials
and differential operators $P_\a(x_\a,\hbar; \bq)$ of~\eqref{order r qc} for $r=2,3,4$.
These formulas show that our quantization
procedure is quite non-trivial.
\begin{itemize}\itemsep=0pt
\item $r=2$:
\begin{gather*}
\det(y+\phi(\bq)) = y^2 -q_2,
\\
P_\a(x_\a,\hbar; \bq)=\bigg(\hbar\frac{\rm d}{{\rm d}x_\a}\bigg)^2 -q_2.
\end{gather*}
\item $r=3$:
\begin{gather*}
\det(y+\phi(\bq)) = y^3-4q_2y+4q_3,
\\
P_\a(x_\a,\hbar; \bq)=
\bigg(\hbar\frac{\rm d}{{\rm d}x_\a}\bigg)^3 -4q_2\bigg(\hbar
\frac{\rm d}{{\rm d}x_\a}\bigg)+4q_3-2\hbar q_2'.
\end{gather*}

\item $r=4$:
\begin{gather*}
\det(y+\phi(\bq)) = y^4 -10q_2+24q_3y-36q_4+9q_2^2,
\\
P_\a(x_\a,\hbar; \bq)=
\bigg(\hbar\frac{\rm d}{{\rm d}x_\a}\bigg)^4 -10q_2\bigg(\hbar
\frac{\rm d}{{\rm d}x_\a}\bigg)^2+(24q_3-10\hbar q_2')
\bigg(\hbar\frac{\rm d}{{\rm d}x_\a}\bigg)
\\ \hphantom{P_\a(x_\a,\hbar; \bq)=}
{}-36q_4+9q_2^2 +3\hbar^2q_2''-12\hbar q_3'.
\end{gather*}
\end{itemize}
\end{ex}

\subsection{Non-Abelian Hodge correspondence and a conjecture of Gaiotto}\label{sub:Gaiotto}
The biholomorphic map~\eqref{biholomorphic quantization}
is defined by fixing a projective structure of
the base curve $C$. Gaiotto~\cite{G}
conjectured that such a correspondence would be
canonically constructed through a
\emph{scaling limit} of non-Abelian Hodge
correspondence. The~conjecture has been solved in~\cite{DFKMMN} for~Hitchin moduli spaces
$\cM_{\Dol}$ and $\cM_{\deR}$
constructed over an arbitrary
complex simple and simply connected
Lie group $G$.
In this subsection, we review the main result of~\cite{DFKMMN}
for $G={\rm SL}(r,\bC)$ and
compare it with our quantization.

The setting of this subsection is the following. The~base curve $C$ is a compact Riemann surface
of genus $g\ge 2$. We denote by $E^{\text{top}}$
the topologically trivial complex
vector bundle of
rank $r$ on~$C$. The~prototype of the correspondence
between stable holomorphic vector bundles
on $C$ and differential geometric data goes
back to Narasimhan--Seshadri~\cite{NS}
(see also~\cite{AB, MFK}). Extending
the classical case,
the stability
condition for an ${\rm SL}(r,\bC)$-Higgs
bundle $(E,\phi)$ translates into a~differential geometric condition, known as
\emph{Hitchin's equations}, imposed on
a set of geometric data as
follows~\cite{Donaldson, H1, Simpson}. The~data consist of a Hermitian fiber metric $h$
on $E^{\rm top}$, a unitary connection
$\nabla$ in $E^{\rm top}$
with respect to~$h$, and a differentiable
$\mathfrak{sl}(r,\bC)$-valued
$1$-form $\phi$ on $C$. The~following system of nonlinear equations is
called Hitchin's equations:
\begin{gather}
\label{Hitchin's}
\begin{cases}
F_\nabla + \big[\phi, \phi^\dagger\big] = 0,\\
\nabla^{0.1}\phi = 0.
\end{cases}
\end{gather}
Here, $F_\nabla$ is the curvature of $\nabla$,
$\phi^\dagger$ is the Hermitian conjugate
of $\phi$ with respect to~$h$, and $\nabla^{0,1}$ is the Cauchy--Riemann
part of $\nabla$. $\nabla^{0,1}$ gives rise to a
natural complex structure in $E^{\rm top}$,
which we simply denote by $E$. Then
$\phi$ becomes a holomorphic Higgs field in $E$
because of the second equation of~\eqref{Hitchin's}. The~stability condition
for the Higgs pair $(E,\phi)$ is equivalent to
Hitchin's equations~\eqref{Hitchin's}. Define a one-parameter family
of connections
\begin{gather}
\label{twister}
\nabla(\zeta) := \frac{1}{\zeta}\cdot \phi + \nabla + \zeta \cdot \phi^\dagger,\qquad
\zeta\in \bC^*.
\end{gather}
Then the flatness of $\nabla(\zeta)$ for all $\zeta$
is equivalent to~\eqref{Hitchin's}.

The \emph{non-Abelian Hodge correspondence}
\cite{Donaldson, H1, Mochizuki, Simpson} (see also \cite{BB,B2012,W2008,W2008-2})
is the following diffeomorphic correspondence
\begin{gather*}
\nu\colon \ \cM_{\Dol} \owns (E,\phi)
\longmapsto \big(\widetilde{E},\widetilde{\nabla}\big)\in \cM_{\deR}.
\end{gather*}
First, we construct the solution $(\nabla,\phi,h)$ of
Hitchin's equations corresponding to
stable $(E,\phi)$. It~induces a family of flat
connections $\nabla(\zeta)$.
Then define a complex structure $\widetilde{E}$
in $E^{\rm top}$ by~\mbox{$\nabla(\zeta=1)^{0,1}$}.
Since $\nabla(\zeta)$ is flat, $\widetilde{\nabla}:=
\nabla(\zeta=1)^{1,0}$ is automatically a holomorphic
connection in~$\widetilde{E}$.
Thus $(\widetilde{E},\widetilde{\nabla})$
becomes a holomorphic connection. Stability
of $(E,\phi)$ implies
that the resul\-ting connection is irreducible, hence
$(\widetilde{E},\widetilde{\nabla})\in \cM_{\deR}$.
Since this correspondence goes through the
real unitary connection $\nabla$, the assignment
$E\longmapsto \widetilde{E}$ is not a
holomorphic deformation of~vector
bundles.

Extending the idea of the one-parameter
family~\eqref{twister}, Gaiotto
conjectured the following:

\begin{conj}[Gaiotto~\cite{G}]
\label{conj:Gaiotto}
Let $(\nabla, \phi, h)$ be the solution
of~\eqref{Hitchin's} corresponding
to a~stable Higgs bundle
$\left(E_0,\phi(\bq)\right)$ on the
${\rm SL}(r,\bC)$-Hitchin
section~\eqref{Hitchin section}.
Consider
the following two-pa\-ra\-meter family of
connections
\begin{gather*}
\nabla(\zeta,R) := \frac{1}{\zeta}\cdot R\phi + \nabla + \zeta \cdot R\phi^\dagger,
\qquad
\zeta\in \bC^*, \qquad
R\in \bR_+.
\end{gather*}
Then the scaling limit
\begin{gather*}
\lim_{\substack{R\rar 0,\,\zeta\rar 0\\
\zeta/R = \hbar}} \nabla(\zeta,R)
\end{gather*}
exists for every $\hbar\in \bC^*$, and forms
an $\hbar$-family of ${\rm SL}(r,\bC)$-opers.
\end{conj}

\begin{rem}
The existence of the limit is non-trivial, because
the Hermitian metric $h$ blows up as $R\rar 0$.
\end{rem}

\begin{rem}
Unlike the case of non-Abelian Hodge
correspondence, the Gaiotto limit works
only for a point in the Hitchin section.
\end{rem}

\begin{thm}[\cite{DFKMMN}]
Gaiotto's conjecture holds for an arbitrary
simple and simply connected complex
algebraic group $G$.
\end{thm}

The universal covering of a
 compact Riemann surface $C$ is the upper-half plane
 $\bH$. The~fundamental group $\pi_1(C)$ acts
 on $\bH$ through a representation
 \begin{gather*}
\rho\colon\ \pi_1(C)\lrar {\rm PSL}(2,\bR) = \Aut(\bH),
 \end{gather*}
 and generates an analytic isomorphism
 \begin{gather*}
 C \isom \bH\big/\rho\big(\pi_1(C)\big).
 \end{gather*}
 The representation $\rho$ lifts to~${\rm SL}(2,\bR)
 \subset {\rm SL}(2,\bC)$, and induces a projective
 structure in $C$ subordinating its complex
 structure coming from $\bH$.
 This projective structure is what we call the
 \emph{Fuchsian} projective structure.

\begin{cor}[Gaiotto correspondence and
quantization]
Under the same setting of
Conjecture~$\ref{conj:Gaiotto}$, the limit
oper of~{\rm \cite{DFKMMN}} is given by
\begin{gather}
\label{Gaiotto limit}
\lim_{\substack{R\rar 0,\,\zeta\rar 0\\
\zeta/R = \hbar}} \nabla(\zeta,R)
= {\rm d}+\frac{1}{\hbar}\phi(\bq) = \nabla^\hbar(\bq),\qquad
\hbar\ne 0,
\end{gather}
with respect to the Fuchsian projective coordinate
system. The~operator~\eqref{Gaiotto limit}
is a connection in the $\hbar$-filtered extension
$(F_\hbar^\bullet, E_\hbar)$ of
Definition~$\ref{def:filtered extension}$.
In particular, the correspondence
\begin{gather*}
(E_0,\phi(\bq))\overset{\gam}{\longmapsto}
\big(E_\hbar, \nabla^\hbar(\bq)\big)
\end{gather*}
is biholomorphic, unlike the non-Abelian Hodge
correspondence.
\end{cor}

\begin{proof}
The key point is that since $E_0$ is made out
of $K_C$, the fiber metric $h$ naturally comes from
the metric of~$C$ itself. Hitchin's equations
\eqref{Hitchin's} for $\bq=0$ then become a harmonic
equation for the metric of~$C$, and its solution
is given by the constant curvature hyperbolic
metric. This metric in turn defines the
Fuchsian projective structure in $C$. For~more
detail, we~reefer to~\cite{O-Paris, DFKMMN}.
\end{proof}


\section{Geometry of singular spectral curves}
\label{sect:spectral}

The quantization mechanism of
Section~\ref{sect:opers} applies to
all Hitchin spectral curves, and it is not
sensitive to whether the spectral curve is smooth or
not.
However, the local quantization mechanism
of~\cite{OM1,OM2} using the
PDE version of topological recursion
requires a \emph{non-singular} spectral curve. The~goal of this section is to
review the systematic
construction of the non-singular
models of singular Hitchin spectral curves
of~\cite{OM2}.
Then in Section~\ref{sect:WKB}, we
prove that for the case of
meromorphic ${\rm SL}(2,\bC)$-Higgs bundles,
the PDE recursion of topological type
based on the non-singular model
of this section provides
WKB analysis for the quantum curves
constructed in Section~\ref{sect:opers}.

Since the quantum curve reflects the
geometry of $\Sigma\subset \overline{T^*C}$
through semi-classical limit,
we~first need
 to identify
the choice of the blow-up space
$\operatorname{Bl}(\overline{T^*C})$
in which the non-singular model~$\widetilde{\Sigma}$ of the spectral curve
is realized as a smooth divisor.
This geometric information determines part of
the initial data for topological recursion,
i.e., the spectral curve of Eynard--Orantin,
and the differential form $W_{0,1}$.
Geometry of spectral curve also
gives us information of~singularities of the quantization. For~example,
when we have a component of
a spectral curve tangent to the divisor $C_\infty$,
the quantum curve has an irregular singular point,
and the class of the irregularity is determined by
the degree of tangency. We have
given a classification of~sin\-gu\-la\-rities of the quantum curves
in terms of the geometry of spectral curves
in~\cite{OM2}.

In what follows, we give the construction of the
canonical blow-up space $\operatorname{Bl}(\overline{T^*C})$,
and determine the genus of
the normalization $\widetilde{\Sigma}$.
This genus is necessary to identify the
Riemann prime form on it, which determines
another input datum $W_{0,2}$ for the
topological recursion.

There are two different ways of defining
the spectral curve for Higgs bundles with
meromorphic Higgs field. Our definition
uses the compactified
cotangent bundle. This idea also appears in~\cite{KS2013}. The~traditional definition,
which assumes the pole divisor $D$ of the
Higgs field to be redu\-ced,
is suitable for the study of moduli spaces
of parabolic Higgs bundles.
When we deal with non-reduced effective divisors,
parabolic structures do not play any role. Non-reduced
divisors appear naturally when we deal with
classical equations such as the Airy differential
equation, which has an irregular singular point
of class $\frac{3}{2}$
at $\infty\in \bP^1$.

Our point of view of spectral curves is also closely
related
to considering the \emph{stable pairs} of~pure
dimension $1$ on $\overline{T^*C}$. Through
Hitchin's abelianization idea, the moduli space of
stable pairs and the moduli space of Higgs bundles
are identified~\cite{HL}.

The geometric setting we start with is
a meromorphic ${\rm SL}(2,\bC)$-Higgs bundle
$\big(E_0,\phi(q)\big)$ defi\-ned on
a smooth projective algebraic curve $C$
of genus $g\ge 0$
with a fixed projective structure. Here,
\begin{gather*}
q=-\det (\phi(q)) \in H^0\big( C,K_C(D)^{\tensor 2}\big), \qquad
 \phi(q) = \begin{bmatrix}
 &q\\1
 \end{bmatrix}\!,
 \end{gather*}
and $D$ is an effective divisor of~$C$. The~spectral curve
is the zero-locus in
$\overline{T^*C}$ of the characteristic
equation
\begin{gather}
\label{SL2 spectral curve}
\Sigma :=\big(\eta^2- \pi^*(q)\big)_0.
\end{gather}
The only condition we impose here is that
\emph{the spectral curve is irreducible.}
In the language of~Higgs bundles, this condition
corresponds to the stability of $\big(E_0,\phi(q)\big)$.

Recall that $\Pic(\overline{T^*C})$ is
generated by the zero section $C_0$ of
$T^*C$ and fibers of the projection map
$\pi\colon \overline{T^*C}\lrar C$. Since the spectral
curve $\Sigma$ is a double covering of~$C$,
as a divisor it is expressed as
\begin{gather}
\label{Sigma in Pic}
\Sigma = 2C_0+\sum_{j=1}^a \pi^{*}(p_j)\in
\Pic(\overline{T^*C}),
\end{gather}
where
$\sum_{j=1}^a p_j\in\Pic^a(C)$
is a divisor on $C$ of degree $a$.
As an element of the N\'eron--Severi group
\begin{gather*}
\NS(\overline{T^*C}) =
\Pic(\overline{T^*C})/\Pic^0(\overline{T^*C}),
\end{gather*}
it is simply
\begin{gather*}
\Sigma = 2C_0+aF \in \NS(\overline{T^*C})
\end{gather*}
for a typical fiber class $F$. Since the intersection
$F\cdot C_\infty =1$, we have
$a = \Sigma \cdot C_\infty$ in $\NS(\overline{T^*C})$.
From the genus formula
\begin{gather*}
p_a(\Sigma) = \half \Sigma\cdot
(\Sigma +K_{\overline{T^*C}}) +1
\end{gather*}
and
\begin{gather*}
K_{\overline{T^*C}} = -2C_0 + (4g-4)F
\in \NS(\overline{T^*C}),
\end{gather*}
we find that the arithmetic genus of the spectral
curve $\Sigma$ is
\begin{gather}
\label{pa}
p_a(\Sigma) = 4g-3+a,
\end{gather}
where $a$ is the number of intersections
of $\Sigma$ and $C_\infty$.
Now we wish to find the geometric genus
of $\Sigma$.

Recall the following from~\cite{OM2}:

 \begin{Def}[discriminant divisor] The discriminant divisor of
the spectral curve~\eqref{SL2 spectral curve}
 is a~divisor on $C$ defined by
\begin{gather}
\label{discriminant}
\Delta:=(q)_0-(q)_\infty,
\end{gather}
where
\begin{gather}
\label{q 0}
(q)_0= \sum_{i=1}^m m_ir_i , \qquad m_i>0,\qquad q_i\in C,
\\
\label{q infinity}
(q)_\infty= \sum_{j=1}^n n_jp_j, \qquad n_j>0,\qquad p_j\in C.
\end{gather}
\end{Def}

Since $q$ is a meromorphic section of $K_C^{\tensor 2}$,
\begin{gather}\label{deg Delta}
\deg \Delta =\sum_{i=1}^m m_i - \sum_{j=1}^n n_j = 4g-4.
\end{gather}

\begin{thm}[geometric genus formula]\label{thm:geometric genus formula}Define
\begin{gather}\label{delta}
\delta=|\{i\mid m_i \equiv 1 \!\!\mod 2\}|+|\{j\mid n_j \equiv 1\!\! \mod 2\}|.
\end{gather}
Then
the geometric genus of the
spectral curve $\Sigma$
of~\eqref{SL2 spectral curve}
is given by
\begin{gather}\label{pg}
 p_g(\Sigma) = 2g-1+\half \delta.
\end{gather}
We note that~\eqref{deg Delta}
 implies $\delta \equiv 0\mod 2$.
\end{thm}

\begin{rem}
If $\phi$ is a holomorphic Higgs field,
then $m= \delta = 4g-4$ and $n=0$.
Therefore, we recover the genus formula $g(\Sigma) = 4g-3$
of~\cite[equation~(2.5)]{OM1}.
In this case, the
Hitchin fibration~\eqref{Hitchin fibration} is a family of \emph{Prym} varieties,
which are $(3g-3)$-dimensional Abelian varieties
associated with the ramified covering $\pi\colon \Sigma\lrar C$.
\end{rem}

\begin{proof}
Since $\Sigma\subset
\overline{T^*C}$ is a double covering
of~$C$ in a ruled surface, locally at
every singular point $p$,
$\Sigma$ is either irreducible, or
reducible and consisting of two components.
When irreducible, it is locally isomorphic
to
\begin{gather}
\label{cusp}
t^2 - s^{2m+1} = 0, \qquad m\ge 1.
\end{gather}
If it has two components, then it is locally isomorphic to
\begin{gather}
\label{even}
t^2 -s^{2m} =(t-s^m)(t+s^m)=0.
\end{gather}
Note that the local form of
$\Sigma$ at a ramification point of
$\pi\colon \Sigma\lrar C$
 is written as~\eqref{cusp} with $m=0$.
By extending the terminology ``singularity''
to ``critical points'' of the morphism $\pi$,
we include a ramification point as a cusp
with $m=0$.

Let ${\nu}\colon \TSig\lrar \Sigma$
be the non-singular model of $\Sigma$.
 Then $ \tpi=\pi\circ{\nu}\colon\TSig \lrar C $
 is a double sheeted covering of~$C$ by a
 smooth curve $\TSig$. If
 $\Sigma$ has two components at a
 singularity $P$
 as in~\eqref{even}, then~$\tpi^{-1}(P)$
 consists of two points and $\tpi$ is
 not ramified there. If $P$ is a cusp~\eqref{cusp}, then
 $\tpi^{-1}(P)$ is a ramification point
 of the covering $\tpi$.
 Thus the invariant $\delta$ of~\eqref{delta} counts the total number of~cusp singularities of $\Sigma$
 and the ramification points of
 $\pi\colon \Sigma\lrar C$.
 Then the
 Riemann--Hurwitz formula
 gives us
 \begin{gather*}
 2-2g\big(\TSig\big) -\delta
 = 2\left( 2-2g(C)-\delta\right),
 \end{gather*}
hence
\begin{gather}
\label{pg 2}
p_g(\Sigma) = 2 g(C) -1+\half \qquad (\text{the number of cusps}).
\end{gather}
The genus formula~\eqref{pg} follows
from~\eqref{pg 2}.
\end{proof}

Our purpose is to apply topological recursion
of Section~\ref{sect:TR} to a singular spectral curve of~the form
$\Sigma$ of~\eqref{SL2 spectral curve}. To this end, we need to
construct in a canonical way the normalization
$\nu\colon\TSig\lrar \Sigma$ through a sequence of
blow-ups of the ambient space $\overline{T^*C}$.
This is because we need to construct
differential forms on $\TSig$ that reflect
 geometry of $i\colon\Sigma\lrar \overline{T^*C}$.
We thus proceed to analyze the local
structure of $\Sigma$ at each
singularity using the global
equation~\eqref{SL2 spectral curve} in what follows.

\begin{Def}[construction of the blow-up space]
\label{def:Bl}
The blow-up space $\operatorname{Bl}(\overline{T^*C})$
\begin{gather}
 \label{blow-up}
\begin{split}
&\xymatrix{
\widetilde{\Sigma}
\ar[dd]_{\tilde{\pi}}
 \ar[rr]^{\tilde{i}}\ar[dr]^{\nu}&&\operatorname{Bl}(\overline{T^*C})
\ar[dr]^{\nu}
\\
&\Sigma \ar[dl]_{\pi}\ar[rr]^{i} &&
\overline{T^*C}\ \ar[dlll]^{\pi}
\\
C 		}
\end{split}
\end{gather}
is defined by blowing up
$\overline{T^*C}$ in the following way:
{\samepage\begin{itemize}\itemsep=0pt
\item At each $r_i$ of~\eqref{q 0}, blow up
 $r_i\in \Sigma\cap C_0\subset
 \overline{T^*C}$
a total of $\big\lfloor \frac{m_i}{2}\big\rfloor$ times.
\item At each $p_j$ of~\eqref{q infinity}, blow up
at the intersection $\Sigma\cap \pi^{-1}(p_j)
\subset C_\infty$
a total of
$\big\lfloor \frac{n_j}{2}\big\rfloor$ times.
\end{itemize}

}
\end{Def}

 \begin{rem} Let us consider when both $(q)_0$ and $(q)_\infty$ are
 reduced. From the definition above, in this case
 $\Sigma$ is non-singular, and the two
 genera~\eqref{pa} and~\eqref{pg} agree. The~spectral curve is invariant under the involution
$\sigma\colon \overline{T^*C}\lrar \overline{T^*C}$
of~\eqref{involution}. If $q\in H^0(C,K_C^{\tensor 2})$
is holomorphic, then
 $\pi\colon \Sigma \lrar C$ is simply branched over
$\Delta = (q)_0$, and
$\Sigma$ is a smooth curve of genus
$4g-3$. This is in agreement of~\eqref{pa}
because $n=0$ in this case.
If $q$ is meromorphic, then
its pole divisor is given by
$(q)_\infty $ of degree $n$.
Since $(q)_\infty$ is reduced,
$\pi\colon \Sigma \lrar C$ is ramified at
the intersection of $C_\infty$ and
$\pi^{*}(q)_\infty$. The~spectral curve is also ramified at its intersection
with $C_0$.
Note that $\deg(q)_0 = 4g-4+n$ because
of~\eqref{deg Delta}.
Thus $\pi\colon\Sigma \lrar C$ is simply ramified
at a total of $4g-4+2n$ points. Therefore,
$\Sigma$ is non-singular, and we deduce
that its genus is given by
\begin{gather*}
p_g(\Sigma) = p_a(\Sigma) = 4g-3+n
\end{gather*}
from the Riemann--Hurwitz formula.
As a divisor class, we have
\begin{gather*}
\Sigma = 2C_0+\pi^*(q)_\infty\in
\Pic(\overline{T^*C}),
\end{gather*}
in agreement of~\eqref{Sigma in Pic}.
\end{rem}

\begin{thm}
In the blow-up space
$\operatorname{Bl}(\overline{T^*C})$, we have the following.
\begin{itemize}\itemsep=0pt
\item
The proper transform
$\widetilde{\Sigma}$
of the spectral curve
$\Sigma \subset \overline{T^*C}$
by the birational morphism
$\nu \colon $ $\operatorname{Bl}(\overline{T^*C})\lrar
\overline{T^*C}$ is a smooth curve with
a holomorphic map
$\tilde{\pi}=\pi\circ\nu\colon\widetilde{\Sigma} \lrar C$.

\item
The Galois action
 $\sigma\colon \Sigma\lrar \Sigma$
lifts to an involution of $\widetilde{\Sigma}$,
and the morphism
$\tilde{\pi}\colon\widetilde{\Sigma}\lrar C$
is a Galois covering with the Galois group
$\Gal(\widetilde{\Sigma}/C) =\la \tilde{\sigma}
\ra \isom \bZ/2\bZ$
\begin{gather}
\label{sigma-tilde}
\begin{CD}
\widetilde{\Sigma}@>{\nu}>>\Sigma @>{\pi}>> C
\\
@V{\tilde{\sigma}}VV @VV{\sigma}V @|
\\
\widetilde{\Sigma}@>>{\nu}> \Sigma
@>>{\pi}> C.
\end{CD}
\end{gather}
\end{itemize}
\end{thm}

 \begin{proof}
We need to consider
only when $\Delta$ is non-reduced.
Let $r_i\in \supp(\Delta)$
be a zero of $q$ of~degree
 $m_i>1$. The~curve
germ of $\Sigma$ near $r_i\in \Sigma\cap C_0$
 is given by a formula $y^2 = x^{m_i}$,
where $x$ is the base coordinate
on $C$ and $y$ a fiber coordinate.
We blow up once at $(x,y) = (0,0)$, using a local
parameter $y_1 = y/x$ on the exceptional
divisor. The~proper transform of the curve
germ becomes $y_1^2 = x^{m_i-2}$.
Repeat this process at $(x,y_1) = (0,0)$, until
we reach the equation
\begin{gather*}
y_\ell ^2 = x^{\epsilon},
\end{gather*}
where $\epsilon = 0$ or $1$. The~proper transform
of the curve germ is now non-singular.
We see that after a sequence of
$\big\lfloor \frac{m_i}{2}\big\rfloor$ blow-ups starting at
the point $r_i$,
the proper transform of $\Sigma$ is simply
ramified over $r_i\in C=C_0$ if $m_i$ is odd, and
unramified if $m_i$ is even.
We apply the same sequence of blow-ups at
each $r_i$ with multiplicity greater than $1$.

Let $p_j\in \supp(\Delta)$ be a pole of
$q$ with order
$n_j>1$. The~intersection $P = \Sigma\cap
\pi^{-1}(p_j)$ lies on~$C_\infty$,
and $\Sigma$ has a singularity at $P$. Let
$z=1/y$ be a fiber coordinate of $\pi^{-1}(p_j)$
at the infinity. Then the curve germ of $\Sigma$
at the point $P$ is given by
\begin{gather*}
z^2 = x^{n_j}.
\end{gather*}
The involution $\sigma$ in this coordinate is
simply $z\longmapsto -z$. The~blow-up process we
apply at $P$ is the same as before. After
$\lfloor \frac{n_j}{2}\rfloor$ blow-ups starting at
the point $P\in \Sigma \cap \pi^{-1}(p_j)$,
the proper transform of $\Sigma$ is simply
ramified over $p_j\in C$ if $n_j$ is odd, and
unramified if $n_j$ is even. Again we do this
process for all $p_j$ with a higher multiplicity.

The blow-up space $\operatorname{Bl}(\overline{T^*C})$
is defined as
the application of a total of
\begin{gather*}
\sum_{i=1}^m \left\lfloor \frac{m_i}{2}\right\rfloor
+\sum_{j=1}^n \left\lfloor \frac{n_j}{2}\right\rfloor
\end{gather*}
times blow-ups on $\overline{T^*C}$
as described above. The~proper transform $\TSig$ is
the minimal resolution of $\Sigma$.
Note that the morphism
\begin{gather*}
\tpi = \pi\circ\nu\colon\ \TSig\lrar C
\end{gather*}
is a double covering, ramified exactly at
$\delta$ points. Since $p_a(\TSig)
= p_g(\Sigma)$,~\eqref{pg} follows from
the Riemann--Hurwitz formula applied to~${\tpi}\colon\TSig\lrar C$. It~is also obvious that $\delta$ counts the
number of cusp points of $\Sigma$, including
smooth ramification points
of $\pi$, in agreement
 of
Theorem~\ref{thm:geometric genus formula},
and the fact that $\delta$ counts the total number
of odd cusps on $\Sigma$.

Note that $C_0$ and $C_\infty$ are point-wise
invariant under the involution $\sigma$.
Since $\operatorname{Bl}(\overline{T^*C})$ is constructed
by blowing up points on $C_0$ and $C_\infty$
and their proper transforms,
we have a natural lift
$\tsig\colon \operatorname{Bl}(\overline{T^*C})
\lrar \operatorname{Bl}(\overline{T^*C})$
of $\sigma$ which induces~\eqref{sigma-tilde}.
\end{proof}

\section{A differential version of
topological recursion}
\label{sect:TR}

The Airy example of Section~\ref{sect:Airy}
suggests that the asymptotic expansion
of a solution to a given quantum curve
at its
singularity
contains information of quantum invariants. It~also suggests that the topological recursion
of~\cite{EO1} provides an effective tool
for calculating asymptotic expansions of solutions
for quantum curves. Since a linear differential equation is
characterized by its solutions, topological recursion
can be used as a mechanism
of defining the quantization process
from a spectral curve to a quantum curve.
Then a natural question arises:

\begin{quest}
How are the
two quantizations, one with the construction
of an $\hbar$-family of~opers, and the other
via topological recursion, related?
\end{quest}

In Section~\ref{sect:WKB}, we prove that
topological recursion provides WKB analysis
of the quantum curves constructed through
$\hbar$-families of opers, for the case of
holomorphic and meromorphic
${\rm SL}(2,\bC)$-Higgs bundles. For~this purpose,
in this section
we review the framework of PDE recursion
developed in~\cite{OM1, OM2}. For~the case of singular Hitchin spectral curves,
our particular method of normalization
 of spectral curves of
 Section~\ref{sect:spectral}
 produces the same result of
 quantization of Section~\ref{sect:opers}.

If we consider a family of spectral curves
that degenerate to a singular curve, the necessity of
normalization for WKB analysis may sound
unnatural. We emphasize that the semi-classical
limit of the quantum curve thus obtained remains
the original \emph{singular} spectral curve,
not the normalization,
consistent with~\eqref{family SCL}. Thus our quantization
procedure in terms of PDE recursion is
also a natural process.

Although many aspects of our current
framework can be generalized to arbitrary
complex
simple Lie groups, since our calculation
mechanism of
Section~\ref{sect:WKB}
has been developed so far only for
the ${\rm SL}(2,\bC)$ case, we restrict our
attention to this case in this section.

We start with topological recursion
for a degree $2$ covering, not necessarily restricted
to Hitchin spectral curves. The~key ingredient of the theory is the
\textit{Riemann prime form}
$E(z_1,z_2)$ on the product
$\widetilde{\Sigma}\times \widetilde{\Sigma}$
of a compact Riemann surface $\widetilde{\Sigma}$
with values in a certain line bundle~\cite{Mumford}.
 To define the prime form, we have to make a few more
 extra choices. First, we need to choose
 a theta characteristic $K_{\widehat{\Sigma}}^\half$ such that
 \begin{gather*}
 \dim H^0\Big(\widetilde{\Sigma}, K_{\widetilde{\Sigma}}^\half\Big) = 1.
 \end{gather*}
We also need to choose a symplectic basis for
$H_1(\widetilde{\Sigma}, \bZ)$, usually
referred to as $A$-cycles and $B$-cycles.
We follow Mumford's convention and use
 the unique Riemann prime form as defined in
\cite[p.~3.210]{Mumford}.
See also \cite[Section 2]{OM1}.

\begin{Def}[topological recursion for non-singular covering]~\label{def:non-singular}
Let $C$ be a non-singular projective algebraic
curve together with a choice of a
symplectic basis for $H^1(C,\bZ)$,
 and \mbox{$\tilde{\pi}\colon\widetilde{\Sigma}\lrar C$} a
degree $2$ covering by another non-singular
curve $\widetilde{\Sigma}$. We denote by
$R$ the
ramification divisor of $\tilde{\pi}$. The~covering
$\tilde{\pi}$ is a Galois covering with
the Galois group
$\bZ/2\bZ = \la \tilde{\sigma} \ra$,
and $R$ is the fixed-point divisor of the involution
$\tilde{\sigma}$. We also choose a labeling of
points of $R$, and define the $A$-cycles
of a symplectic bases for $H_1(\widetilde{\Sigma},\bZ)$
as defined in \cite[Section 2]{OM1},
which extend the $A$-cycles of~$C$.
 \textit{Topological recursion}
is an inductive
mechanism of constructing meromorphic
differential forms $W_{g,n}$ on the Hilbert scheme
$\widetilde{\Sigma}^{[n]}$
of $n$-points on $\widetilde{\Sigma}$
for all $g\ge 0$ and
$n\ge 1$ in~the \emph{stable range}
$2g-2+n>0$, from given initial data $W_{0,1}$
and $W_{0,2}$. The~differential
form $W_{g,n}$ is a meromorphic $n$-linear
form, i.e., a $1$-form on each factor
of $\widetilde{\Sigma}^{[n]}$ for
$2g-2+n>0$.

\begin{itemize}\itemsep=0pt
\item $W_{0,1}$ is a meromorphic $1$-form
on $\widetilde{\Sigma}$ to be prescribed
according to the geometric setting we have.
Then define
\begin{gather*}
\Omega := \tilde{\sigma}^*W_{0,1}-W_{0,1},
\end{gather*}
which satisfies $\tilde{\sigma}^*\Omega = -\Omega$.

\item $W_{0,2}$ is defined by
\begin{gather*}
W_{0,2}(z_1,z_2) ={\rm d}_1 {\rm d}_2 \log E(z_1,z_2),
\end{gather*}
where $E(z_1,z_2)$ is the $A$-cycle
 normalized Riemann prime form on
 $\widetilde{\Sigma}\times \widetilde{\Sigma}$.
 $W_{0,2}$ is a~mero\-mor\-phic differential $1\tensor1$-form
 on $\TSig\times\TSig$ with $2$nd order poles along the
 diagonal.

 \item We also define the normalized
Cauchy kernel on $\widetilde{\Sigma}$ by
\begin{gather}
\label{omega a-b}
\omega^{a-b}(z): = {\rm d}_z\log
\frac{E(a,z)}{E(b,z)},
\end{gather}
which is a meromorphic
$1$-form in $z$ with simple poles at $z=a$ of residue $1$
and at $z=b$ of residue $-1$. We note that
the ratio $E(a,z)/E(b,z)$ is a meromorphic
function on the universal covering in $a$ or $b$,
but \emph{not} a meromorphic
function on $\TSig$.
We thus choose a~fun\-da\-men\-tal domain
of the universal covering and restrict $a$ and $b$
in that domain for local
calculations.

\end{itemize}

The inductive formula of the topological
recursion then takes the following shape
for $2g-2+n >0$.
\begin{gather}
W_{g,n}(z_1,\dots,z_n)=\half \frac{1}{2\pi \sqrt{-1}}
\oint_{\Gam}\frac{\omega^{\tilde{z}-z}(z_1)}{\Omega(z)}\nonumber
\\ \qquad
{}\times \Bigg[W_{g-1,n+1}(z,\tilde{z}, z_2,\dots,z_n)
+\sum_{\substack{g_1+g_2=g\\I\sqcup J=\{2,\dots,n\}}}^{{\rm No}(0,1)}
W_{g_1,|I|+1}(z,z_I)W_{g_2,|J|+1}(\tilde{z},z_J)\Bigg].
\label{integral TR}
\end{gather}
Here,
\begin{itemize}\itemsep=0pt

\item The integration contour $\Gam\subset
\TSig$
 is a collection of positively oriented small
loops around each
 point $p\in \supp (\Omega)_0\cup \supp(R)$.

\item The integration is taken with respect to~$z\in \Gam$. Thus we chose a fundamental
 domain of~the universal covering of $\TSig$
 that contains $\supp (\Omega)_0\cup \supp(R)$,
 and perform the integration locally as residue
 calculations.

\item $\tilde{z} = \tilde{\sigma}(z)$ is the Galois
conjugate of $z\in \widetilde{\Sigma}$.

\item The expression
$1/\Omega$ is a meromorphic
section of $K_C^{-1}$ which is multiplied
to meromorphic quadratic differentials
on $\widetilde{\Sigma}$ in $z$-variable.

\item ``No $(0,1)$'' means that
$g_1=0$ and $I=\varnothing$, or $g_2=0$ and
$J=\varnothing$, are excluded in the summation.

\item The sum runs over all partitions of $g$ and
set partitions of $\{2,\dots,n\}$, other than those
containing the $(0,1)$-geometry.

\item $|I|$ is the cardinality of the subset
$I\subset \{2,\dots,n\}$.

\item $z_I=(z_i)_{i\in I}$.
\end{itemize}
\end{Def}

\begin{rem}
The integrand of~\eqref{integral TR} is not
a well-defined differential form in $z\in \TSig$,
due to the definition of $\omega^{\tsig(z)-z}(z_1)$
mentioned above. What the formula defines
is a sum of residues at each
point $p\in \supp (\Omega)_0\cup \supp(R)$,
which depends on the choice of the
domain of $\omega^{\tsig(z)-z}(z_1)$ as a function
in $z$.
\end{rem}

\begin{rem}
Topological recursion depends on the
choice of the integration contour $\Gam$.
Since the integrand of the right-hand side of
\eqref{integral TR} has other poles
than the ramification divisor $R$ and
zeros of $\Omega$, other choices of
$\Gam$ are equally possible.
\end{rem}

{\sloppy\begin{rem}
Topological recursion~\eqref{integral TR} can be defined
for far more general situations. The~bottle neck of the formalism is
difficulty of integration over a high genus non-hyperelliptic
Riemann surface.
So the actual calculations have not been
done much beyond the cases when the spectral curve
$\TSig$ is of genus $0$, or hyperelliptic.
\end{rem}

}

When we have a non-singular Hitchin
spectral curve $i\colon \Sigma \hookrightarrow T^*C$
associated with a holomorphic ${\rm SL}(2,\bC)$-Higgs bundle $(E,\phi)$, we apply
Definition~\ref{def:non-singular}
to~$\widetilde{\Sigma} = \Sigma$,
$\tilde{\sigma} = \sigma$, and
$W_{0,1} = i^*\eta$,
where $\sigma$ is the involution of~\eqref{involution} and
$\eta$ is the tautological $1$-form
on $T^*C$.

Under the same setting as in
 topological recursion,
we define

\begin{Def}[PDE recursion for a smooth covering of degree $2$]
\textit{PDE recursion} is the
following partial differential equation
for all $(g,n)$ subject to~$2g-2+n\ge 2$
defined on an open neighborhood
$U^n$ of $\widetilde{\Sigma}^{[n]}$ (or the
universal covering, if the global treatment is necessary):
\begin{gather}
{\rm d}_1 F_{g,n}(z_1,\dots,z_n)= \sum_{j=2}^n
\bigg[\frac{\omega^{z_j-\tilde{\sigma}(z_j)}
(z_1)}{\Omega(z_1)}\,
{\rm d}_1F_{g,n-1}\big(z_{[\hat{j}]}\big)
-\frac{\omega^{z_j-\tilde{\sigma}(z_j)}
(z_1)}{\Omega(z_j)}\,
{\rm d}_jF_{g,n-1}\big(z_{[\hat{1}]}\big)\bigg]\nonumber
\\ \hphantom{{\rm d}_1 F_{g,n}(z_1,\dots,z_n)=}
{}+\frac{1}{\Omega(z_1)}{\rm d}_{u_1}{\rm d}_{u_2}
\bigg[F_{g-1,n+1}\big(u_1,u_2,z_{[\hat{1}]}\big)\nonumber
\\ \hphantom{{\rm d}_1 F_{g,n}(z_1,\dots,z_n)=}
{}+\sum_{\substack{g_1+g_2=g\\
I\sqcup J=[\hat{1}]}}^{\text{stable}}
F_{g_1,|I|+1}(u_1,z_I)F_{g_2,|J|+1}(u_2,z_J)\bigg]
\bigg|_{\substack{u_1=z_1\\u_2=z_1}}.
\label{differential TR}
\end{gather}
Here, the index subset $[\hat{j}]$ denotes
the complement of $j\in \{1,2,\dots,n\}$. The~final summation runs over all indices
subject to be in the stable range, i.e.,
$2g_1-1 + |I|>0$ and $2g_2-1+|J|>0$. The~initial data for the PDE recursion are
a function $F_{1,1}(z_1)$ on $U$ and
a symmetric function
$F_{0,3}(z_1,z_2,z_3)$ on $U^3$.
\end{Def}

\begin{rem}
\label{rem:PDE TR}
The PDE recursion of topological type~\eqref{differential TR}
is obtained from~\eqref{integral TR}
by replacing the original contour $\Gam$
about the ramification divisor
with the
Cauchy kernel
integration contour of~\cite{OM1,OM2} about
the diagonal
poles of $\omega^{\tsig(z)-z}(z_1)$,
$W_{0,2}(z,z_j)$, and $W_{0,2}\big(\tsig(z),z_j\big)$
at $z = z_j$ and~$ z = \tsig(z_j)$, $j=1, 2, \dots, n$.
This means the only difference
comes from the alternative choice of the contour of integration $\Gam$.
All other ingredients of the formula are the same.
Other than
the case of $g(\TSig) = 0$,
topological
recursion and PDE recursion are not equivalent.
In many enumerative problems
\cite{BHLM, CMS, OM3,DMSS,EMS,
MP2012,MS, MZ}, PDE recursions are established
through the Laplace transform of combinatorial
relations. These PDE recursions can then
be turned into the universal form of
topological recursion.
In these examples, the residue integral around
ramification points of the spectral curve
and Cauchy kernel integrations,
i.e., residue calculations
around the diagonals $z = z_j$, $z = \tsig(z_j)$, $j=1, 2, \dots, n$, are equivalent
by continuous deformation of the contour $\Gam$.
\end{rem}


\section{WKB analysis of quantum curves}
\label{sect:WKB}

We consider in this section
a meromorphic ${\rm SL}(2,\bC)$-Higgs bundle $(E_0,\phi(q))$
associated with a~meromorphic quadratic differential
$q\in H^0\big(C, K_C(D)^{\tensor 2}\big)$
with poles along an effective divisor~$D$
on a curve $C$ of arbitrary genus.
This includes the case of holomorphic
Higgs bundles when $g(C)\ge 2$ and
$D=\varnothing$. As before, we have a fixed
spin structure $K_C^\half$ and a~projective coordinate system on $C$.

Let $\left(E_\hbar, \nabla^\hbar(q)\right)$
be the biholomorphic quantization result of
Theorem~\ref{thm:quantization-meromorphic}. The~corresponding Rees $\cD_C$-module
is generated by a single differential
operator
\begin{gather}
\label{Pa}
P_\a(x_a,\hbar) = \bigg(\hbar\frac{\rm d}{{\rm d}x_\a}\bigg)^2-q_\a
\end{gather}
 on each
projective coordinate
neighborhood $U_\a$.

\begin{thm}[WKB analysis for ${\rm SL}(2,\bC)$-quantum curves]
\label{thm:WKB}
Let
$q\in H^0\big(C, K_C(D)^{\tensor 2}\big)$
be a~quad\-ra\-tic differential
with poles along an effective divisor $D$
on a curve $C$ of arbitrary genus.
We choose and fix a spin structure and
a projective coordinate system on $C$.
Theorem~$\ref{thm:quantization-meromorphic}$ tells us that we have a
unique Rees $\cD_C$-module
$\cE$ on $C$ as the quantization of the possibly
singular spectral curve
\begin{gather}
\label{eta2=q}
\Sigma = (\eta^2 - q)_0 \subset \overline{T^*C}.
\end{gather}
Then PDE recursion~\eqref{differential TR}
 with an appropriate choice of the initial
 data
provides an all-order WKB analysis
for the generator of the Rees $\cD_C$-module
$\cE$
on a small
 neighborhood in $C$ of each zero or a pole
 of $q$ of odd order.

More precisely, the WKB analysis is given as follows.
\begin{itemize}\itemsep=0pt

\item Take an arbitrary point $p\in
\supp(\Delta)$ in the discriminant divisor of~\eqref{discriminant} of odd degree.
Choose a small enough simply connected
coordinate
neighborhood $U_\a$ of~$C$ with a
projective coordinate $x_\a$, so that
$\tilde{\pi}^{-1}(U_\a)\subset \widetilde{\Sigma}$
in the normalization~\eqref{blow-up} is also simply connected.
Let $z$ be a local coordinate
of $\tilde{\pi}^{-1}(U_\a)$ and denote by a function
$x_\a = x_\a(z)$ the projection~$\tilde{\pi}$.

\item The formal WKB expansion
we wish to construct is a solution
to the equation
\begin{gather}
\label{Ppsi=0}
P_\a(x_a,\hbar)\psi_\a(x_\a,\hbar)
 = \bigg[\bigg(\hbar\frac{\rm d}{{\rm d}x_\a}\bigg)^2
-q_\a\bigg] \psi_\a(x_\a,\hbar) = 0
\end{gather}
of the specific form
\begin{gather}
\label{psi WKB}
\psi_\a (x_\a,\hbar) = \exp\bigg(\sum_{m=0}^\infty
\hbar^{m-1} S_m(x_\a)\bigg) =\exp F_\a(x_\a,\hbar).
\end{gather}
Here, $F_\a (x_\a,\hbar)$ is a formal
Laurent series in $\hbar$ starting with the
 power $-1$. In WKB analysis,
 this series in $\hbar$ does not converge.

\item
Equation \eqref{Ppsi=0} is equivalent to
\begin{gather}
\label{F eq}
\hbar^2 \frac{{\rm d}^2}{{\rm d}x_\a^2}F_\a +\hbar^2\frac{{\rm d}F_\a}{{\rm d}x_a}
\frac{{\rm d}F_\a}{{\rm d}x_a} -q_\a=0.
\end{gather}
We interpret the above equation as an infinite sequence of
ordinary differential equations for each
power of $\hbar$:
\begin{gather}
\label{h 0}
\hbar^0{\text{-terms}}\colon\quad (S_0'(x_a))^2 -q_\a=0,
\\
\label{h 1}
\hbar^1{\text{-terms}}\colon\quad
S_0''(x_\a)+2S_0'(x_\a)S_1'(x_\a) =0,
\\
\label{h 2}
\hbar^{2}{\text{-terms}}\colon\quad
S_1''(x_\a) + 2S_2'(x_\a)S_0'(x_\a)+(S_1'(x_\a))^2 = 0,
\\
\label{h m+1}
\hbar^{m+1}{\text{-terms}}\colon\quad
S_m''(x_\a) +\sum_{a+b=m+1}S_a'(x_\a)S_b'(x_\a)=0, \qquad m\ge 2.
\end{gather}
The symbol \ $'$ denotes the $x_\a$-derivative.

\item Solve~\eqref{h 0},~\eqref{h 1},
and~\eqref{h 2} to find $S_0(x_\a)$, $S_1(x_\a)$, and $S_2(x_\a)$.

\item Construct the normalization
$\tilde{\pi}\colon \widetilde{\Sigma}\lrar\Sigma$ as in~\eqref{blow-up},
and define
\begin{gather*}
W_{0,1} := \nu^*i^*\eta.
\end{gather*}

\item Define
\begin{gather}
\label{F11}
F_{1,1}(z_1) = -\int^{z_1}
\frac{W_{0,2}(z_1,\tilde{\sigma}(z_1))}
{\Omega(z_1)},
\end{gather}
where integration means a primitive of
the meromorphic $1$-form
$\frac{W_{0,2}(z_1,\tilde{\sigma}(z_1))}
{\Omega(z_1)}$ on $\tilde{\pi}^{-1}U_\a$.

\item Define
 \begin{gather}
 \label{F03}
 F_{0,3}(z_1,z_2,z_3) = \iiint \big({-}W(z_1,z_2,z_3)
 + 2\big(f(z_1)+f(z_2)+f(z_3)\big)\big),
 \end{gather}
 where
 \begin{gather*}
 W(z_1,z_2,z_3) = \frac{1}{\Omega(z_1)} \big(W_{0,2}(z_1,z_2)
 W_{0,2}\big(z_1,\tilde{\sigma}(z_3)\big) + W_{0,2}(z_1,z_3)
 W_{0,2}\big(z_1,\tilde{\sigma}(z_2)\big) \big)
 \\ \hphantom{ W(z_1,z_2,z_3) =}
 {}+{\rm d}_2 \bigg( \frac{\omega^{\tilde{\sigma}(z_2)-z_2}(z_1)
 W_{0,2}(z_2,\tilde{\sigma}(z_3))} {\Omega(z_2)} \bigg)
 \\ \hphantom{ W(z_1,z_2,z_3) =}
 {}+{\rm d}_3 \bigg( \frac{\omega^{\tilde{\sigma}(z_3)-z_3}(z_1)
 W_{0,2}(z_2,\tilde{\sigma}(z_3))} {\Omega(z_3)} \bigg)
 \end{gather*}
 and
 \begin{gather*}
 f(z) := \widetilde{S}_2(z) - \bigg(F_{1,1}(z) - \frac{1}{6}\iiint^z
 W(z_1,z_2,z_3)\bigg).
 \end{gather*}
Here, $\widetilde{S}_2(z)
=S_2\big(x_\a(z)\big)$ is the lift of $S_2(x_a)$ to~$\tilde{\pi}^{-1}(U_\a)\subset
 \widetilde{\Sigma}$.

\item Note that we have
\begin{gather*}
S_2(x_\a) = F_{1,1}\big(z(x_\a)\big)+\frac{1}{6}
F_{0,3}\big(z(x_\a),z(x_\a),z(x_\a)\big)
\end{gather*}
for a local section $z\colon U_\a\lrar \tilde{\pi}^{-1}(U_\a)$.

\item {\sloppy Solve PDE recursion~\eqref{differential TR}
using the initial data~\eqref{F11} and~\eqref{F03},
and determine
\mbox{$F_{g,n}(z_1,\dots,z_n)$} for $2g-2+n\ge 2$.

}

\item Define
\be
\label{Sm in Fgn}
S_m(x_\a) = \sum_{2g-2+n=m-1} \frac{1}{n!}
F_{g,n}\big(z(x_\a)\big),\qquad m\ge 3,
\ee
where $F_{g,n}\big(z(x_\a)\big)$ is the principal
specialization of $F_{g,n}(z_1,\dots,z_n)$
evaluated at a local section $z=z(x_\a)$ of
$\tilde{\pi}\colon \widetilde{\Sigma}\lrar C$
on $U_\a$.
\end{itemize}
Then~\eqref{psi WKB} gives the WKB
expansion of the solution to the generator
of the Rees $\cD_C$-module.
\end{thm}

\begin{rem}
The WKB method is to solve~\eqref{F eq}
iteratively and find $S_m(x_\a)$ for all $m\ge 0$. Here,
\eqref{h 0} is the
\textit{semi-classical limit}
of~\eqref{Ppsi=0}, and~\eqref{h 1} is the
\textit{consistency condition} we need for
solving the WKB expansion.
Since the $1$-form ${\rm d}S_0(x)$ is a local section
of $T^*C$, we identify $y=S_0'(x)$. Then~\eqref{h 0} is the local expression of the
spectral curve equation~\eqref{eta2=q}.
This expression is the same everywhere
on $C\setminus \supp(\Delta)$. We note
 $q$ is globally defined.
Therefore, we
recover the spectral curve $\Sigma$ from
the differential operator~\eqref{Pa}.
\end{rem}

\begin{rem}
$W_{1,1}:={\rm d}F_{1,1}$ and $W_{0,3}:={\rm d}_1{\rm d}_2{\rm d}_3
F_{0,3}$ are solutions of~\eqref{integral TR}
for $2g-2+n=1$ with respect to the contour
of integration along the diagonal divisors
mentioned in Remark~\ref{rem:PDE TR}.
\end{rem}

The key idea of~\cite{OM1,OM2} is the
principal specialization of symmetric functions,
which in our case means
 \emph{restriction}
of a PDE on a symmetric function to the
main diagonal of the variables. Differential forms
pull back, but PDEs do not. The~essence of Theorem~\ref{thm:WKB}
is that the principal specialization of
PDE recursion~\eqref{differential TR} is
exactly the quantum curve equation
\eqref{Ppsi=0}.

\begin{lem}[{\cite[Lemma A.1]{MS}}]
Let $f(z_1,\dots,z_n)$ be a symmetric function
in $n$ variables.
Then
\begin{gather*}
\frac{\rm d}{{\rm d}z}f(z,z,\dots,z)= n\bigg[\frac{\partial}{\partial u}f(u,z,\dots,z)
\bigg]\bigg|_{u=z},
\\
\frac{{\rm d}^2}{{\rm d}z^2}f(z,z,\dots,z)
=n\bigg[\frac{\partial^2}{\partial u ^2}
f(u,z,\dots,z)\bigg]\bigg|_{u=z}
\!\!\!\!\!\!+n(n-1)\bigg[\frac{\partial^2}{\partial u_1 \partial u_2}
f(u_1,u_2,z,\dots,z)\bigg]\bigg|_{u_1=u_2=z}\!\!.
\end{gather*}
For a function in one variable $f(z)$, we have
\begin{gather*}
\lim_{z_2\rar z_1}\big[\omega^{z_2-b}(z_1)(f(z_1)-f(z_2))\big]={\rm d}_1f(z_1),
\end{gather*}
where $\omega^{z_2-b}(z_1)$ is the $1$-form
of~\eqref{omega a-b}.
\end{lem}

\begin{proof}[Proof of Theorem~\ref{thm:WKB}]
First let $P\in \Sigma\cap C_\infty$ be an odd cusp singularity
on the fiber $\pi^{-1}(p)$ of a point
$p\in C = C_0 \subset \overline{T^*C}$. The~quadratic differential $q$ has a pole of
 odd order at $p$, and the normalization
$\tilde{\pi}\colon \widetilde{\Sigma}\lrar C$ is
simply ramified at a point $Q
\in \widetilde{\Sigma}$ over $p$.
 We choose
a local projective coordinate $x$ on $C$ centered
at $p$. The~Galois action
of $\tilde{\sigma}$ on $\widetilde{\Sigma}$
fixes $Q$. As we have shown in Case~1 of the proof of
Theorem~\ref{thm:geometric genus formula},
locally over $p$,
the spectral curve $\Sigma$ has the shape
\begin{gather*}
z_0^2 = c(x) x^{2\mu+1},
\end{gather*}
where $z_0=1/y$, $y$ is the fiber coordinate
on $T_p^*C$, and $c(x)$ is
 a unit $c(x)\in \cO_{C,p}^*$. The~quadratic differential $q$
has a local expression
\begin{gather*}
q = \frac{1}{c(x) x^{2\mu+1}}.
\end{gather*}

Define $z_1 = z_0/x$. The~proper transform
of $\widetilde{\Sigma}$ after the first
blow up at $P$ is locally written~by
\begin{gather*}
z_1^2 = c(x)x^{2\mu-1}.
\end{gather*}
Note that $z_1$ is an affine coordinate of the first
exceptional divisor.
 Repeating this process $\mu$-times, we end up with
 a coordinate $z_{\mu-1} = z_\mu x$ and
 an equation
 \begin{gather*}
 z_\mu ^2 = c(x)x.
 \end{gather*}
 Here again, $z_\mu$ is an affine coordinate of
 the exceptional divisor created by the
 $\mu$-th blow-up.
 Write $z=z_\mu$ so that the proper transform
 of the $\mu$-times blow-ups is given by
 \begin{gather}
 \label{z and x}
 z^2 = c(x)x.
 \end{gather}
 Note that the Galois action of $\tilde{\sigma}$
 at $Q$ is simply $z\longmapsto -z$.
 Solving~\eqref{z and x} as a functional equation,
 we obtain a Galois invariant local expression
 \begin{gather*}
 x = x(z) = c_Q(z^2) z^2,
 \end{gather*}
 where $c_Q\in \cO_{\widetilde{\Sigma},Q}^*$
 is a unit element. This function
 is precisely the local expression of the
 normalization $\tilde{\pi}\colon\TSig
 \lrar C$ at $Q\in \TSig$.
 Since
 \begin{gather*}
 z = z_\mu = \frac{z_0}{x^\mu}=\frac{1}{yx^\mu},
 \end{gather*}
 we have thus obtained the normalization
 coordinate $z$ on the desingularized
 curve $\TSig$ near $Q$:
 \begin{gather*}
 \begin{cases}
 x = x(z) = c_Q\big(z^2\big) z^2,\\[1ex]
 y = y(z) = \dfrac{1}{zx^\mu}
 = c_Q\big(z^2\big)^{-\mu} z^{-2\mu-1}.
 \end{cases}
 \end{gather*}
 It gives a parametric
 equation for the singular spectral curve $\Sigma$:
 \begin{gather*}
y^2 = \frac{1}{c(x)x^{2\mu+1}}.
 \end{gather*}

As we have shown in the proof of Theorem~\ref{thm:geometric genus formula},
 the situation is the same for a zero of $q$ of~odd order.

For the purpose of local calculation
near $Q\in \TSig$, we use
the following local expressions:
\begin{gather}
\omega^{\tsig(z) - z}(z_1)=
\bigg(\frac{1}{z_1-\tsig(z)}-\frac{1}{z_1-z} + O(1)\bigg){\rm d}z_1,\nonumber
\\
\label{eta local}
\eta=y\,{\rm d}x = h(z)\,{\rm d}z := \frac{1}{z^{2\mu}}(1 + O(z))\,{\rm d}z.
\end{gather}
Here, we adjust the normalization coordinate
$z$ by a constant factor to make~\eqref{eta local}
simple.

Using the notation $\partial_z = \partial/\partial z$,
we have a local formula equivalent to~\eqref{differential TR} that is valid
for $2g-2+n\ge 2$:
\begin{gather}
\partial_{z_1} F_{g,n}(z_1,\dots,z_n)= -\!\sum_{j=2}^n\!
\bigg[\frac{\omega^{z_j-\tsig(z_j)}(z_1)}{2h(z_1)\,{\rm d}z_1}\,
\partial_{z_1}F_{g,n-1}(z_{[\hat{j}]})-\frac{\omega^{z_j-\tsig(z_j)}
(z_1)}{{\rm d}z_1\cdot 2h(z_j)}\,
\partial_{z_j}F_{g,n-1}\big(z_{[\hat{1}]}\big)\bigg]\nonumber
\\ \hphantom{\partial_{z_1} F_{g,n}(z_1,\dots,z_n)=}
{}-\frac{1}{2h(z_1)}\frac{\partial^2}{\partial u_1\partial u_2}
\bigg[F_{g-1,n+1}\big(u_1,u_2,z_{[\hat{1}]}\big)
\\ \hphantom{\partial_{z_1} F_{g,n}(z_1,\dots,z_n)=}\nonumber
{}+\sum_{\substack{g_1+g_2=g\\
I\sqcup J=[\hat{1}]}}^{\text{stable}}
F_{g_1,|I|+1}(u_1,z_I)F_{g_2,|J|+1}(u_2,z_J)\bigg]
\bigg|_{\substack{u_1=z_1\\u_2=z_1}}.
\label{Fgn local}
\end{gather}
Let us apply principal specialization. The~left-hand side becomes $\frac{1}{n}
\partial_z F_{g,n}(z,\dots,z)$.
To calculate the contributions from
the first line of the right-hand side of~\eqref{Fgn local},
we choose $j>1$ and set $z_i=z$ for all $i$
except for $i=1,j$. Then take the limit $z_j\rar z_1$.
In this procedure, we note that the contributions
from the simple pole
of $\omega^{z_j-\tsig(z_j)}(z_1)$
at $z_1 = \tsig(z_j)$ cancel at $z_1=z_j$.
Thus we~obtain
\begin{gather*}
-\sum_{j=2}^n \frac{1}{z_1-z_j}
\bigg(\frac{1}{2h(z_1)}\,\partial_{z_1}F_{g,n-1}
(z_1,z,\dots,z)
-\frac{1}{2h(z_j)}\,\partial_{z_j}F_{g,n-1}
(z_j,z,\dots,z)\bigg)\bigg|_{z_1=z_j}
\\ \qquad
{}=-\sum_{j=2}^n\partial_{z_1}
\bigg(\frac{1}{2h(z_1)}\,\partial_{z_1}F_{g,n-1}(z_1,z,\dots,z)\bigg)
\\ \qquad
{}=-(n-1)\partial_{z_1}
\bigg(\frac{1}{2h(z_1)}\,\partial_{z_1}F_{g,n-1}(z_1,z,\dots,z)\bigg)
\\ \qquad
{}=-(n-1)\partial_{z_1}
\bigg(\frac{1}{2h(z_1)}\bigg)\partial_{z_1}F_{g,n-1}(z_1,z,\dots,z)
-\frac{n-1}{2h(z_1)}\,\partial_{z_1}^2F_{g,n-1}(z_1,z,\dots,z).
\end{gather*}
The limit $z_1\rar z$ then produces
\begin{gather}
-\partial_z \frac{1}{2h(z)}\,
\partial_z F_{g,n-1}(z,\dots,z)
-\frac{1}{2h(z)}\,\partial_z^2 F_{g,n-1}(z,\dots,z)\nonumber
\\ \qquad
{}+\frac{(n-1)(n-2)}{2h(z)}\frac{\partial^2}
{\partial u_1\partial u_2}F_{g,n-1}(u_1,u_2,z,\dots,z)\bigg|_{u_1=u_2=z}.
\label{unstable}
\end{gather}
To calculate the principal specialization of
the second line of the right-hand side
of~\eqref{Fgn local},
we~note that since all points $z_i$'s for $i\ge 2$
are set to be equal,
a set partition by index sets $I$ and $J$
becomes a partition of $n-1$ with a combinatorial
factor that counts the redundancy. The~re\-sult~is
\begin{gather}
-\frac{1}{2h(z)}\frac{\partial^2}{\partial u_1\partial u_2}
F_{g-1,n+1}(u_1,u_2,z,\dots,z)\bigg|_{u_1=u_2=z}\nonumber
\\ \qquad
{}-\frac{1}{2h(z)}\sum_{\substack{g_1+g_2=g\\n_1+n_2=n-1}}^{\text{stable}}
\binom{n-1}{n_1}\partial_zF_{g_1,n_1+1}(z,\dots,z)\,
\partial_zF_{g_2,n_2+1}(z,\dots,z).
\label{stable}
\end{gather}
Assembling~\eqref{unstable} and~\eqref{stable}
together, we obtain
\begin{gather}
\frac{1}{2h(z)}\bigg[\partial_z^2 F_{g,n-1}(z,\dots,z)
+\sum_{\substack{g_1+g_2=g\\n_1+n_2=n-1}}^{\text{stable}}
\binom{n-1}{n_1}\partial_z F_{g_1,n_1+1}(z,\dots,z)\,
\partial_zF_{g_2,n_2+1}(z,\dots,z)\bigg]\nonumber
\\ \qquad
{}+\frac{1}{n}\partial_zF_{g,n}(z,\dots,z)
+\partial_z \frac{1}{2h(z)}\,\partial_z F_{g,n-1}(z,\dots,z)\nonumber
\\ \qquad
{}=\frac{(n-1)(n-2)}{2h(z)}\frac{\partial^2}
{\partial u_1\partial u_2}F_{g,n-1}(u_1,u_2,z,\dots,z)\bigg|_{u_1=u_2=z}\nonumber
\\ \qquad\hphantom{=}
{}-\frac{1}{2h(z)}\frac{\partial^2}{\partial u_1\partial u_2}
F_{g-1,n+1}(u_1,u_2,z,\dots,z)\bigg|_{u_1=u_2=z}.
\label{ps1}
\end{gather}

We now apply the operation
$\sum_{2g-2+n=m}\frac{1}{(n-1)!}$ to~\eqref{ps1} above, and write the result in
terms~of
\begin{gather*}
S_m(z) :=S_m(x(z))= \sum_{2g-2+n=m-1}\frac{1}{n!}F_{g,n}(z,\dots,z)
\end{gather*}
of~\eqref{Sm in Fgn}
to fit into the WKB formalism~\eqref{psi WKB}.
We observe that summing over all possibilities of~$(g,n)$ with the fixed value of $2g-2+n$,
the right-hand side of~\eqref{ps1} exactly cancels
out. Thus we have established
that the functions $S_m(z)$ of~\eqref{Sm in Fgn}
for $m\ge 2$ satisfy the recursion formula
\begin{gather}
\label{Sm recursion}
\frac{1}{2h(z)}\Bigg(\frac{{\rm d}^2S_m}{{\rm d}z^2}+
\sum_{\substack{a+b=m+1\\a,b\ge 2}}
\frac{{\rm d}S_a}{{\rm d}z} \frac{{\rm d}S_b}{{\rm d}z}\Bigg)
+\frac{{\rm d}S_{m+1}}{{\rm d}z}+\frac{{\rm d}}{{\rm d}z}
\bigg(\frac{1}{2h(z)}\bigg) \frac{{\rm d}S_m}{{\rm d}z}=0.
\end{gather}

Using~\eqref{eta local} we identify the derivation
\begin{gather}
\label{d/dx}
\frac{\rm d}{{\rm d}x} = \frac{y}{h(z)}\frac{{\rm d}}{{\rm d}z},
\end{gather}
which is the push-forward $\tpi_*({\rm d}/{\rm d}z)$
of the vector field ${\rm d}/{\rm d}z$. The~transformation~\eqref{d/dx} is
singular at $z=0$.
If we allow terms $a=0$ or $b=0$ in~\eqref{Sm recursion}, then what we have in
addition is
\begin{gather*}
\frac{1}{2h(z)}\cdot 2\,\frac{{\rm d}S_0}{{\rm d}z}
\frac{{\rm d}S_{m+1}}{{\rm d}z} = \frac{1}{h(z)}
\frac{h(z)}{y} \frac{{\rm d}S_0}{{\rm d}x}\frac{{\rm d}S_{m+1}}{{\rm d}z}
=\frac{{\rm d}S_{m+1}}{{\rm d}z},
\end{gather*}
since ${\rm d}S_0= y{\rm d}x$.
In other words, the $\frac{{\rm d}S_{m+1}}{{\rm d}z}$ term
already there in~\eqref{Sm recursion} is absorbed
in the split differentiation for $a=0$ and $b=0$.

From~\eqref{d/dx} and~\eqref{h 0}, we find that
the second derivative with respect
to~$x$ is given by
\begin{gather*}
\frac{{\rm d}^2}{{\rm d}x^2} = \frac{\rm d}{{\rm d}x}
\bigg(\frac{S_0'}{h(z)}\frac{{\rm d}}{{\rm d}z}\bigg)
=\frac{(S_0')^2}{h(z)^2}\frac{{\rm d}^2}{{\rm d}z^2}
+\frac{S_0'}{h(z)}\frac{{\rm d}}{{\rm d}z}\bigg(\frac{S_0'}{h(z)}
\bigg) \frac{{\rm d}}{{\rm d}z},
\end{gather*}
denoting by $S_0'={\rm d}S_0/{\rm d}x$. Then~\eqref{h m+1} yields
\begin{gather}
\label{hbar m+1}
\frac{(S_0')^2}{h(z)^2}
\bigg(\frac{{\rm d}^2}{{\rm d}z^2} S_m +\sum_{a+b=m+1} \frac{{\rm d}S_a}{{\rm d}z}\frac{{\rm d}S_b}{{\rm d}z}\bigg)
+\frac{S_0'}{h(z)}\frac{{\rm d}}{{\rm d}z}\bigg(\frac{S_0'}{h(z)}
\bigg) \frac{{\rm d}S_m}{{\rm d}z}=0.
\end{gather}
The coefficients of
${\rm d}S_m/{\rm d}z$ in~\eqref{hbar m+1} are
\begin{gather*}
2 \frac{(S_0')^2}{h(z)^2}\frac{{\rm d}S_1}{{\rm d}z}
+\frac{S_0'}{h(z)}\frac{{\rm d}}{{\rm d}z}\bigg(\frac{S_0'}{h(z)}\bigg)
=2 \frac{(S_0')^2}{h(z)^2}\frac{h(z)}{S_0'} S_1'
+\frac{\rm d}{{\rm d}x}\bigg(\frac{S_0'}{h(z)}\bigg)
\\ \hphantom{2 \frac{(S_0')^2}{h(z)^2}\frac{{\rm d}S_1}{{\rm d}z}
+\frac{S_0'}{h(z)}\frac{{\rm d}}{{\rm d}z}\bigg(\frac{S_0'}{h(z)}\bigg)}
{}=\frac{1}{h(z)}\big(2S_0'S_1'+S_0''\big)
+S_0'\frac{\rm d}{{\rm d}x}\bigg(\frac{1}{h(z)}\bigg)
=S_0'\frac{\rm d}{{\rm d}x}\bigg(\frac{1}{h(z)}\bigg)
\\ \hphantom{2 \frac{(S_0')^2}{h(z)^2}\frac{{\rm d}S_1}{{\rm d}z}
+\frac{S_0'}{h(z)}\frac{{\rm d}}{{\rm d}z}\bigg(\frac{S_0'}{h(z)}\bigg)}
{}=\frac{(S_0')^2}{h(z)^2}\cdot 2h(z)\frac{{\rm d}}{{\rm d}z}\bigg(\frac{1}{2h(z)}\bigg)
=\frac{2(S_0')^2}{h(z)}\frac{{\rm d}}{{\rm d}z}\bigg(\frac{1}{2h(z)}\bigg).
\end{gather*}
This is exactly what the last term of~\eqref{Sm recursion} has, after
adjusting overall multiplication by $\frac{2(S_0')^2}{h(z)}$.
Therefore, we have established that~\eqref{h 0},~\eqref{h 1}, and~\eqref{h 2}
make~\eqref{h m+1} equivalent to~\eqref{Sm recursion}.
This competed the proof of~Theorem~\ref{thm:WKB}.
\end{proof}

\begin{rem}In the examples of various
Hurwitz numbers considered in~\cite{BHLM, MS}, PDE recursions can be applied
to quantize the spectral curves (generalized Lambert curves)
and obtain second-order linear \emph{partial} differential equations.
Since spectral curves are analytic, the direct quantization
yields difference-differential equations in one variable.
These different quantization results are compatible
in the sense that the same
asymptotic solution~\eqref{psi WKB} satisfies both equations. A~geometric interpretation
is still missing for the
analysis of these two different quantization mechanisms.
\end{rem}

\begin{rem}
Geometry of normalization of the
singular spectral curve leads us to
the analysis of Stokes phenomena. It~is beyond the scope of current paper,
and will be treated elsewhere.
\end{rem}


\section{A simple classical example}
\label{sect:Airy}

Riemann and Poincar\'e
worked on
an interplay between algebraic geometry
of curves in a ruled surface and
the asymptotic expansion
 of an analytic solution
 to a differential equation defined on the
base curve of the ruled surface. The~theme of the current paper lies exactly on this
link, looking at this classical subject from a
modern point of view.
The
simple examples for ${\rm SL}(2,\bC)$-meromorphic
Higgs bundles
on $\bP^1$
 illustrate the relation between
a Higgs bundle, the compactified cotangent bundle
of a curve, a quantum curve, a classical differential
equation, non-Abelian Hodge correspondence,
and the quantum invariants that the quantum curve
captures.


\subsection{The Higgs bundle
for the Airy function}

The Higgs bundle $(E,\phi)$ we consider
consists of the base curve $C=\bP^1$
and a particular vector bundle
\begin{gather*}
E_0=K_{\bP^1}^\half\dsum K_{\bP^1}^{-\half}
=\cO_{\bP^1}(-1)\dsum \cO_{\bP^1}(1)
\end{gather*}
 of rank $2$ on ${\bP^1}$. A~meromorphic Higgs field is given by
\begin{gather}
\label{2x2}
\phi=\begin{bmatrix}
\phi_{11}&\phi_{12}\\
\phi_{21}&\phi_{22}
\end{bmatrix}\!\!
\colon \quad E\lrar E\tensor K_{\bP^1}(m), \qquad m\ge 0.
\end{gather}
Each matrix component is given by
\begin{gather*}
 \phi_{11}\colon\quad K_{\bP^1}^\half \lrar K_{\bP^1}^{\half}\tensor K_{\bP^1}(m)
=K_{\bP^1}^{\frac{3}{2}}(m),\qquad
\phi_{11}\in H^0(C,K_{\bP^1}(m)),
\\
\phi_{12}\colon\quad K_{\bP^1}^{-\half} \lrar K_{\bP^1}^{\half}\tensor K_{\bP^1}(m)
=K_{\bP^1}^{\frac{3}{2}}(m), \qquad
\phi_{12}\in H^0\big(C,K_{\bP^1}^{\tensor 2}(m)\big),
\\
\phi_{21}\colon\quad K_{\bP^1}^\half \lrar K_{\bP^1}^{-\half}\tensor K_{\bP^1}(m)
=K_{\bP^1}^\half(m),\qquad
\phi_{21}\in H^0(C,\cO_{\bP^1}(m)),
\\
\phi_{22}\colon\quad K_{\bP^1}^{-\half} \lrar K_{\bP^1}^{-\half}\tensor K_{\bP^1}(m)
=K_{\bP^1}^{\frac{1}{2}}(m),\qquad
\phi_{22}\in H^0(C,K_{\bP^1}(m)).
\end{gather*}
Since we are considering a point on a Hitchin section,
we take $\phi_{21}=1$ to be the identity map
$K_{\bP^1}^\half \overset{\isom}{\lrar} K_{\bP^1}^\half\hookrightarrow
K_{\bP^1}^\half(m)$,
and $\phi_{11}=\phi_{22}=0$. When we allow
singularities, we can make
other choices for $\phi_{21}$ as well.

The {Planck constant}
$\hbar$ has a geometric meaning~\eqref{Planck}
as a parameter of the extension classes
of line bundles. For~$\bP^1$, it is
\begin{gather*}
\hbar\in \Ext^1\Big(K_{\bP^1}^{-\half},K_{\bP^1}^{\half}\Big)
\isom H^1\big(\bP^1,K_{\bP^1}\big) = \bC.
\end{gather*}
It determines the unique extension
\begin{gather}
\label{extension}
0\lrar K_{\bP^1}^\half \lrar E_\hbar
\lrar K_{\bP^1}^{-\half} \lrar 0,
\end{gather}
where
\begin{gather}
\label{E_hbar}
E_\hbar \isom \begin{cases}
\cO_{\bP^1}(-1) \dsum \cO_{\bP^1}(1), &\hbar = 0,
\\
\cO_{\bP^1}\dsum \cO_{\bP^1}, &\hbar\ne 0
\end{cases}
\end{gather}
as a vector bundle, since
every vector bundle on $\bP^1$ splits.

The quantization of the Higgs field $\phi$ is
an $\hbar$-connection of Deligne in $E_\hbar$
defined on $\bP^1$ and is given by $\hbar \nabla^\hbar$, where
\begin{gather}
\label{hbar-connection intro}
\nabla^\hbar ={\rm d} +\frac{1}{\hbar}\phi\colon\ E_\hbar \lrar E_\hbar \tensor K_{\bP^1}(m),
\end{gather}
and ${\rm d}$ is the exterior differentiation operator
acting on sections of the trivial
bundle $E_\hbar$ for~$\hbar\ne 0$. The~operator
$\nabla^\hbar $ is a meromorphic
connection in the vector bundle $E_\hbar$.
Of course $d+\phi$ is never a connection
in~general, because $\phi$ is a Higgs field
belonging to a different bundle and satisfying a
different transition rule with respect to
coordinate changes. However,
as explained in~Section~\ref{sect:opers}
(see also~\cite{O-Paris, OM4}), a~Higgs field $\phi$
associated with a complex simple Lie group
on a Hitchin section gives rise to a
connection $\nabla^\hbar={\rm d}+\frac{1}{\hbar}\phi$ in $E_\hbar$ with respect to the
coordinate system associated with
a \emph{projective structure} subordinating
the complex structure of the base curve.
Since our examples are constructed on
$\bP^1$, the affine coordinate
$x\in \bA^1=\bP^1\setminus\{\infty\}$ is a
natural coordinate representing the
projective structure. Hence the
expression $d+\phi$ makes sense as a
connection
in $E_\hbar$ for every ${\rm SL}(2,\bC)$-Higgs bundle
$(E_0,\phi)$.
 This is due to
the vanishing of the Schwarzian derivatives
for the coordinate change in
a projective structure
 (see~\cite{O-Paris} for more detail on
 how the Schwarzian derivative plays a role here).

To see the effect of quantization, i.e., the passage
from the Higgs bundle to an $\hbar$-family
of connections~\eqref{hbar-connection intro},
 let us use the local coordinate
and write everything concretely. The~transition function
defined on $\bC^*=U_\infty \cap U_0$
 of the vector bundle
$E_0$
on $\bP^1=U_\infty \cup U_0$
is given by
$\left[\begin{smallmatrix}
x\\
&\frac{1}{x}
\end{smallmatrix}\right]$,
where $U_0=\bP^1\setminus \{\infty\} = \bA^1 $ and
$U_\infty= \bP^1\setminus \{0\}$. With respect
to the same coordinate,
the extension $E_\hbar$ is given by
$
\left[\begin{smallmatrix}
x&\hbar\\
&\frac{1}{x}
\end{smallmatrix}\right]
$.
The equality
\begin{gather*}
\begin{bmatrix}
1\\
-\frac{1}{\hbar x}&1
\end{bmatrix}
\begin{bmatrix}
x &\hbar\\
&\frac{1}{x}
\end{bmatrix}
\begin{bmatrix}
&-\hbar\\
\frac{1}{\hbar}&x
\end{bmatrix}
=
\begin{bmatrix}
1\\
&1
\end{bmatrix}
\end{gather*}
proves~\eqref{E_hbar}. The~local expressions of the $1$-form
$\phi_{11}$ and the quadratic differential
$\phi_{12}$ satisfy
\begin{gather*}
(\phi_{11})_u du = (\phi_{11})_x {\rm d}x,
\qquad
(\phi_{12})_u du^2 = (\phi_{12})_x {\rm d}x^2,
\end{gather*}
where $u=1/x$ is a coordinate on $U_\infty$.
Then the local
expressions of the Higgs field~\eqref{2x2}
with $\phi_{22}=0$
satisfy the following
transition relation with respect to~$E_0$:
\begin{gather*}
\label{Higgs transition}
\begin{bmatrix}
(\phi_{11})_u&-(\phi_{12})_u\\-1&
\end{bmatrix}{\rm d}u
=\begin{bmatrix}
x&\\&\frac{1}{x}
\end{bmatrix}
\begin{bmatrix}
(\phi_{11})_x&(\phi_{12})_x\\1&
\end{bmatrix}{\rm d}x
\begin{bmatrix}
x&\\&\frac{1}{x}
\end{bmatrix}^{-1}.
\end{gather*}
The negative signs are due to~${\rm d}u = -\frac{1}{x^2}{\rm d}x$.
For the case of ${\rm SL}(2,\bC)$-Higgs bundles,
we further assume
$\tr(\phi)=\phi_{11} = 0$.
In this case,
we note that~\eqref{Higgs transition}
is equivalent to the gauge transformation rule
of connection matrices with respect to~$E_\hbar$:
\begin{gather*}
-\frac{1}{\hbar}
\begin{bmatrix}
&(\phi_{12})_u\\1&
\end{bmatrix}{\rm d}u
=\frac{1}{\hbar}
\begin{bmatrix}
x&\hbar\\&\frac{1}{x}
\end{bmatrix}
\begin{bmatrix}
&(\phi_{12})_x\\1&
\end{bmatrix}{\rm d}x
\begin{bmatrix}
x&\hbar\\&\frac{1}{x}
\end{bmatrix}^{-1}-{\rm d}
\begin{bmatrix}
x&\hbar\\&\frac{1}{x}
\end{bmatrix}
\begin{bmatrix}
x&\hbar\\&\frac{1}{x}
\end{bmatrix}
^{-1}.
\end{gather*}
In other words,
\begin{gather*}
{\rm d}_u-\frac{1}{\hbar}
\begin{bmatrix}
&(\phi_{12})_u\\1&
\end{bmatrix}{\rm d}u
=\begin{bmatrix}
x&\hbar\\&\frac{1}{x}
\end{bmatrix}
\bigg({\rm d}_x+\frac{1}{\hbar}
\begin{bmatrix}
&(\phi_{12})_x\\1&
\end{bmatrix}{\rm d}x\bigg)
\begin{bmatrix}
x&\hbar\\&\frac{1}{x}
\end{bmatrix}^{-1}.
\end{gather*}
Therefore, $\nabla^\hbar = {\rm d} + \frac{1}{\hbar} \phi$ is a globally defined
connection in $E_\hbar$ for every $\hbar \ne 0$, and
\begin{gather*}
\big(\hbar\nabla^\hbar \big)\big|_{\hbar = 0}=\phi
\end{gather*}
is the original Higgs field.

Let us start with a particular
spectral curve, the algebraic curve
\begin{gather*}
\Sigma \subset \bF_2 = \bP\big(K_{\bP^1}\dsum \cO_{\bP^1}\big) = \overline{T^*\bP^1}
\end{gather*}
embedded in the Hirzebruch surface
with the defining equation
\begin{gather}
\label{Airy xy}
y^2-x = 0
\end{gather}
on $U_0$.
Here, $y$ is
a fiber coordinate of the cotangent bundle
$T^*\bP^1 \subset \bF^2$ over $U_0$. The~Hirzebruch surface is the
natural compactification of the cotangent bundle
$T^*\bP^1$, which is the
total space of~the canonical bundle $K_{\bP^1}$.
We denote by
$\eta \in H^0\big(T^*\bP^1, \pi^* K_{\bP^1}\big)$
the tautological $1$-form asso\-ci\-ated with the
projection $\pi\colon T^*\bP^1\lrar \bP^1$. It~is
expressed as $\eta = y{\rm d}x$ in terms of
the affine coordinates. The~holomorphic symplectic form
on $T^* \bP^1$ is given by $-{\rm d}\eta = {\rm d}x\wedge {\rm d}y$. The~$1$-form~$\eta$ extends to~$\bF_2$ as a meromorphic differential form and defines
 a divisor
\begin{gather*}
(\eta) = C_0-C_\infty,
\end{gather*}
where $C_0$ is the zero-section of $T^*\bP^1$, and
$C_\infty$ the section at infinity of
$\overline{T^*\bP^1}$. The~Picard group $\Pic(\bF_2)$ of the Hirzebruch
surface is generated by the class $C_0$ and a fiber
class $F$ of~$\pi$.
Although~\eqref{Airy xy} is a perfect
parabola in the affine plane, it has a quintic cusp
singularity at~$x=\infty$.
Let $(u,w)$ be a coordinate system
 on another affine
chart of $\bF_2$ defined by
\begin{gather*}
\begin{cases}
x = {1}/{u},\\
y\,{\rm d}x = v\, {\rm d}u, \qquad w = 1/v.
\end{cases}
\end{gather*}
Then $\Sigma$ in the $(u,w)$-plane is given by
\begin{gather}
\label{Airy uw}
w^2 = u^5.
\end{gather}
The expression of
$\Sigma\in \NS(\bF_2)$ as an element of
the N\'eron--Severy group of
$\bF_2$, in this case the same
as $\Pic(\bF_2)$, is thus given by
$\Sigma = 2C_0 + 5 F$.

Define a stable Higgs pair
$\left(E_0,\phi(q)\right)$ on $\bP^1$ with
$E_0=K_{\bP^1}^\half\dsum K_{\bP^1}^{-\half}
= \cO_{\bP^1}(-1)\dsum \cO_{\bP^1}(1)$
and
\begin{gather*}
\phi(q) = \begin{bmatrix}
& q\\
1
\end{bmatrix}\!\!\colon\
E_0\lrar E_0\tensor K_{\bP^1}(4).
\end{gather*}
Here, we choose a meromorphic quadratic differential
$q\in H^0(\bP^1,K_{\bP^1}(2)^{\tensor 2})$
that has a simple zero at $0\in \bP^1$ and
a pole of order $5$ at $\infty \in \bP^1$. Up to a
constant factor,
 there is only one such differential
\begin{gather*}
q = x({\rm d}x)^2 = \frac{1}{u^5}({\rm d}u)^2
\in H^0\big(\bP^1,K_{\bP^1}(2)^{\tensor 2}\big) = \bC.
\end{gather*}
The spectral curve $\Sigma$ of $(E_0,\phi(q))$ is given by
the characteristic equation
\begin{gather}
\label{char}
\det(\eta-\pi^*\phi) =\eta^2 -\pi^*\tr(\phi)+\pi^* \det(\phi)= \eta^2 - q=0
\end{gather}
in $\bF_2$. As explained above, $(E_0,\phi(q))$
uniquely determines a meromorphic \emph{oper}
\begin{gather}
\label{NAH-Airy}
\nabla^\hbar(q) =
{\rm d} +\frac{1}{\hbar} \phi(q) = {\rm d} +\frac{1}{\hbar}
\begin{bmatrix}
&q\\1
\end{bmatrix}
\end{gather}
on the extension $E_\hbar$ of~\eqref{extension}
over $\bP^1$~\cite{O-Paris}.
Indeed, the case $q=0$ of~\eqref{NAH-Airy}
for $\hbar = 1$
is the \emph{non-Abelian Hodge correspondence}
\cite{Donaldson, H1, Simpson}
\begin{gather*}
H^0(C,\End(E_0)\tensor K_{\bP^1})\owns
\begin{bmatrix}
0&0\\1&0
\end{bmatrix}
\Longleftrightarrow {\rm d}+
\begin{bmatrix}
0&0\\1&0
\end{bmatrix}\!
\colon\ E_1 \lrar E_1 \tensor K_{\bP^1}.
\end{gather*}
The \textbf{quantization} procedure
of this
paper is the following $\hbar$-deformation
\begin{gather*}
H^0(C,\End(E_0)\tensor K_{\bP^1}(m))\owns
\begin{bmatrix}
&q\\1&
\end{bmatrix}
\Longleftrightarrow {\rm d}+ \frac{1}{\hbar}
\begin{bmatrix}
&q\\1&
\end{bmatrix}\!
\colon\ E_\hbar \lrar E_\hbar \tensor K_{\bP^1}(m),
\end{gather*}
where $q\in H^0\left(C,K_{\bP^1}^{\tensor 2}
\tensor \cO_{\bP^1}(m)\right)$
is a meromorphic quadratic differential on $C$.

Let $\psi(x,\hbar)$ denote an analytic
function in $x$ with a formal parameter $\hbar$
such that
\begin{gather*}
\hbar \nabla_q^\hbar \begin{bmatrix}
-\hbar \psi'\\
\psi
\end{bmatrix} = 0,\qquad
\hbar\ne 0,
\end{gather*}
where $'$ denotes the
$x$ differentiation.
Then it satisfies a
Schr\"odinger equation
\begin{gather}
\label{Airy qc}
\bigg(\bigg(\hbar\frac{\rm d}{{\rm d}x}\bigg)^2-x\bigg)\psi(x,\hbar) = 0.
\end{gather}
The differential operator
\begin{gather*}
P(x,\hbar):=\bigg(\hbar\frac{\rm d}{{\rm d}x}\bigg)^2-x
\end{gather*}
quantizing the spectral
curve $\Sigma$ of~\eqref{Airy xy} is an example
of a
quantum curve.
Reflecting
the fact~\eqref{Airy uw} that $\Sigma$ has a
quintic cusp singularity at $x=\infty$,
\eqref{Airy qc} has an
\emph{irregular} singular point
of \emph{class} $\frac{3}{2}$ at $x=\infty$.

Let us recall the definition of regular and irregular singular points of
a second-order differential equation here.

\begin{Def}
Let
\begin{gather}\label{second}
\bigg(\frac{{\rm d}^2}{{\rm d}x^2}+a_1(x)\frac{\rm d}{{\rm d}x}+a_2(x)\bigg)\psi(x) = 0
\end{gather}
be a second-order differential equation
defined around a neighborhood of $x=0$ on a
small disc $|x|< \epsilon$ with meromorphic
coefficients $a_1(x)$ and $a_2(x)$ with poles
at $x=0$. Denote by $k$ (resp.~$\ell$) the
order of the pole of
$a_1(x)$ (resp.\ $a_2(x)$) at $x=0$.
If $k \le 1$ and $\ell\le 2$, then~\eqref{second}
has a~\textit{regular singular point} at $x=0$.
Otherwise, consider the \emph{Newton polygon}
 of the order of poles of the coefficients of~\eqref{second}. It~is the upper part of
 the convex hull of three
 points $(0,0)$, $(1, k)$, $(2,\ell)$. As a convention,
 if $a_j(x)$ is identically $0$, then
 we assign $-\infty$ as its pole order. Let $(1,r)$
 be the intersection point of
 the Newton polygon and the line $x=1$. Thus
 \begin{gather*}
r= \begin{cases}
 k, &2k\ge \ell,\\
\dfrac{\ell}{2}, &2k\le \ell.
 \end{cases}
 \end{gather*}
 The differential equation~\eqref{second}
 has an \textit{irregular singular point of class}
 $r-1$ at $x=0$ if $r>1$.
\end{Def}

The class $\frac{3}{2}$ at $\infty$
indicates how the asymptotic expansion
of the solution $\psi$ looks like.
Indeed, any non-trivial solution has
an essential singularity at $\infty$.
We note that every solution of~\eqref{Airy qc} is an entire function
for any value of $\hbar \ne 0$.
Applying our main result of this paper,
we construct a particular
\emph{all-order} asymptotic expansion
 of this entire solution
\begin{gather}
\label{Psi expansion}
\psi(x,\hbar) = \exp F(x,\hbar),\qquad
F(x,\hbar) := \sum_{m=0}^\infty \hbar^{m-1} S_m(x),
\end{gather}
valid for
$|\Arg(x)|<\pi$,
 and $\hbar >0$. Here, the first two terms of
 the asymptotic expansion are given by
\begin{gather}
\label{S0-Airy}
S_0(x) = \pm \frac{2}{3} x^{\frac{3}{2}},
\\
\label{S1-Airy}
S_1(x) = -\frac{1}{4} \log x.
\end{gather}
Although the \emph{classical limit}
$\hbar \rar 0$
of~\eqref{Airy qc} does not make sense
under the expansion~\eqref{Psi expansion},
the \emph{semi-classical limit} through the
\emph{WKB} analysis
\begin{gather}
\label{WKB-Airy}
\bigg[{\rm e}^{-S_1(x)}{\rm e}^{-\frac{1}{\hbar}S_0(x)}
\bigg(\hbar^2 \frac{{\rm d}^2}{{\rm d}x^2}-x\bigg)
{\rm e}^{\frac{1}{\hbar}S_0(x)}{\rm e}^{S_1(x)}\bigg]
\exp\bigg(\sum_{m=2}^\infty\hbar^{m-1} S_m(x)\bigg) = 0
\end{gather}
has a well-defined limit $\hbar \rar 0$. The~result
is $S_0'(x)^2 = x$,
which gives~\eqref{S0-Airy}, and also~\eqref{Airy xy}
by defining ${\rm d}S_0 = \eta$. The~vanishing of the
$\hbar$-linear terms of~\eqref{WKB-Airy}
is $2S_0'(x)S_1'(x) + S_0''(x) = 0$, which gives~\eqref{S1-Airy} above.

The entire solution in $x$ for $\hbar\ne 0$ and the choice of
$S_0(x)= - \frac{2}{3} x^{\frac{3}{2}}$
is called the \emph{Airy function}
\begin{gather}
\label{Airy}
\operatorname{Ai}(x,\hbar) =\frac{1}{2\pi} \hbar^{-\frac{1}{6}}
\int_{-\infty}^\infty
\exp\bigg({\frac{{\rm i}px}{\hbar^{2/3}}}+{{\rm i}\frac{p^3}{3}}\bigg){\rm d}p.
\end{gather}
The surprising discovery of Kontsevich~\cite{K1992} (cf.~\cite{DVV,W1991}) is that $S_m(x)$ for
$m\ge 2$ has the following \emph{closed} formula
\begin{gather}
\label{Airy Sm}
S_m(x) := \sum_{2g-2+n=m-1} \frac{1}{n!}F_{g,n}^{\text{Airy}}(x),
\\
\label{Airy principal}
F_{g,n}^{\text{Airy}}(x):= \frac{(-1)^n}{2^{2g-2+n}}
x^{-\frac{(6g-6+3n)}{2}}
\sum_{\substack{d_1+\dots+d_n\\=3g-3+n}}
\la \tau_{d_1}\cdots \tau_{d_n}\ra_{g,n}
\prod_{i=1}^n |(2d_i-1)|!! ,
\end{gather}
where the coefficients
\begin{gather*}
\la \tau_{d_1}\cdots \tau_{d_n}\ra_{g,n} =\int_{\Mbar_{g,n}}
\psi_1 ^{d_1}\cdots \psi_{n}^{d_n}
\end{gather*}
are the cotangent class intersection
numbers on the moduli space
$\Mbar_{g,n}$ of stable curves of
genus~$g$ with $n$ non-singular
marked points. The~expansion
coordinate $x^{\frac{3}{2}}$ of~\eqref{Airy principal}
indicates the class of the irregular singularity of
the Airy differential equation.

Although~\eqref{Airy Sm} is not
a generating function of all intersection numbers,
the quantum cur\-ve~\eqref{Airy xy} alone actually determines
every intersection number $\la \tau_{d_1}\cdots \tau_{d_n}\ra_{g,n}$.
This mechanism is topological recursion of~\cite{EO1}. PDE recursion
 computes \emph{free energies}
\begin{gather}
\label{Airy free energy}
F_{g,n}^{\text{Airy}}(t_1,\dots,t_n)
:= \frac{(-1)^n}{2^{2g-2+n}}
\sum_{\substack{d_1+\dots+d_n\\=3g-3+n}}
\la \tau_{d_1}\cdots \tau_{d_n}\ra_{g,n}
\prod_{i=1}^n \left(\frac{t_i}{2}\right)^{2d_i+1} |(2d_i-1)|!!
\end{gather}
as a function in $n$ variables
from $\Sigma$ through the process of
blow-ups of $\bF_2$, and the exterior derivative of free energies
are the symplectic invariants of~\cite{EO1}.

\subsection{Blowing up a Hirzebruch surface}

Let us now give a detailed algebraic geometry
procedure for this
example.
We start with the spectral curve $\Sigma$
of
\eqref{Airy xy}. Our goal is to
come up with~\eqref{Airy qc}. The~first step is to blow up $\bF_2$ and
to construct a normalization of $\Sigma$. The~construction of $\operatorname{Bl}(\overline{T^*C})$ is
given in Definition~\ref{def:Bl}. It~is the minimal
resolution of the divisor
 \begin{gather*}
 \Sigma = \left(\det(\eta - \pi^*\phi)\right)_0
 \end{gather*}
 of the characteristic polynomial. The~discriminant of the
defining equation~\eqref{char}
of the spectral curve is
\begin{gather*}
-\det(\phi) = x({\rm d}x)^2 = \frac{1}{u^5}({\rm d}u)^2.
\end{gather*}
It has a simple zero at $x=0$ and a pole of order
$5$ at $x=\infty$. The~\emph{geometric genus formula}~\eqref{pg} for the general base curve
$C$ reads
\begin{gather*}
g(\widetilde{\Sigma}) = 2g(C)-1+\half \delta,
\end{gather*}
where $\delta$ is the sum of the number of cusp singularities
of $\Sigma$ and the ramification points
of~$\pi$: $\Sigma\lrar C$ (Theorem~\ref{thm:geometric genus formula}).
In our case, it tells us that
$\widetilde{\Sigma}$ is a non-singular curve
of genus $0$, i.e., a~$\bP^1$,
after blowing up $\lfloor \frac{5}{2}\rfloor = 2$ times.

The center of blow-up is $(u,w)=(0,0)$ for
the first time. Put $w=w_1 u$,
and denote by~$E_1$ the exceptional
divisor of the first blow-up. The~proper transform
of $\Sigma$ for this blow-up,
$w_1^2 = u^3$, has a cubic
cusp singularity, so we blow up again at
the singular point. Let $w_1=w_2 u$,
and denote by $E_2$ the exceptional
divisor created by the second blow-up. The~self-intersection of~the proper transform
of $E_1$ is $-2$.
We then obtain the
desingularized curve $\TSig$,
locally given by~$w_2 ^2 = u$. The~proof of
Theorem~\ref{thm:geometric genus formula}
also tells us that $\TSig\lrar \bP^1$
is ramified at
two points. Choose the affine coordinate
 $t=2 w_2$
 of the exceptional divisor added at the second
 blow-up. Our choice of the constant factor
 is to make the formula the same as in~\cite{DMSS}.
We have
\begin{gather}
\label{Airy xy in t}
\begin{cases}
x = \dfrac{1}{u} = \dfrac{1}{w_2^2}
 = \dfrac{4}{t^2},\\
y = -\dfrac{u^2}{w} = - \dfrac{u^2}{w_2u^2}
= -\dfrac{2}{t}.
\end{cases}
\end{gather}
In the $(u,w)$-coordinate, we see that the parameter
$t$ is a normalization parameter of the
quintic cusp singularity:
\begin{gather*}
\begin{cases}
u = \dfrac{t^2}{4},
\\[1.5ex]
w = \dfrac{t^5}{32}.
\end{cases}
\end{gather*}
Note that $\TSig$ intersects transversally
with the proper transform of $C_\infty$. The~blow-up space
$\operatorname{Bl}(\bF^2)$ is the result of the twice blow-ups of the Hirzebruch surface:
\begin{gather}
 \label{blow-up P1}
\begin{split}
&\xymatrix{
\widetilde{\Sigma}
\ar[dd]_{\tilde{\pi}} \ar[rr]^{\tilde{i}}\ar[dr]^{\nu}&&\operatorname{Bl}(\overline{T^*\bP^1})
\ar[dr]^{\nu}
\\
&\Sigma \ar[dl]_{\pi}\ar[rr]^{i} &&
\overline{T^*\bP^1}=\bF_2. \ar[dlll]^{\pi}
\\
\bP^1 }
\end{split}
\end{gather}

\begin{figure}[hbt]
\centering
\includegraphics[scale=0.95]{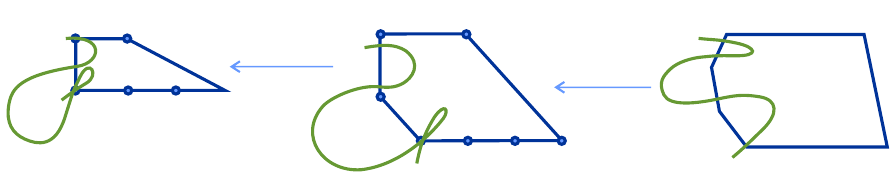}
\put(-382,54){\makebox(-0,0)[lb]{\footnotesize $0$}}
\put(-395,30){\makebox(-0,0)[lb]{\small $P$}}
\put(-395,5){\makebox(-0,0)[lb]{$\Sigma$}}
\put(-367,67){\makebox(-0,0)[lb]{\small $C_0$}}
\put(-364,53){\makebox(-0,0)[lb]{\footnotesize $-2$}}
\put(-361,28){\makebox(-0,0)[lb]{\footnotesize $+2$}}
\put(-335,27){\makebox(-0,0)[lb]{\small $C_\infty$}}
\put(-295,54){\makebox(-0,0)[lb]{\footnotesize ${\rm Bl}_p \overline{T^*C}$}}
\put(-246,48){\makebox(-0,0)[lb]{\small $F$}}
\put(-224,26){\makebox(-0,0)[lb]{\footnotesize $-1$}}
\put(-234,48){\makebox(-0,0)[lb]{\footnotesize $-1$}}
\put(-234,15){\makebox(-0,0)[lb]{\small $E$}}
\put(-215,6){\makebox(-0,0)[lb]{\small $Q$}}
\put(-201,6){\makebox(-0,0)[lb]{\footnotesize $+1$}}
\put(-179,5){\makebox(-0,0)[lb]{\small $C_\infty$}}
\put(-217,69){\makebox(-0,0)[lb]{\small $C_0$}}
\put(-74,26){\makebox(-0,0)[lb]{\footnotesize $-1$}}
\put(-81,48){\makebox(-0,0)[lb]{\footnotesize $-1$}}
\put(-80,40){\makebox(-0,0)[lb]{\footnotesize $-2$}}
\put(-90,20){\makebox(-0,0)[lb]{\small $E_2$}}
\put(-100,40){\makebox(-0,0)[lb]{\small $E_1$}}
\put(-100,58){\makebox(-0,0)[lb]{\small $F$}}
\put(-52,69){\makebox(-0,0)[lb]{\small $C_0$}}
\put(-46,16){\makebox(-0,0)[lb]{\small $o$}}
\put(-48,3){\makebox(-0,0)[lb]{\small $C_\infty$}}
\put(-68,5){\makebox(-0,0)[lb]{\small $R$}}
\put(-164,45){\makebox(-0,0)[lb]{\footnotesize ${\rm Bl}_Q({\rm Bl}_p \overline{T^*C})$}}
\caption{}
\end{figure}

Topological recursion~\eqref{integral TR}
 requires
a globally defined meromorphic $1$-form $W_{0,1}$
on $\widetilde{\Sigma}$ and
a~symmetric meromorphic $2$-form $W_{0,2}$
on the product $\widetilde{\Sigma}
\times \widetilde{\Sigma}$ as the initial data.
We choose
\begin{gather}
\begin{cases}
\label{W0102}
W_{0,1} = \tilde{i}^*\nu^*\eta,\\
W_{0,2} = {\rm d}_1{\rm d}_2 \log E_{\widetilde{\Sigma}},
\end{cases}
\end{gather}
where
$E_{\widetilde{\Sigma}}$ is a normalized
Riemann prime form on
${\widetilde{\Sigma}}$
(see \cite[Section~2]{OM1}). The~form $W_{0,2}$
depends only on the intrinsic geometry of
the smooth curve $\widetilde{\Sigma}$. The~geometry of~\eqref{blow-up P1}
is encoded in~$W_{0,1}$.

Now we apply PDE
recursion~\eqref{differential TR}
to the geometric data~\eqref{blow-up P1} and~\eqref{W0102}.
We claim that topological recursion of~\cite{EO1} for the geometric data we
are considering
now is exactly the same as the recursive
equation of \cite[equation~(6.12)]{DMSS} applied to
the curve~\eqref{Airy xy in t}
\emph{realized as a plane parabola}
in~$\bC^2$.
This is because topological
recursion~\eqref{integral TR} has two
residue contributions, one each from $t=0$ and
$t=\infty$. As proved in \cite[Section~6]{DMSS},
the integrand on the right-hand side
of the recursion formula~\cite[equation~(6.12)]{DMSS} does not have any
pole at $t=0$. Therefore, the residue contribution
from this point is $0$.
PDE recursion is obtained by
deforming the contour of integration to enclose
only poles of the differential forms $W_{g,n}$.
Since $t=0$ is a regular point, the two methods
have no difference.

The $W_{0,2}$ of~\eqref{W0102}
is simply $\frac{{\rm d}t_1\cdot {\rm d}t_2}{(t_1-t_2)^2}$
because $\widetilde{\Sigma}\isom \bP^1$.
Since $t$ of~\eqref{Airy xy in t}
is a normalization coordinate,
we have
\begin{gather*}
W_{0,1} = \tilde{i}^*\nu^*(\eta) = y(t)\,{\rm d}x(t)
= \frac{16}{t^4},
\end{gather*}
in agreement of \cite[equation~(6.8)]{DMSS}.
Noticing that the solution to
topological recursion
is unique from the
initial data, we conclude that
\begin{gather*}
{\rm d}_1\cdots {\rm d}_n F_{g,n}^\Airy \big(x(t_1),
\dots,x(t_n)\big) = W_{g,n}.
\end{gather*}
By setting the constants of integration by
integrating from $t=0$ for
PDE recursion, we obtain
the expression~\eqref{Airy free energy}.
Then its principal specialization
gives~\eqref{Airy principal}. The~equivalence of PDE
 recursion and the quantum curve
equation Theorem~\ref{thm:WKB} then proves
\eqref{Airy qc} with the expression of
\eqref{Psi expansion} and~\eqref{Airy Sm}.

In this process, what is truly amazing is that
the single differential equation
\eqref{Airy qc}, which is our
quantum curve, knows everything
about the free energies~\eqref{Airy free energy}.
This is because we can recover
the spectral curve $\Sigma$ from the quantum
curve. Then the procedures we need to apply,
the blow-ups and PDE recursion,
are canonical. Therefore, we actually recover
\eqref{Airy free energy} as explained above.

It is surprising to see that a simple entire function
\eqref{Airy}
contains so much geometric information.
Our expansion~\eqref{Psi expansion} is
an expression of this entire function viewed
from its essential singularity. We can extract
 rich information of the solution by
restricting the region where the asymptotic
expansion is valid.
If we consider~\eqref{Psi expansion} only as a
formal expression in $x$ and $\hbar$, then
we cannot see how the coefficients are
related to quantum invariants.
Topological recursion~\cite{EO1}
is a key to connect the two worlds:
the world of quantum invariants, and the world
of holomorphic functions and differentials.
This relation is also knows as a
\textit{mirror symmetry}, or in analysis,
 simply as the \emph{Laplace transform}. The~intersection numbers
$\la \tau_{d_1}\cdots \tau_{d_n}\ra_{g,n}$
belong to
the $A$-model, while the spectral
curve $\Sigma$ of~\eqref{Airy xy}
and free energies
belong to the $B$-model.
We consider~\eqref{Airy free energy} as
an example of
the Laplace transform, playing the role of
mirror symmetry~\cite{OM3, DMSS}.
In the context of Hitchin theory, mirror symmetry also plays a different role
through Langland duality (cf.~\cite{HM,KW,W2008-2,W2010}). It is unclear to us how these
two different mirror symmetries are interrelated.

\subsection*{Acknowledgements}
The authors wish to thank Philip Boalch
for many useful discussions and comments
on their work on quantum curves. In particular,
his question proposed at the American
Institute of~Mathematics Workshop,
\emph{Spectral data for Higgs bundles} in
September--October 2015, was critical for the
development of the theory presented in this paper.

 This joint research is
 carried out while the authors
 have been staying in the following ins\-ti\-tutions
 in the last several years:
the American Institute of Mathematics in
California,
the Banff International Research Station,
Institutul de
 Matematic\u{a} ``Simion Stoilow''
 al Academiei Rom\newtie{a}ne,
 Institut Henri Poincar\'e,
 Institute for Mathematical
Sciences at the National University of Singapore,
 Kobe University,
 Leibniz Universit\"at Hannover,
 Lorentz Center Leiden,
 Mathematisches Forschungsinstitut Oberwolfach,
 Max-Planck-Institut f\"ur Mathematik-Bonn,
 and
 Osaka City University Advanced
 Mathematical Institute.
 Their generous financial
 support, hospitality, and stimulating research
 environments are greatly appreciated.

The authors also thank
J\o rgen Andersen,
Vincent Bouchard,
Tom Bridgeland,
Bertrand Eynard,
Edward Frenkel,
Tam\'as Hausel,
Kohei Iwaki,
Maxim Kontsevich,
Laura Schaposnik,
Carlos Simpson,
Albert Schwarz,
Yan Soibelman,
Ruifang Song,
J\"org Teschner,
and Richard Wentworth,
for useful comments, suggestions, and discussions.
During the preparation of
this work, the research of O.D.\ was supported by
 GRK 1463 \emph{Analysis,
Geometry, and String Theory} at the
Leibniz Universit\"at
 Hannover, and a grant from
 MPIM-Bonn. The~research of M.M.\ was supported
by IH\'ES, MPIM-Bonn,
NSF grants DMS-1104734, DMS-1309298,
DMS-1619760, DMS-1642515,
and NSF-RNMS: Geometric Structures And
Representation Varieties (GEAR Network,
DMS-1107452, 1107263, 1107367).

\LastPageEnding

\end{document}